\documentclass[reqno,12pt]{amsart}

\usepackage{amsfonts,amsmath,amsthm,amssymb,amscd}
\usepackage[all]{xy}
\usepackage{longtable}

\usepackage{color}

\usepackage[mathscr]{eucal}
\usepackage{mathrsfs}

\newcommand{\version}{Ver.~0.0}
\newcommand{\setversion}[1]{\renewcommand{\version}{Ver.~{#1}}}
\setversion{0.0 [2011/06/25 18:57]}
\setversion{0.5 [2011/11/13 14:42:55]}
\setversion{0.6 [2012/02/23 15:58]}
\setversion{0.7 [2012/04/04]}
\setversion{0.8 [2012/04/20 14:47]}
\setversion{0.81 [2012/04/20 18:10]}
\setversion{0.9 [2012/06/10 03:29:27]}
\setversion{0.91 [2012/06/29 18:09:46]}
\setversion{1.0 [2013/02/25]}
\setversion{1.1 [2013/02/26 14:15:56]}
\setversion{1.2 [2013/03/15 17:10:36]}
\setversion{1.3 [2013/04/03 01:04:42]}

\title [On orbits in double flag varieties for symmetric pairs]
{On orbits in double flag varieties for symmetric pairs}

\dedicatory{Dedicated to Jiro Sekiguchi on the occasion of his sixtieth birthday.}

\author{Xuhua He}
\address{
Department of Mathematics\\
Hong-Kong University of Science and Technology\\
Clear Water Bay, Kowloon\\ Hong-Kong}
\email{maxhhe@ust.hk}
\thanks{X. H. is partially supported by HKRGC grant 602011.}

\author{Kyo Nishiyama}
\address{
Department of Physics and Mathematics\\
Aoyama Gakuin University\\
Fuchinobe 5-10-1, Sagamihara 229-8558, Japan}
\email{kyo@gem.aoyama.ac.jp}

\thanks{K. N.  is supported by JSPS Grant-in-Aid for Scientific Research (B) \#{21340006}.}

\author{Hiroyuki Ochiai}
\address{
Faculty of Mathematics\\
Kyushu University\\
744, Motooka, Nishi-ku, Fukuoka 819-0395, Japan}
\email{ochiai@imi.kyushu-u.ac.jp}

\thanks{H. O. is supported by JSPS Grant-in-Aid for Scientific Research (A) \#{19204011}, and JST CREST}

\author{Yoshiki Oshima}
\address{
Kavli IPMU (WPI)\\
The University of Tokyo\\
5-1-5 Kashiwanoha, Kashiwa, Chiba 277-8583, Japan}
\email{yoshiki.oshima@ipmu.jp}
\thanks{Y. O. was supported by Grant-in-Aid for JSPS Fellows (10J00710).}


\theoremstyle{plain}
\newtheorem{theorem}{Theorem}
\newtheorem{proposition}[theorem]{Proposition}
\newtheorem{corollary}[theorem]{Corollary}
\newtheorem{lemma}[theorem]{Lemma}
\newtheorem{assumption}[theorem]{Assumption}

\newtheorem{itheorem}{Theorem}

\theoremstyle{definition}

\theoremstyle{remark}
\newtheorem{remark}[theorem]{\upshape Remark}

\numberwithin{equation}{section}
\numberwithin{theorem}{section}

\newcommand{\bbJ}{\mathbb{J}}
\newcommand{\bbK}{\mathbb{K}}
\newcommand{\bbL}{\mathbb{L}}

\newcommand{\bbG}{\mathbb{G}}
\newcommand{\C}{\mathbb{C}}
\newcommand{\bbP}{\mathbb{P}}
\newcommand{\bbQ}{\mathbb{Q}}
\newcommand{\bbT}{\mathbb{T}}
\newcommand{\bbU}{\mathbb{U}}

\newcommand{\bbv}{\mathbf{v}}

\newcommand{\bbw}{\mathbf{w}}

\newcommand{\lie}[1]{\mathfrak{#1}}

\newcommand{\Lie}{\mathop\mathrm{Lie}\nolimits{}}

\newcounter{thmenum}
\newenvironment{thmenumerate}{%
\begin{list}{$(\thethmenum)$}{%
\usecounter{thmenum}
\setlength{\labelsep}{.5em}
\setlength{\labelwidth}{-7pt}
\setlength{\topsep}{0pt}
\setlength{\partopsep}{0pt}
\setlength{\parsep}{0pt}
\setlength{\leftmargin}{3pt}
\setlength{\rightmargin}{0pt}
\setlength{\itemindent}{\leftmargin}
\setlength{\itemsep}{0pt}
}}
{\end{list}}

\newcommand{\mycomment}[1]{} 

\newlength{\lengthcup}
\newcommand{\diag}{\qopname\relax o{diag}}
\newcommand{\rank}{\qopname\relax o{rank}}

\newcommand{\Aut}{\qopname\relax o{Aut}}
\newcommand{\Ad}{\qopname\relax o{Ad}}
\newcommand{\ad}{\mathop{\mathrm{ad}}\nolimits{}}

\newcommand{\calorbit}{\mathcal{O}}

\newcommand{\GFl}{\mathfrak{X}}

\newcommand{\KFl}{\mathcal{Z}}
\newcommand{\GKFl}[2]{\GFl_{#1}\times\KFl_{#2}}

\newcommand{\bsl}{\backslash}

\newcommand{\psg}{parabolic subgroup }
\newcommand{\psgs}{parabolic subgroups }

\newcommand{\jwjprime}{{}^{J}W^{J'}}
\newcommand{\jwj}[2]{{}^{#1}W^{#2}}
\newcommand{\jbbwj}[2]{{}^{#1}\mathbb{W}^{#2}}

\newcommand{\wP}[1]{P^{#1}}
\newcommand{\wPprime}[1]{\wP{#1}\cap P'}
\newcommand{\wPcap}[2]{\wP{#2}\cap{#1}}

\newcommand{\kgbw}[1]{\mathscr{V}({#1})}

\newcommand{\opposite}[1]{{#1}^{\star}}

\newcommand{\Ptilde}[1]{\tilde{P}_{#1}{}}

\newcommand{\Pmin}{P_{\mathrm{min}}}

\newcommand{\openBL}{B_{L'}^{\star}}

\newcommand{\Adj}{\Ad}
\newcommand{\adj}{\ad}
\newcommand{\bbC}{\C}

\begin{document}

\begin{abstract}
Let $ G $ be a connected, simply connected semisimple algebraic group over the complex number field, and let $ K $ be the fixed point subgroup of an involutive automorphism of $ G $ 
so that $ (G, K) $ is a symmetric pair.   

We take parabolic subgroups $ P $ of $ G $ and $ Q $ of $ K $ respectively and  
consider the product of partial flag varieties $ G/P $ and 
$ K/Q $ with diagonal $ K $-action, which we call 
a \emph{double flag variety for symmetric pair}.  
It is said to be \emph{of finite type} if there are only finitely many $ K $-orbits on it.  

In this paper, we give a parametrization of $ K $-orbits on $ G/P \times K/Q $ in terms of quotient spaces of unipotent groups without assuming the finiteness of orbits.  
If one of $ P \subset G $ or $ Q \subset K $ is a Borel subgroup, the finiteness of orbits is closely related to 
spherical actions.  
In such cases, we give a complete classification of double flag varieties of finite type, namely, we obtain classifications of $ K $-spherical flag varieties $ G/P $ and 
$ G $-spherical homogeneous spaces $ G/Q $.
\end{abstract}

\maketitle



\section*{Introduction}

Let $ G $ be a connected, simply connected semisimple algebraic group over the complex number field $ \bbC $.

Let $ P_i \; (i = 1, 2, \dots, k) $ be parabolic subgroups of $ G $ 
and consider partial flag varieties $ \GFl_{P_i} := G/P_i $ of $ G $.   
We are interested in the product of flag varieties 
$ \GFl_{P_1} \times \GFl_{P_2} \times \cdots \times \GFl_{P_k}, $ on 
which $ G $ acts diagonally.  
We say a multiple flag varieties $ \GFl_{P_1} \times \GFl_{P_2} \times \cdots \times \GFl_{P_k}$ is of {\it finite type} if it admits only finitely many $G$-orbits.
It is an interesting problem to classify multiple flag varieties of finite type.  
According to a result of 
Magyar-Weyman-Zelevinsky (\cite{MWZ.1999, MWZ.2000}), 
$ k $ must be less than or equal to $ 3 $ 
if a multiple flag variety is of finite type and if $ G $ is of classical type.

Let us consider a triple flag variety 
$ \GFl_{P_1} \times \GFl_{P_2} \times \GFl_{P_3} $. 
If $ P_3 = B $ is a Borel subgroup, 
then the triple flag variety is of finite type 
if and only if 
$ \GFl_{P_1} \times \GFl_{P_2} $ 
is a spherical $G$-variety.  
For maximal parabolic subgroups $ P_1 $ and $ P_2 $, 
Littelmann classified such spherical double flag varieties 
(\cite{Littelmann.1994}, see also \cite{Panyushev.1993}).  
For general parabolic subgroups $ P_1 $ and $ P_2 $, 
Stembridge \cite{Stembridge.2003} 
classified them completely.
In \cite{MWZ.1999, MWZ.2000}, 
they classified the triple flag varieties 
$ \GFl_{P_1} \times \GFl_{P_2} \times \GFl_{P_3} $ 
of finite type for $G=SL_n$ and $G=Sp_n$.
They also gave a complete enumeration of the $ G $-orbits on such triple
 flag varieties. 

Let $ K $ be the subgroup of fixed points of a non-trivial involution $ \theta $ of $ G $.
Take parabolic subgroups $ P $ of $ G $ and $ Q $ of $ K $.  
Then we call $ \GKFl{P}{Q} (= G/P \times K/Q) $ a {\it double flag variety for a symmetric pair} $ (G, K) $, 
where $ \KFl_Q := K/Q $ (see \cite{NO.2011}).  
The subgroup $ K $ naturally acts on $ \GKFl{P}{Q} $ diagonally.   
This notion is a generalization of the triple flag varieties 
(see \S~\ref{subsection:triple.flag.variety}), 
and we say it is of \emph{finite type} if there exist only finitely many $ K $-orbits.

In some cases, finiteness of $ G $-orbits on a triple flag variety 
implies finiteness of $ K $-orbits on a double flag variety for $ (G, K) $.  
In \cite{NO.2011}, we investigated such situations and got two sufficient conditions 
for $ \GKFl{P}{Q} $ to be of finite type: 

\begin{itheorem}[{\cite[Theorem~3.1]{NO.2011}}]
\label{NO.theorem:triple.flag.to.double.flag}
Let $ P' $ be a $ \theta $-stable parabolic subgroup of $ G $ such that $ P' \cap K = Q $.  
If the number of $ G $-orbits on $ \GFl_{P} \times \GFl_{\theta(P)} \times \GFl_{P'} $ is finite, 
then there are only finitely many $ K $-orbits on the double flag variety $ \GKFl{P}{Q} $.
\end{itheorem}

\begin{itheorem}[{\cite[Theorem~3.4]{NO.2011}}]
\label{NO.theorem:intersection.of.parabolics}
Let $ P_i \; (i = 1, 2, 3) $ be a parabolic subgroup of $ G $.  
Suppose that $ \GFl_{P_1} \times \GFl_{P_2} \times \GFl_{P_3} $
has finitely many $G$-orbits
and that $Q:= P_1 \cap P_2$ is a parabolic subgroup of $K$.
Then $\GKFl{P_3}{Q}$ has finitely many $K$-orbits.  

Moreover, if $ P_3 $ is a Borel subgroup $ B $ and the product $ P_1 P_2 $ is open in $ G $, 
then the converse is also true, i.e., 
the double flag variety 
$ \GKFl{B}{Q} $ 
is of finite type 
if and only if 
the triple flag variety 
$ \GFl_{P_1} \times \GFl_{P_2} \times \GFl_{B} $ 
is of finite type.
\end{itheorem}

Using these two theorems, 
we can produce many examples of double flag varieties of finite type.  
However, a complete classification is not known yet.  

In this paper, we study $ K $-orbit structure on an arbitrary double flag variety $ \GKFl{P}{Q} $ 
which is not necessarily of finite type and, 
as a result, we get some criteria for the finiteness of orbits.
Our method relies on the Bruhat decomposition and 
``KGB-decomposition'' (i.e., the $ K $-orbit decomposition of flag varieties; 
see \S~\ref{subsection:reduction.to.KGB}).  
The set of $ K $-orbits in $ \GKFl{P}{Q} $ is decomposed into 
a finite disjoint union of some quotient spaces   
parametrized by elements of Weyl groups (or ``Bruhat parameters'') and 
``KGB-parameters''.  
For each of these parameters, 
we construct a certain double coset space of unipotent subgroups related to $ P $ and $ Q $, 
which admits an action of a subgroup of Levi component of $Q$.  
The quotient spaces are obtained from this action.

Though general description of the orbit space structure of 
$ K \bsl (\GKFl{P}{Q}) $ 
is much complicated, 
it becomes considerably simpler if $ Q $ is a Borel subgroup of $ K $.  
So let us give a parametrization of orbits in this special case here, and 
for general case we refer to Theorem~\ref{main.thm:classify.K.orbits.on.double.flag.variety}.  

Let $ B $ be a $ \theta $-stable Borel subgroup of $ G $ which contains 
a $ \theta $-stable maximal torus $ T $, 
and $ W $ the Weyl group of $ G $.  
We denote by $ U_B $ the unipotent radical of $ B $ so that $ B = T U_B $.
We write 
$B_K = B \cap K = T_K U_{B_K} $, a Borel subgroup of $K$.

\begin{itheorem}
\label{itheorem:parametrization.of.K.orbits.when.Q.is.Borel}%
Let $P$ be a standard parabolic subgroup of $G$ containing $ B $, 
and $ W_P $ the subgroup of $ W $
 corresponding to the Levi component of $ P $.  
Then the $K$-orbits on the double flag variety are parametrized as follows:
\begin{equation*}
 K \bsl (\GKFl{P}{B_K}) \simeq 
\coprod_{w \in W_P\bsl W} \Bigl( (w^{-1} P w \cap U_B) \bsl U_B / U_{B_K} \Bigr) \Bigm/T_K ,
\end{equation*}
where the maximal torus $ T_K $ of $ K $ acts on the double coset space via conjugation.
\end{itheorem}

Let us discuss finiteness of $ K $-orbits on the double flag varieties.  
If $ P $ is a Borel subgroup of $G$ or $ Q $ is a Borel subgroup of $ K $, 
 then we will see that the finiteness of the double flag variety
 is reduced to the sphericity of a certain linear action.
The classification of spherical linear actions were established by
 Kac~\cite{Kac.1980} for irreducible case and independently by Benson-Ratcliff~\cite{Benson.Ratcliff.1996} and Leahy~\cite{Leahy.1998} for reducible case. 
We therefore obtain a classification of double flag varieties of finite type
 in such cases.

If $ Q = B_K $, we can apply Panyushev's theorem to obtain 
that the conormal bundle 
$ T^{\ast}_{\calorbit} \GFl_P $ (or the normal bundle 
$ T_{\calorbit} \GFl_P $) is $ K $-spherical for any $ K $-orbit $ \calorbit $
 in $\GFl_P$ if and only if 
the flag variety $ \GFl_P $ is $ K $-spherical (see \cite{Panyushev.1999.MMath}).  
As a consequence,
 taking $\calorbit$ as the closed orbit through the base point $e\cdot P$,
 we see that $ \GKFl{P}{Q} $ is of finite type 
 if and only if the action of Levi component of $P\cap K$ on 
the fiber of $ T_{\calorbit} \GFl_P $ at $e\cdot P$ is spherical.
We refer to Theorem~\ref{thm:finiteness.iff.MFaction}
 for the details.  
Using the tables of multiplicity free actions by 
Benson-Ratcliff \cite{Benson.Ratcliff.1996}, 
we obtain a complete classification of 
the double flag varieties 
$ \GKFl{P}{B_K} $ of finite type in Theorem~\ref{theorem:k.spherical}:  

\begin{itheorem}
\label{itheorem:k.spherical}
Let $G$ be a connected simple algebraic group  
 and $(G,K)$ a symmetric pair.
Let $P$ be a parabolic subgroup of $G$.  
Then the double flag variety $ G/P \times K/B_K $
 is of finite type 
 if and only if 
it appears in Table~\ref{table:k.spherical}.  
The table also serves as a complete list of $ K $-spherical partial flag varieties $ G/P $.
\end{itheorem}

Recently, for $G=SL_n$, Petukhov classified reductive subgroups $ H $ of $ G $ 
and parabolic subgroups $ P $, 
for which a partial flag variety $ G/P $ is $ H $-spherical \cite{Petukhov.2011}.  
We thank the referee for pointing out the reference.

On the other hand, if $ P = B $, 
a double flag variety 
$ \GKFl{B}{Q} $ is of finite type 
if and only if 
$ G/Q $ is a $ G $-spherical variety.  
In this case, Theorem~\ref{main.thm:classify.K.orbits.on.double.flag.variety} implies that a double flag variety is of finite type
if and only if 
a certain linear action of a reductive subgroup of $K$ is
a spherical action
(Theorem~\ref{theorem:case.P=B}).  
So we can again use tables in \cite{Benson.Ratcliff.1996} 
to get a classification of such 
double flag varieties of finite type.  
See Theorem~\ref{theorem:classification.P=B} and 
Table~\ref{table:g.spherical} for details.

Another motivation to study double flag varieties $ \GKFl{B}{Q} $ of finite type comes from the theory of character sheaves. Character sheaves were first introduced by Lusztig \cite{Lu}. There are certain $G$-equivariant simple perverse sheaves on $G$, which provide a geometric theory of characters of a connected reductive group over an arbitrary algebraically closed field. 

Recently, some generalizations of character sheaves have been studied. Finkelberg, Ginzburg and Travkin developed the theory of mirabolic character sheaves in \cite{FG} and \cite{FGT.2009}. Following the work of Kato \cite{Kato.2009}, Henderson and Trapa suggested the theory of exotic character sheaves in \cite{Henderson.Trapa.2012}. These character sheaves are certain $K$-equivariant simple perverse sheaves on $V \times G/K$, where $V$ is some $K$-module. Here in the mirabolic case, $(G, K, V)=(GL_n \times GL_n, (GL_n)_{\diag}, \bbC^n)$ and in the exotic case, $(G, K, V)=(GL_{2n}, Sp_{n}, \bbC^{2n})$. A key ingredient is that there are only finitely many $K$-orbits on the generalized flag $V \times G/B$. 

One may hope that there is a generalization of character sheaves on $K/Q \times G/K$, which generalizes both the mirabolic character sheaves and exotic character sheaves. In order to do this, one first need to know when a double flag variety $ \GKFl{B}{Q} $ has only finitely many $K$-orbits. We believe that the classification of double flag varieties $ \GKFl{B}{Q} $ of finite type is a necessary ingredient for establishing the (conjectural) generalization of character sheaf theory.

The Robinson-Schensted correspondence is a bijection correspondence between permutations and pairs of standard Young tableaux of the same shape. Steinberg gave a geometric interpretation of this correspondence, by showing that both sides naturally parametrize the irreducible components of the Steinberg variety, which is by definition the conormal variety of the product of flag varieties. It is interesting to study a similar question for conormal variety of double flag of finite type. The mirabolic case was obtained by Travkin \cite{Travkin.2009}, Finkelberg-Ginzburg-Travkin \cite{FGT.2009} and the exotic case was obtained by Henderson-Trapa \cite{Henderson.Trapa.2012}, in which they also made some conjectures relating the exotic Robinson-Schensted correspondence to exotic character sheaves. 

This paper is organized as follows.
We fix basic notation and terminology in \S~\ref{section:preliminaries}.
In Theorem~\ref{main.thm:classify.K.orbits.on.double.flag.variety},
 we give a parametrization of $K$-orbits in $\GKFl{P}{Q}$.
When $Q \subset K$ is a Borel subgroup, we see in Proposition~\ref{proposition:L_JcapK.action.on.Lie.algebra.of.V} that a part of the $K$-orbit decomposition in $\GKFl{P}{Q}$ is reduced to a linear action.
Classification of the double flag varieties $\GKFl{P}{Q}$ of finite type
 in the extreme case, namely, the case where $Q=B_K$ or $P=B$
 is our main result in this paper.
In \S~\ref{section:spherical.actions},
 we reduce the finiteness of the double flag variety $\GKFl{P}{Q}$ to 
 the sphericity of a linear action
 for $Q=B_K$ (Theorem~\ref{thm:finiteness.iff.MFaction})
 and for $P=B$ (Theorem~\ref{theorem:case.P=B}).
We also recall one of Stembridge's results in Theorem~\ref{theorem:triple.classify}
 which classifies the triple flag varieties
 $G/P_1 \times G/P_2 \times G/B$ of finite type.
Classifications of double flag varieties $\GKFl{P}{Q}$ of finite type are
 given in Theorem~\ref{theorem:k.spherical} with Table~\ref{table:k.spherical}
 for $Q=B_K$ and Theorem~\ref{theorem:classification.P=B} 
 with Table~\ref{table:g.spherical} for $P=B$.  

\subsection*{Acknowledgment} 
The authors thank Roger Howe for private communication about 
his note on $ K $-spherical flag varieties with Nolan Wallach and Jozsef Horvath.

The authors also express deep thanks to two anonymous referees 
for invaluable suggestions/comments and pointing out appropriate references, 
 which greatly improved the paper.

\section{Preliminaries}
\label{section:preliminaries}

\subsection{}
\label{subsection:notation}

Let $ G $ be a connected, simply connected semisimple algebraic group over
 the complex number field $ \bbC $ 
and $ \theta $ a non-trivial involutive automorphism of $ G $.  
We put $ K = G^{\theta} =\{g\in G: \theta(g)=g\} $,
 the subgroup of fixed elements of $ \theta $, which is connected and reductive by our assumption on $ G $ 
(see \cite[Theorem 8.1]{Steinberg.1968}).  
We denote the Lie algebra of $ G $ (resp.\ $ K $) by $ \lie{g} $ (resp.\ $ \lie{k} $).  
In the following, we use similar notation; 
for an algebraic group we use a Roman capital letter, and for its Lie algebra the corresponding German small letter.

Let $ B \subset G $ be a $ \theta $-stable Borel subgroup and 
take a $ \theta $-stable maximal torus $ T $ in $ B $.  
We consider the root system $ \Delta = \Delta(\lie{g}, \lie{t}) $, the Weyl group $ W = W_G = N_G(T)/Z_G(T) $ 
with respect to $ T $, 
and the positive system $ \Delta^+ $ corresponding to $ B $.  
Then $ \Delta^+ $ determines a set of simple roots $ \Pi $.  
Since $ B $ and $ T $ are $ \theta $-stable, 
$ \theta $ naturally acts on $ W_G $ and $ \Delta $, 
and preserves $ \Delta^+ $ and $ \Pi $.  

We say that a parabolic subgroup $P$ of $G$ is standard if
 $P \supset B$.
There exists a one-to-one correspondence between
 the standard parabolic subgroups $P$
 and the subsets $ J \subset \Pi $; 
the root subsystem $ \Delta_J $ generated by $ J $ is the root system of 
the standard Levi component $ L $ of $ P $.  
Notice that the $ \theta $-stable parabolic subgroups correspond exactly to the $ \theta $-stable subsets in $ \Pi $.  
If $ P $ corresponds to $ J $, then we will write $ P = P_J $ 
and write $ P_J = L_J U_J $ for 
 the Levi decomposition, where $ L_J $ is the standard Levi factor and $ U_J $ is the unipotent radical.  
We denote the Weyl group of $ \Delta_J $ by $ W_J $
 and put $W^J:=W/W_J$.
For two subsets $ J, J' \subset \Pi $, 
 put $ \jwjprime := W_J \backslash W / W_{J'} $.
In the following, we often take representatives of elements of $W$
 in $N_G(T)$ and regard $W$ as a subset of $N_G(T)$ or of $G$.
Similarly, we take representatives of elements of $W^J$,
 $ \jwjprime $ in $W$
 and then in $N_G(T)$ so that we have $W^J, \jwjprime \subset N_G(T)$.

\subsection{}
\label{subsection:spherical}

In this subsection, we assume that $G$ is a connected reductive algebraic group
 over $\bbC$.
Let $X$ be an irreducible normal $G$-variety.
If $X$ has an open $B$-orbit for a Borel subgroup $B$ of $G$,
 it is called a spherical variety and
 the $G$-action is called a spherical action.
It is well-known that $X$ is spherical 
 if and only if there are only finitely
 many $B$-orbits in $X$ (\cite{Brion.MM.1986, Vinberg.1986}).
The $G$-action on $X$ induces a $G$-action
 on the ring of regular functions $\bbC [X]$.
It is easy to see that if $X$ is spherical, then
 the decomposition of
 $\bbC [X]$ into irreducible $G$-modules
 is multiplicity-free.
The converse is also true if $X$ is an affine variety
 (\cite{Vinberg.Kimelfeld.1978}).
If $X$ is isomorphic to a vector space and the $G$-action is linear,
 $\bbC [X]$ can be identified with the symmetric power $S(X^{\ast})$ of
 the dual space $X^{\ast}$.
Hence a linear $G$-action is spherical if and only if the decomposition
 of $S(X^{\ast})$ into irreducible $G$-modules is multiplicity-free.

\section{Parametrization of $K$-orbits in the double flag variety}
\label{section:double.flag}

Suppose that 
$ G $ is a connected, simply connected semisimple algebraic group over $ \bbC $
 with an involutive automorphism $\theta$
 and let $ K:=G^\theta $.  
Let $ P $ be a parabolic subgroup of $G$, and 
 $ Q $ a parabolic subgroup of $ K $.
We denote the partial flag varieties
 $G/P$ and $ K/Q $ by $ \GFl_P $ and $ \KFl_Q $, respectively.  
The product $ \GKFl{P}{Q} $ is called a \emph{double flag variety for symmetric pair} $ (G, K) $.  
We say that a double flag variety $ \GKFl{P}{Q} $
 is of \emph{finite type} 
 if there are only finitely many orbits on the product $ \GKFl{P}{Q} $ 
 with respect to the diagonal $ K $-action
  (see \cite{NO.2011}).

In this paper, we study the structure of the orbit space 
$ K \backslash (\GKFl{P}{Q}) $, 
and give a parametrization of orbits.  
As a consequence of the parametrization, we get a criterion to 
determine if the double flag variety is of finite type.

It is known that there exists a $\theta$-stable parabolic subgroup $ P' $
 of $ G $ such that $ Q = P'\cap K$ (\cite[Theorem 2]{Brion.Helminck.2000}).
Then by replacing $P$ with its conjugate subgroup,
 we may assume that $P$ and $P'$ are standard parabolic subgroups
 for a $\theta$-stable Borel subgroup $B$.
We use notations in 
 \S~\ref{subsection:notation} for our $G$, $K$, and $B$.
Write $J, J' \subset \Pi$ for the subsets such that 
 $ P = P_J $ and $ P' = P_{J'}$.
Let $P=LU$ and $P'=L'U'$ be the standard Levi decompositions.
This means that $L$ is the Levi subgroup of $P$ such that $L\supset T$
 and $U$ is the unipotent radical of $P$.
Similarly for $L'$ and $U'$.

We parametrize $ K $-orbits on $ \GKFl{P}{Q} $ using
 reduction by two well-known decompositions: the Bruhat decomposition
 and the KGB decomposition.  

First we reduce the orbit space by the Bruhat decomposition.

\subsection{Reduction by Bruhat decomposition}
\label{subsection:reduction.to.Bruhat}

Notice that there is a bijection
\begin{equation*}
K \backslash (\GKFl{P}{Q}) \simeq P \backslash G /Q , \qquad
K \cdot (g P, k Q) \mapsto P g^{-1} k Q \quad (g \in G, k \in K).
\end{equation*}
Since $ Q = P' \cap K $, we have the following reduction map $\Phi$.
\begin{equation*}
\xymatrix @C-4ex {
K \backslash (\GKFl{P}{Q}) 
   \ar[rrdd]_{\Phi} \ar[rr]^{\sim}
   & \quad
   & 
   P \backslash G /Q
   \ar[d]^{\text{proj}}
\\
   &
   & 
   P \backslash G / P' 
   \ar[d]^{\simeq}
   &
   \makebox[0pt][l]{$= \coprod_{w \in \jwjprime} P w P'$} 
\\
   &
   & 
   \jwjprime & 
   \makebox[0pt][l]{$ = W_J \backslash W / W_{J'} $}
}
\hspace*{.3\textwidth}
\end{equation*}
Thus we can reduce the determination of the
 orbit structure $ K \backslash (\GKFl{P}{Q}) $ to 
 the analysis of the fiber 
\begin{equation}
\label{eq:fiber.of.Bruhat.reduction}
\Phi^{-1}(w) \simeq P \backslash P w P'/ Q \quad (w\in \jwjprime).
\end{equation}

Let us fix $ w \in \jwjprime $ in the following.  
We put 
\begin{equation*}
\wP{w} := w^{-1} P w .
\end{equation*}
More generally, we write
 $ H^g = g^{-1} H g $ for any subgroup $ H \subset G$ and $ g \in G $.

\begin{lemma}
\label{lemma:bij.cosets}
The map $ (\wPprime{w}) \backslash P' / Q \to P \backslash P w P'/ Q $ given by 
$ (\wPprime{w}) \, a \, Q \mapsto P \, w a \, Q $ for $ a \in P' $
is bijective.
\end{lemma}

This is a consequence of a more general lemma below.  
In the lemma, $ G, H, H' $ refer arbitrary groups, which are different from the present notation.

\begin{lemma}
\label{lemma:general.quotient.of.dbl.cosets}
Let $ G $ be a group and $ H, H' $ its subgroups.
Let $A \subset H'$ be a subset and $g \in G$ an element.
Then the map
 $(H^g \cap H')\backslash (H^g \cap H')A \to
 H\backslash HgA$
 given by $(H^g \cap H')a \mapsto Hga$ for $a\in A$
 is bijective.
\end{lemma}

\begin{proof}
The surjectivity is clear.
If $ Hga_1= Hga_2 $ for $a_1,a_2\in A$,
 then $a_2a_1^{-1} \in H^g \cap H'$.
Hence $(H^g \cap H')a_1=(H^g \cap H')a_2$ and
 the map is injective.
\end{proof}

We apply Lemma~\ref{lemma:general.quotient.of.dbl.cosets}
in the setting where $G=G$, $H=P$, $H'=P'$, $A=P'$ and $g=w$.
Taking quotients by right $Q$-action, we get Lemma~\ref{lemma:bij.cosets}.

\subsection{Reduction by smaller symmetric spaces}
\label{subsection:reduction.to.KGB}

The orbit structure $ K \bsl (\GKFl{P}{Q}) $ is reduced to the structure of 
fibers \eqref{eq:fiber.of.Bruhat.reduction} of $ \Phi $.  
In this subsection, we further reduce it by the KGB decomposition, or by KGP decomposition we should say, 
 for smaller symmetric spaces.  For KGB decomposition, we refer the readers to 
\cite{Richardson.Springer.1990, Richardson.Springer.1994, Richardson.Springer.1993}
and 
\cite{Lusztig.Vogan.1983}.

Put $ L'_K := L' \cap K $ and consider 
$ \wPcap{L'}{w} $, which is a parabolic subgroup of $ L' $ by 
\cite[Proposition 2.8.9]{Carter.1985}.  
Then $ L'_K $ is a symmetric subgroup of $ L' $, 
and it is known that $ (\wPcap{L'}{w}) \bsl L' / L'_K $ is a finite set.  
Let us denote this finite set by $ \kgbw{w} $ for $w \in \jwjprime$.

\begin{lemma}
\label{lemma:map.w.to.v}
The map 
\begin{equation*}
\Psi_w:(\wPprime{w}) \bsl P' / Q \to (\wPcap{L'}{w}) \bsl L' / L'_K
 =\kgbw{w} 
\end{equation*}
given by $ (\wPprime{w}) \, a \, Q \mapsto (\wPcap{L'}{w}) \, b \, L'_K $ is well-defined, 
where $ a = b u \in L' U'=P' $ is the Levi decomposition.
\end{lemma}

\begin{proof}
We put $ U'_K := U' \cap K $ so that $Q=L'_K U'_K$
 and consider the following diagram.
\begin{equation*}
\xymatrix @C-4ex {
(\wPprime{w}) \bsl P' / Q 
   \ar[rrdd]_{\textstyle\Psi_w} 
   \ar[rr]^>(0)>>>>>{=}
   & \quad\quad
   & 
   (\wPprime{w}) \bsl L' U' / L'_K U'_K 
   \ar[d]^{\text{proj.}}
\\
   &
   & 
   (\wPprime{w}) \bsl L' U' / L'_K U' 
   \ar[d]_-{\textstyle{\wr}}^{\iota^{-1}}
\\
   &
   & 
   (\wPcap{L'}{w})  \bsl L' / L'_K 
   \makebox[0pt][l]{\;$ = \kgbw{w} $}
}
\end{equation*}
Here a bijective map 
\begin{equation}
\label{eqn:iota.bijection}
\iota : 
   (\wPcap{L'}{w}) \bsl L' / L'_K 
   \to 
   (\wPprime{w}) \bsl L' U' / L'_K U' 
\end{equation}
is induced by the inclusion $ L' \hookrightarrow L' U' $ and the second vertical arrow in the diagram is 
 the inverse of $ \iota$.  
The bijectivity of $ \iota $ is deduced from the following general lemma.

\begin{lemma}
\label{lemma:gen.double.coset.lemma:A}
Let $L_1 \ltimes U_1$ be a semidirect product group
 of two groups $L_1$ and $U_1$.
Let $L_2, L_3 \subset L_1$ and $U_2 \subset U_1$ be subgroups and assume that 
$L_2$ normalizes $U_2$ so that $ L_2 \ltimes U_2 $ is a subgroup of $ L_1 \ltimes U_1 $.  
Then the natural inclusion map induces the following bijections:
\[
\xymatrix @1 @M=5pt {
L_2 \backslash L_1 / L_3 
\ar[r]^-{\sim}
&
L_2 \backslash (L_1U_1)/(L_3 U_1)
\ar[r]^-{\sim}
&
(L_2U_2) \backslash (L_1U_1)/(L_3 U_1).
}
\]
\end{lemma}

\begin{proof}
It is easy to see that the both maps are well-defined and surjective.  
So it is enough to see that the composite map is injective.

Suppose $l,l' \in L_1$ satisfy
$L_2 U_2  \, l \, L_3 U_1 = L_2 U_2 \, l' \, L_3 U_1$.
This means that there exist
$l_2 \in L_2, \; l_3 \in L_3, \; u_2 \in U_2, \; u_1 \in U_1$ such that
$l' = (l_2 u_2) \, l \, (l_3 u_1) $,
 or equivalently  
 $l' = (l_2 \, l \, l_3) \bigl( ( (l \, l_3)^{-1}\, u_2 \, (l \, l_3) ) u_1 \bigr) \in L_1 \ltimes U_1$.
By the uniqueness of the semidirect product decomposition,
we have $l'=l_2 \, l \, l_3$ and hence 
 $l' \in L_2 \, l \, L_3$.
\end{proof}

To see that the map $ \iota $ in \eqref{eqn:iota.bijection} is bijective, 
we use Lemma~\ref{lemma:gen.double.coset.lemma:A}
 in the setting where 
$L_1=L'$, $U_1=U'$,
$L_2 = \wP{w} \cap L'$,
$U_2 = \wP{w} \cap U'$, and
$L_3 = L'_K$.
Note that 
\[
\wP{w} \cap P' =
(\wP{w} \cap L') \ltimes
(\wP{w} \cap U')
\]
holds 
(see \cite[Theorem 2.8.7 and Proposition 2.8.9]{Carter.1985}).  
\end{proof}

Let us summarize the above situation into a diagram:
\begin{equation*}
\xymatrix @C=3ex @M=5pt {
P \, w a \, Q \ar@{}[r]|(0.45){\in} \ar@{|->}[d] & 
    P \bsl P w P' / Q \;\; \ar[r]^(0.4){\;\;\sim\;\;} \ar@{->>}[d]_(0.4){{\Psi}_w} & 
    \;\; (\wPprime{w}) \bsl P' / Q \ar@{}[r]|(.55){\ni} \ar@{->>}[d]^(0.45){\text{projection}} & 
    (\wPprime{w}) \, a \, Q \ar@{|->}[d]
\\
(\wPcap{L'}{w}) \, b \, L'_K \ar@{}[r]|(0.45){\in} & 
    (\wPcap{L'}{w}) \bsl L' / L'_K \;\; \ar[r]^(0.46){\;\;\sim\;\;} & 
    \;\; (\wPprime{w}) \bsl P' / Q U' \ar@{}[r]|(.55){\ni} & 
    (\wPprime{w}) \, b \, Q U' 
}
\end{equation*}
where $ a = b u \in L' U' $ is the Levi decomposition.
Let us take representatives of
 $\kgbw{w} = (\wPcap{L'}{w}) \bsl L' / L'_K $ from $ L' $  
 and identify them with $ \kgbw{w} $ in the following.

\subsection{Parametrization of orbits in the double flag variety}

Now we get a rough parametrization of orbits, first by the Bruhat decomposition 
$ P \bsl G / P' \simeq \jwjprime $, then next by 
the KGB decomposition for the smaller symmetric space $ L'/L'_K $.

The following lemma describes the fiber $ \Psi_w^{-1}(v) = P \bsl P \, w v \, U' Q /Q $ for $ v \in \kgbw{w} $.

\begin{lemma}
\label{lemma:conjugate.action.on.double.cosets}
Let us fix $ w \in \jwjprime $ and $ v \in \kgbw{w} $.
\begin{thmenumerate}
\item
$ \Psi_w^{-1}(v) = P \bsl P \, w v \, U' Q /Q \simeq (\wPprime{w v}) \bsl(\wPprime{w v}) \, U' Q /Q $.
\item 
We can define the following surjective map:
\begin{equation*}
\xymatrix @R=.7ex {
(\wPcap{U'}{w v}) \bsl U' / U_K' \;\; \ar[r] & \;\; (\wPprime{w v}) \bsl(\wPprime{w v}) \, U' Q /Q \\
(\wPcap{U'}{w v}) \, u \, U_K' \;\; \ar@{|->}[r] & \;\; (\wPprime{w v}) \, u \, Q  .
}
\end{equation*}
\item
\label{lemma:conjugate.action.on.double.cosets:tem:3}
The above surjection factors through to a bijection 
\begin{equation*}
\xymatrix @1 {
\Bigl( (\wPcap{U'}{w v}) \bsl U' / U_K' \Bigr) \Bigm/\bigl(\wPcap{L'_K}{w v}\bigr) \;\; \ar[r]^<<<<{\sim} & 
    \;\; (\wPprime{w v}) \bsl(\wPprime{w v}) \, U' Q /Q .
}
\end{equation*}
Notice that $ \wPcap{L'_K}{w v} $ normalizes 
$ \wPcap{U'}{w v} $ and $ U'_K $, hence 
the conjugation action of $\wPcap{L'_K}{w v}$ on $U'$
 induces an action on $ (\wPcap{U'}{w v}) \bsl U' / U_K' $. The corresponding quotient space is the one we considered above. 
\end{thmenumerate}
\end{lemma}

\begin{proof}
(1)\ 
This follows from Lemma~\ref{lemma:general.quotient.of.dbl.cosets}.

(2)\ 
Since 
$ \wPcap{U'}{w v} \subset \wPcap{P'}{w v} $ and 
$ U'_K \subset Q $, 
our map is just a projection.

(3)\ 
We use the following general lemma, in which the notations are independent of the rest of the arguments.

\begin{lemma}
\label{lemma:gen.double.coset.lemma:B}
Let $L_1 \ltimes U_1$ be a semidirect product group
 of two groups $L_1$ and $U_1$.
Let $L_2, L_3 \subset L_1$ and $U_2,U_3\subset U_1$ be
subgroups and assume that $L_i$ normalizes $U_i$ for $i=2,3$ so that
$ L_2 \ltimes U_2 $ and 
$ L_3 \ltimes U_3 $ are subgroups of $ L_1 \ltimes U_1 $.  
\begin{thmenumerate}
\item
The conjugation action of the group $L_2 \cap L_3$ on $U_1$
by $u \mapsto l \, u \, l^{-1}$
induces a well-defined action of $L_2 \cap L_3$ on
$U_2 \backslash U_1 /U_3$.
\item
The natural map
 $U_2 \backslash U_1 / U_3
 \to (L_2 U_2)\backslash (L_2 U_1 L_3) / (L_3 U_3)$,
 $U_2 u U_3 \mapsto (L_2U_2)u(L_3U_3)$ 
induces a bijective map
\[
\xymatrix @1 @M=5pt {
\varphi:
(U_2 \backslash U_1 /U_3) / (L_2 \cap L_3) \ar[r]^-{\sim} & (L_2 U_2)\backslash (L_2 U_1 L_3)/(L_3 U_3).
}
\]
\end{thmenumerate}
\end{lemma}

\begin{proof}
The claim (1) is obvious.  

Let us prove (2).
The surjectivity of $\varphi$
 is clear since we can always take a representative of the right-hand side in $U_1$.  
We give a proof of injectivity. 
For $u, u' \in U_1$, let us assume that $(L_2 U_2) \, u \, (L_3 U_3) = (L_2 U_2) \, u' \, (L_3 U_3)$.
Then there exist $l_2 \, u_2 \in L_2 U_2$ and $l_3 \, u_3 \in L_3 U_3$ such that
$u' = (l_2 \, u_2) \, u \, (l_3 \, u_3)$.
We rewrite it as
\begin{equation*}
u'= (l_2 \, l_3) \bigl( (l_3^{-1} \, u_2 \, u \, l_3) \, u_3 \bigr) \in L_1 U_1 .
\end{equation*}
By the uniqueness of the semidirect product,
$l_2 \, l_3 = e $ and $u'=(l_3^{-1} \, u_2 \, u \, l_3) \, u_3$.
Therefore, we have $l_2=l_3^{-1} \in L_2 \cap L_3$ and
\begin{equation*}
u' = (l_2 \, u_2 \, l_2^{-1}) \, (l_2 \, u \, l_2^{-1}) \, u_3 \in U_2 \, (l_2 \, u \, l_2^{-1}) \, U_3.
\end{equation*}
This shows $u' \in (L_2 \cap L_3) \cdot (U_2 \, u \, U_3)$, where $ \cdot $ denotes the conjugation action.
\end{proof}

To prove Lemma~\ref{lemma:conjugate.action.on.double.cosets} \eqref{lemma:conjugate.action.on.double.cosets:tem:3},
 we apply Lemma~\ref{lemma:gen.double.coset.lemma:B}~(2) in the
 setting where 
\begin{equation*}
L_1 = L', \;\; 
L_2=\wP{w v} \cap L', \;\;
L_3 = L'_K, \;\;
U_1= U', \;\;
U_2=\wP{w v} \cap U', \;\;
U_3= U'_K.
\end{equation*}
We need to use again 
$
\wP{w v} \cap P' =
(\wP{w v} \cap L') \ltimes
(\wP{w v} \cap U') 
$
(\cite[Theorem 2.8.7 and Proposition 2.8.9]{Carter.1985}), 
and $Q=L'_K U'_K$.
\end{proof}

Lemma~\ref{lemma:conjugate.action.on.double.cosets} with two reductions 
(\S~\ref{subsection:reduction.to.Bruhat} and \S~\ref{subsection:reduction.to.KGB}) 
gives us a parametrization of 
$ K $-orbits in the double flag variety $ \GKFl{P}{Q} $.

\begin{theorem}
\label{main.thm:classify.K.orbits.on.double.flag.variety}
Let $ P = P_J $ and $ P' = P_{J'} $ be standard \psgs of $ G $ and 
assume that $ P' $ is $ \theta $-stable with the standard ($ \theta $-stable) Levi decomposition $ P' = L' U' $.  
Define $ Q := P' \cap K $, which is a parabolic subgroup of $ K $,
 and put $ L'_K := L'\cap K$, $ U'_K := U' \cap K $.  
\begin{thmenumerate}
\item
\label{main.thm:item:classification.via.unipotent.double.cosets}
The $ K $-orbits in the double flag variety
 $ \GKFl{P}{Q} = G/P \times K/Q $ are parametrized as follows:
\begin{equation*}
K \bsl (\GKFl{P}{Q}) \simeq \coprod_{w \in \jwjprime} 
\coprod_{v \in \kgbw{w}} 
\Bigl( (\wPcap{U'}{w v}) \bsl U' / U_K' \Bigr) \Bigm/ \wPcap{L'_K}{w v},
\end{equation*}
Here we write $ \jwjprime := W_J \bsl W / W_{J'} $, 
 $ \kgbw{w} := (P^w \cap L') \bsl L'/L'_K $
 and identify them with their representatives.
\item
The double flag variety $ \GKFl{P}{Q} $ is of finite type if and only if for any $ w \in \jwjprime $ and $ v \in \kgbw{w} $, the conjugation action of $\wPcap{L'_K}{w v}$ on the double coset space 
$(\wPcap{U'}{w v}) \bsl U' / U_K' $ has only finitely many orbits.
\end{thmenumerate}
\end{theorem}

\begin{proof}
The claim (1) was already proved.  
Since $ \jwjprime $ is a finite set and $ \kgbw{w} $ is also finite for any $ w \in \jwjprime $, the claim (2) follows.
\end{proof}

\begin{corollary}
The double flag variety $ \GKFl{P}{Q} $ is of finite type 
if and only if for any $g \in P W L'$, the conjugation action of $\wPcap{L'_K}{g}$ on 
$ (\wPcap{U'}{g}) \bsl U' / U'_K $ has only finitely many orbits.
\end{corollary}

\begin{proof}
This follows directly from
 Theorem~\ref{main.thm:classify.K.orbits.on.double.flag.variety}
 once one knows 
\begin{equation*}
\bigcup_{w \in \jwjprime} \bigcup_{v \in \kgbw{w}} P w v L'_K 
= \bigcup_{w \in \jwjprime} P w L' = P W L' .
\end{equation*}
\end{proof}

\section{Reduction to linear actions}
\label{section:reduction.to.linear}

Under the setting of \S~\ref{section:double.flag}, 
 we now assume that $ P' = B $ 
so that $ Q = B \cap K =: B_K $ is a Borel subgroup of $ K $
 and we consider the double flag variety 
$\GFl_P \times \KFl_{B_K} =G/P \times K/{B_K}$.

We take a $ \theta $-stable maximal torus $ T $ as in \S~\ref{section:preliminaries}, 
and denote by $ B = T U_B $ a Levi decomposition of $ B $ ($ U_B $ denotes the unipotent radical of $ B $).  
In our former notation, 
\begin{align*}
&
P' = B = T U_B = L' U', 
\\
&
Q  = B_K = T_K U_{B_K} = L'_K U'_K 
\qquad (T_K = T \cap K, \; U_{B_K} = U_B \cap K),
\\
&
\jwjprime = \jwj{J}{\emptyset} =: \jwj{J}{} \ni w, 
\\
&
\wPcap{L'}{w} = \wPcap{T}{w} = T \quad (\text{for any $ w \in \jwj{J}{} $}),
\\
&
\kgbw{w} = (\wPcap{L'}{w}) \bsl L' / L'_K = T \bsl T / T_K = \{ e \},
\\
&
\wPcap{L'_K}{w v} = T_K , \quad
\wPcap{U'}{w} = \wPcap{U_B}{w}. 
\end{align*}
Then Theorem~\ref{main.thm:classify.K.orbits.on.double.flag.variety} (1)
 in this case can be rewritten as follows.

\begin{proposition}
\label{proposition:parametrization.of.K.orbits.when.Q.is.Borel}
Let $P=P_J$ be a standard
 parabolic subgroup of $G$ and $B_K$ a Borel subgroup of $K$.
Then the $K$-orbits on the double flag variety are parametrized as follows:
\begin{equation}
\label{eq:Korbit.decomp.when.Q.is.Borel}
K \bsl (\GKFl{P}{B_K}) \simeq 
\coprod_{w \in \jwj{J}{}} \Bigl( (w^{-1} P w \cap U_B) \bsl U_B / U_{B_K} \Bigr) \Bigm/T_K.
\end{equation}
\end{proposition}

\begin{remark}
If $ \rank G = \rank K $, then we have $ T = T_K $.
We note that $ \rank G = \rank K $ if and only if the involutive automorphism $\theta$ is inner.
\end{remark}

We can reduce the complicated quotient space in the above theorem 
 to a quotient of a vector space by linear action of torus
 if $ w $ is the longest element.  To do so, we need some preparations.

\subsection{Linearization of unipotent double coset spaces}

In this subsection, we study the unipotent double coset spaces appearing in Proposition~\ref{proposition:parametrization.of.K.orbits.when.Q.is.Borel}.  
We will prove that the double coset space can be reduced to the quotient space
 of a linear space by a linear action under certain assumptions.  

Let $ U $ be a unipotent group on which the torus $ T^1 = \bbC^{\times} $ acts by group automorphisms.  
Let us denote by $ \rho : T^1 \to \Aut U $ the given action.  
Then 
$ T^1 $ acts on the Lie algebra 
$ \lie{u} = \Lie U $ by differential, which we also denote by the same letter $ \rho $.  

We assume the following in this subsection:

\begin{assumption}
The weights of $ T^1 $ on $ \lie{u} $ are all positive. 
\end{assumption}

\begin{lemma}
\label{lemma:isomorphism.by.product.of.exp.map}
Let $ \lie{u} = \lie{u}_1 \oplus \cdots \oplus \lie{u}_n $ be a decomposition of $ \lie{u} $ as a $ T^1 $-module and $U_i=\exp(\lie{u}_i)$.  
Then the multiplication map 
\begin{equation*}
\xymatrix @R=3pt {
\varphi : U_1 \times \cdots \times U_n \ar[r] & U 
\\
(g_1, \dots, g_n) \ar@{|->}[r] & g_1 \cdots g_n
}
\end{equation*}
is an isomorphism.
\end{lemma}

\begin{proof}
The argument goes by the induction of the dimension of $\lie{u}$.
Let $m$ be the maximum weight appearing in $\lie{u}$.
We may assume that the weight $m$ appears in $\lie{u}_i$ for some $1\leq i\leq n$.
Take a non-zero weight vector $X \in \lie{u}_i$ with weight $m$.
Let $\lie{z} = \bbC X$.
Then $\lie{z}$ is contained in the center 
of the Lie algebra $\lie{u}$.

Now we set up the induction.
Let $\bar{\lie{u}}_i := \lie{u}_i / \lie{z}$,
$\bar{\lie{u}}_j := \lie{u}_j$ for $j\neq i$,
$\bar{\lie{u}} := \bigoplus_{j=1}^n \bar{\lie{u}}_j$,
$Z := \exp(\lie{z})$,
$\bar{U} := U/Z$.
Then the map
$\bar\varphi: \bigoplus_{j=1}^n \bar{\lie{u}}_j \rightarrow \bar{U}$
is defined by
$(Y_1,\dots, Y_n) \mapsto \exp Y_1 \cdots \exp Y_n$.
We assume $\bar\varphi$ is bijective by induction hypothesis.

Let us prove the surjectivity of $\varphi$.
Take $u \in U$ and write $\bar{u} \in \bar{U}$ for its image by the quotient map.
Since $\bar\varphi$ is surjective, there exists 
$(Y_1,\dots,Y_n) \in \bar{\lie{u}}$ such that
$\bar\varphi(Y_1,\dots,Y_n) = \bar{u}$.
If we take a lift $X_i \in \lie{u}_i$ of $Y_i \in \lie{u}_i / \lie{z}$
 and put $X_j:=Y_j$ for $j\neq i$,
then the image of 
$\varphi(X_1,\dots,X_n) \in U$
and that of $u \in U$ in $\bar{U}$ are equal.
Hence there exists $z_i \in Z$ such that
$u= \varphi(X_1,\dots, X_n) z_i$.
Take $Z_i \in \lie{z}$ such that $\exp(Z_i) = z_i$.
Then $u= \varphi(X_1,\dots,X_{i-1},X_i+Z_i,X_{i+1},\cdots, X_n)$, showing
the surjectivity of $\varphi$. 

The injectivity of $\varphi$ is similarly proved.
Suppose $(X_1,\dots, X_n), (X'_1,\dots, X'_n) \in \lie{u}$
has the same image by $\varphi$ in $U$.
Put $(Y_1,\dots, Y_n)\in \bar{\lie{u}}$ the image of $(X_1,\dots, X_n)$.
Similarly for $(Y'_1,\dots, Y'_n)$.
Since 
$\bar\varphi(Y_1,\dots,Y_n) = \bar\varphi(Y'_1,\dots, Y'_n)$,
 we have $(Y_1,\dots, Y_n) = (Y'_1,\dots, Y'_n)$
 by the injectivity assumption of $\bar\varphi$.
Hence $X_j = X'_j$ for $j\neq i$.
Then
\begin{align*}
&\exp(X_1) \cdots \exp(X_{i-1}) \exp(X_i) \exp(X_{i+1})\cdots  \exp(X_n)\\
&= \exp(X_1) \cdots \exp(X_{i-1}) \exp(X'_i) \exp(X_{i+1})\cdots  \exp(X_n)
\end{align*}
and this implies $\exp(X_i) = \exp(X'_i)$.
Since the exponential map is bijective for a unipotent group,
we have $X_i = X'_i$.

The lemma follows from the fact that a bijective morphism $f:X\to Y$ between algebraic varieties over an algebraically closed field of characteristic zero is an isomorphism if $Y$ is normal (see \cite[Theorem 5.1.6 (iii) and Theorem 5.2.8]{Springer.1998}).
\end{proof}

Let $ U $ and $ T^1 $ be as above.  
Let $ U_1 , U_2 \subset U $ be subgroups of $ U $ which are stable under the action of $ T^1 $.  
Take a decomposition $ \lie{u} = (\lie{u}_1 + \lie{u}_2) \oplus \lie{V} $ as a $ T^1 $-module.  
We further assume that $ \lie{V} $ is stable under the adjoint action of $ U_1 \cap U_2 $.

\begin{lemma}
\label{proposition:linearinzation.of.double.cosets}
Under the above notations and assumptions, 
the map
\begin{equation*}
\Phi:\lie{V} / (U_1 \cap U_2) \to  U_1 \bsl U/ U_2
\end{equation*}
given by $(U_1 \cap U_2)\cdot Z\mapsto U_1(\exp Z)U_2$
is bijective.
Here $ U_1 \cap U_2 $ acts on $ \lie{V} $ by the adjoint action.
\end{lemma}

\begin{proof}
We take a $T^1$-stable complementary subspace
 $ \lie{w}_i \; (i = 1, 2) $ of $ \lie{u}_1 \cap \lie{u}_2 $ in $ \lie{u}_i $
 so that 
$ \lie{u}_i = \lie{w}_i \oplus ( \lie{u}_1 \cap \lie{u}_2 ) $ 
is a decomposition of a $ T^1 $-module.  
Then we have a decomposition of $ \lie{u} $ as 
\begin{equation*}
\lie{u} = \lie{w}_1 \oplus \lie{V} \oplus ( \lie{u}_1 \cap \lie{u}_2 ) \oplus \lie{w}_2. 
\end{equation*}
By applying Lemma~\ref{lemma:isomorphism.by.product.of.exp.map} to this decomposition, we see that every element $u\in U$ is uniquely written as
 $u=u_1(\exp Z)u_3u_2$ where
 $u_1\in \exp \lie{w}_1$, $Z\in \lie{V}$, $u_3 \in U_1\cap U_2$,
 and $u_2 \in \exp \lie{w}_2$.
Therefore, the map $\Phi$ is surjective.
To prove the injectivity, 
 suppose that $ \exp Z_1 = u_1 (\exp Z_2) u_2 $ for $ Z_1, Z_2 \in \lie{V} $ and $ u_i \in U_i \; (i = 1, 2) $.  
By Lemma~\ref{lemma:isomorphism.by.product.of.exp.map} again,
 we have $ u_1 = u_1'' \, u_1' $ where 
$ u_1'' \in \exp (\lie{w}_1) $ and 
$ u_1'  \in \exp (\lie{u}_1 \cap \lie{u}_2) = U_1 \cap U_2 $.  
Also we can write 
$ u_2 = u_2' \, u_2''$ where 
$ u_2'  \in U_1 \cap U_2 $ and  
$ u_2'' \in \exp (\lie{w}_2) $.  
Then we can compute as
\begin{align*}
u_1 (\exp Z_2) \, u_2
&= u_1'' u_1' (\exp Z_2) \, u_2' u_2'' 
\\
&= u_1'' (\exp \Adj (u_1') Z_2) \, (u_1' u_2') u_2''
\in 
\exp \lie{w}_1 \, \exp \lie{V} \, \exp ( \lie{u}_1 \cap \lie{u}_2 ) \, \exp \lie{w}_2 .
\end{align*}
Since the decomposition is unique, 
we get 
$ Z_1 = \Adj(u_1') Z_2 $, 
which shows $ Z_1 $ and $ Z_2 $ are in the same 
$ \Adj(U_1 \cap U_2) $-orbit.
\end{proof}

\subsection{Parametrization of open stratum via linear actions}

In general, we cannot apply
 Lemma~\ref{proposition:linearinzation.of.double.cosets}
to Proposition~\ref{proposition:parametrization.of.K.orbits.when.Q.is.Borel}
 because our assumptions do not hold for general $ w \in \jwj{J}{} = W_J \bsl W $.  
However, we can apply it to 
the fiber of the longest element of $ \jwj{J}{}  $, which corresponds to
 the open stratum of the Bruhat decomposition $P\bsl G / B$.  
Note that this stratum gives an open $ K $-stable set in $ \GKFl{P}{B_K} $.  

Let us denote by $ w_0 \in W $ the longest element in the Weyl group and 
take the minimal representative $ \tilde{w}_0 \in \jwj{J}{} $ of 
a coset $ W_J w_0  $ containing the longest element.   
We put 
\begin{align}
&
\opposite{J} := - w_0(J) \subset \Pi, 
\label{Eq:definition.of.Jstar}
\\
&
P_{\opposite{J}} = L_{\opposite{J}} U_{\opposite{J}} 
\text{ : Levi decomposition}, 
\notag
\\
&
\wP{\tilde{w}_0} \cap U_B = L_{\opposite{J}} \cap U_B = U_{L_{\opposite{J}}} 
\text{ : a maximal unipotent subgroup of $ L_{\opposite{J}} $}.
\notag
\end{align}
With this notation, we can 
state the following proposition which linearizes the unipotent double coset in 
\eqref{eq:Korbit.decomp.when.Q.is.Borel}.

\begin{proposition}
\label{proposition:L_JcapK.action.on.Lie.algebra.of.V}
\begin{thmenumerate}
\item
The intersection 
$
L_{\opposite{J}} \cap K = 
L_{\opposite{J} \cap \theta(\opposite{J})} \cap K 
$ is a connected reductive group 
and $ B_K \cap L_{\opposite{J}} $ is a Borel subgroup of $ L_{\opposite{J}} \cap K $.
\item
The reductive subgroup $ L_{\opposite{J}} \cap K $ acts on
 $ \lie{u}_{P_{\opposite{J}}}
 \cap \lie{g}^{-\theta} $ by the adjoint action 
and the exponential map induces 
a bijective map 
\begin{equation}
\xymatrix @1 @M=5pt {
(\lie{u}_{P_{\opposite{J}}}
 \cap \lie{g}^{-\theta} )
 / (B_K \cap L_{\opposite{J}}) \ar[r]^-{\sim} & 
\Bigl( (\wP{\tilde{w}_0} \cap U_B) \bsl U_B / U_{B_K} \Bigr) \bigm/ T_K .
}
\label{eq:fiber.is.spherical}
\end{equation}
Note that the target space is the term for $ w=\tilde{w}_0 $
 in Equation \eqref{eq:Korbit.decomp.when.Q.is.Borel}.
\end{thmenumerate}
\end{proposition}

\begin{proof}
(1)\ 
It follows that 
$
L_{{\opposite{J}}} \cap K = L_{{\opposite{J}}} \cap L_{\theta({\opposite{J}})} \cap K = 
L_{{\opposite{J}} \cap \theta({\opposite{J}})} \cap K 
$.  
Then $ L_{{\opposite{J}} \cap \theta({\opposite{J}})} \cap K$
 is a Levi component
 of the parabolic subgroup
 $P_{{\opposite{J}} \cap \theta({\opposite{J}})} \cap K$
 of $K$. Hence it is connected and reductive.  

We have 
$
B_K \cap L_{\opposite{J}} = (B \cap L_{{\opposite{J}} \cap \theta({\opposite{J}})}) \cap K 
$.  
Since 
$ B \cap L_{{\opposite{J}} \cap \theta({\opposite{J}})} $ is a $ \theta $-stable Borel subgroup of $ L_{{\opposite{J}} \cap \theta({\opposite{J}})} $,  
 it cuts out a Borel subgroup of 
$ (L_{{\opposite{J}} \cap \theta({\opposite{J}})})^{\theta} = L_{{\opposite{J}} \cap \theta({\opposite{J}})} \cap K = L_{\opposite{J}} \cap K $.

(2)\ 
Since $ L_{{\opposite{J}}}$ normalizes $\lie{u}_{P_{\opposite{J}}}$
 and $K$ normalizes $\lie{g}^{-\theta}$,
 the intersection $L_{\opposite{J}}\cap K$ acts on
 $\lie{u}_{P_{\opposite{J}}} \cap \lie{g}^{-\theta}$.
We use Lemma~\ref{proposition:linearinzation.of.double.cosets}  
 by taking 
$ U = U_B$, $U_1 = U_B \cap L_{\opposite{J}} $ and $ U_2 = U_{B_K} $.  
Since $ T_K $ contains a regular element of $ G $, we can take a subtorus 
$ T^1 = \bbC^{\times} $ in $ T_K $ so that the weights in $ \lie{u} = \lie{u}_B $ are all positive.
Then 
$ \lie{u}_B \cap \lie{l}_{\opposite{J}} $ and $ \lie{u}_{B_K} = \lie{u}_B^{\theta} $ are stable under $ T^1 $.  
We have
\begin{align*}
&\lie{u}_1 = \lie{u}_B \cap \lie{l}_{\opposite{J}} , 
\qquad
\lie{u}_2 = \lie{u}_{B_K} = \lie{u}_B^{\theta} , \\
&\lie{u}_1 \cap \lie{u}_2 
= (\lie{u}_B \cap \lie{l}_{\opposite{J}}) \cap \lie{u}_{B_K} 
 = \lie{u}_B^{\theta} \cap \lie{l}_{\opposite{J}}
 = \lie{u}_B^{\theta} \cap (\lie{l}_{\opposite{J}} \cap \lie{l}_{\theta(\opposite{J})}) ,
\\
&\lie{u}_1 + \lie{u}_2 
 = (\lie{u}_B \cap \lie{l}_{\opposite{J}}) + \lie{u}_{B}^{\theta} 
 = (\lie{u}_B \cap (\lie{l}_{\opposite{J}} + \lie{l}_{\theta(\opposite{J})})) + \lie{u}_{B}^{\theta} ,
\\
&\lie{u}_{P_{\opposite{J}}} \cap \lie{g}^{-\theta} = \lie{u}_{P_{\opposite{J}}} \cap \lie{u}_{P_{\theta(\opposite{J})}} \cap \lie{g}^{-\theta},
\\
&\lie{u} = \lie{u}_B 
 = (\lie{u}_B \cap (\lie{l}_{\opposite{J}} + \lie{l}_{\theta(\opposite{J})}))
  \oplus  (\lie{u}_{P_{\opposite{J}}} \cap \lie{u}_{P_{\theta(\opposite{J})}})
 = (\lie{u}_1 + \lie{u}_2) \oplus (\lie{u}_{P_{\opposite{J}}} \cap \lie{g}^{-\theta}).  
\end{align*}
Here $\lie{u}_{P_{\opposite{J}}}$ is the nilradical of
 $\lie{p}_{\opposite{J}}=\Lie P_{\opposite{J}}$. 
Since $ U_1 \cap U_2 = U_{B_K} \cap L_{\opposite{J}} $ is contained in $ L_{\opposite{J}} \cap K$, 
it acts on $ \lie{u}_{P_{\opposite{J}}} \cap \lie{g}^{-\theta} $ via the adjoint action.  
Now we take 
$ \lie{V} = \lie{u}_{P_{\opposite{J}}} \cap \lie{g}^{-\theta}  $ in Lemma~\ref{proposition:linearinzation.of.double.cosets} and conclude that the exponential map induces
 a bijective map
\begin{equation*}
\xymatrix @1 @M=5pt {
(\lie{u}_{P_{\opposite{J}}} \cap \lie{g}^{-\theta})
 / (U_{B_K} \cap L_{\opposite{J}}) \ar[r]^-{\sim} & 
(U_B \cap L_{\opposite{J}}) \bsl U_B / U_{B_K}.
}
\end{equation*}
The torus $ T_K $ acts on the both hand sides by the adjoint (or conjugation) action and we see that 
\begin{equation*}
\xymatrix @1 @M=5pt {
(\lie{u}_{P_{\opposite{J}}} \cap \lie{g}^{-\theta})
 / (B_K \cap L_{\opposite{J}}) \ar[r]^-{\sim} & 
\bigl( (U_B \cap L_{\opposite{J}}) \bsl U_B / U_{B_K} \bigr) \bigm/ T_K. 
}
\end{equation*}
\end{proof}

\section{Spherical actions}
\label{section:spherical.actions}

From the next section, we will give a classification of 
double flag varieties $ \GKFl{P}{Q} $ of finite type when $ P = B $ is a Borel subgroup of $ G $ or 
$ Q = B_K $ is a Borel subgroup of $ K $.  
For this purpose, in this section, we summarize some known properties and 
consequences of our results in \S\S~\ref{section:preliminaries}--\ref{section:reduction.to.linear}.

\subsection{Spherical action of a symmetric subgroup on a partial flag variety}
\label{subsection:Q.is.Borel}

Let us consider a double flag variety $ \GKFl{P}{B_K} $, where $ B_K $ is a Borel subgroup of $ K $.  
Since $ \GKFl{P}{B_K} / K \simeq B_K \backslash G / P $, we get several equivalent conditions for finiteness of orbits:
\\[2ex]
\hfil
\begin{tabular}{c@{}l}
$ \GKFl{P}{B_K} $ {\ } & is of finite type \\
$ \iff $ &
$ \GFl_P = G/P $ has finitely many $ B_K $-orbits 
\\
$ \iff $ &
$ \GFl_P $ is $ K $-spherical, i.e., 
there exists an open dense $ B_K $-orbit on $ \GFl_P $
\\
$ \iff $ &
there exists an open dense $ K $-orbit on $ \GKFl{P}{B_K} $.
\end{tabular}
\hfil
\\[2ex]
In this situation, we can apply the following theorem by Panyushev.

\begin{theorem}[{\cite{Panyushev.1999.MMath}}]
\label{thm:Panyushev.criterion}
Let $ H $ be a connected reductive group which acts on 
a smooth variety $ X $ and $ M \subset X $ a smooth locally closed $ H $-stable subvariety. 
Let $ B $ be a Borel subgroup of $ H $.  
Then generic stabilizers for the actions of $ B $ on $ X $, the normal bundle $ T_M X $, and 
the conormal bundle $ T_M^{\ast} X $ 
are isomorphic.
In particular, the following are equivalent:
 $X$ is $H$-spherical;
 $T_M X $ is $H$-spherical;
 $T_M^{\ast} X $ is $H$-spherical.
\end{theorem}

We now get a criterion for the finiteness of the double flag variety
 $ G/P \times K/B_K $.

\begin{theorem}
\label{thm:finiteness.iff.MFaction}
Let $ B_K $ be a Borel subgroup of $ K $ and $ P = P_J $ a parabolic subgroup of $ G $.  
Then the following are all equivalent.  
\begin{thmenumerate}
\item
The double flag variety $ \GKFl{P}{B_K} = G/P \times K/B_K $ is of finite type.
\item
$ \GFl_P = G/P $ is $K$-spherical.
\item
For some $K$-orbit $\calorbit\subset \GFl_P$ 
the normal bundle 
$T_{\calorbit} \GFl_P$ is $K$-spherical.  
(Hence so is for any $ K $-orbit.)
\item
For some $K$-orbit $\calorbit\subset \GFl_P$ 
the conormal bundle 
$ T_{\calorbit}^{\ast} \GFl_P$ is $K$-spherical.
(Hence so is for any $ K $-orbit.)
\item
\label{item:finiteness.iff.MFaction.BK}
The adjoint action of $L\cap K$ on
 $\lie{u}_{P} \cap \lie{g}^{-\theta}$ is spherical.
\end{thmenumerate}
\end{theorem}

\begin{proof}
We have already seen the equivalence between (1) and (2).  
The equivalence of (2), (3) and (4) follows from 
Theorem~\ref{thm:Panyushev.criterion}.  

To see that (5) is equivalent to others, 
let us consider the $ K $-orbit $ \calorbit $ through the base point $ P $.  
Then $\calorbit\simeq K/(P\cap K)$ and the fiber of
 the conormal bundle $ T_{\calorbit}^{\ast} \GFl_P $ at $P$ is isomorphic to 
 $(\lie{g}/(\lie{p}+\lie{k}))^*\simeq \lie{u}_{P} \cap \lie{g}^{-\theta}$.
Hence $  T_{\calorbit}^{\ast} \GFl_P \simeq K \times^{P \cap K} ( \lie{u}_{P} \cap \lie{g}^{-\theta} )$.
Let us take the opposite Borel subgroup $ B_K^- $ of $ B_K $.  
Since $ B_K^- (P \cap K) $ is open in $ K $, 
 the conormal bundle has an open $ B_K^- $-orbit 
 if and only if there is an open 
 $(P \cap B_K^-)$-orbit in the fiber 
$ \lie{u}_{P} \cap \lie{g}^{-\theta} $.  
Notice that $ P \cap B_K^- = L \cap B_K^- $ is 
a Borel subgroup of $ L \cap K$.  
Thus (5) is equivalent to (4).
\end{proof}

\begin{remark}
Note that Condition (5) is equivalent to 
requiring the orbit space \eqref{eq:fiber.is.spherical} in Proposition~\ref{proposition:L_JcapK.action.on.Lie.algebra.of.V} 
to be finite  
for $ \opposite{J} $ instead of $ J $.  
In fact, the conditions for $ \opposite{J} $ and $ J $ are equivalent in view of 
the equivalence of Conditions (3) and (4) above.
This can be also seen directly by taking a Weyl involution of $G$
 which sends $P_J$ to $P_{\opposite{J}}$ and stabilizes $K$.
\end{remark}

\subsection{Spherical fiber bundle over symmetric space}

In this subsection, we consider the case where $ P = B $ is a Borel subgroup  
 so $ J = \emptyset $ in our previous notation and 
 the double flag variety is $\GKFl{B}{Q}=G/B\times K/Q$. 
Recall that we have assumed that $ B $ is $ \theta $-stable.  

We summarize the notation here, which is adapted to the present situation.  
\begin{align*}
&\jwjprime = \jwj{}{J'} \ni w ,
\\
&\wPcap{L'}{w} = w^{-1} B w \cap L' \text{ : a Borel subgroup of $ L' $,} 
\\
&\kgbw{w} = (\wPcap{L'}{w}) \bsl L' / L'_K = (w^{-1} B w \cap L') \bsl L' / L'_K 
\ni v \text{ : a representative in $ L' $,}
\\
&\wPcap{L'_K}{w v} = (L' \cap (w v)^{-1} B w v) \cap K, 
\quad
\text{ 
\begin{minipage}[t]{.5\textwidth}
where $ L' \cap (w v)^{-1} B w v $
 runs over all the Borel subgroups in $ L' $ 
up to $ L'_K $-conjugacy,
\end{minipage}
}
\\
&\wPcap{U'}{w v} = (w v)^{-1} B w v \cap U' 
= v^{-1} ( w^{-1} B w \cap U') v 
= v^{-1} (U' \cap w^{-1} U_B w) v .
\end{align*}
Thus we have 
\begin{multline*}
\Bigl( (\wPcap{U'}{w v}) \bsl U' / U'_K \Bigr)
 \Bigm/ \bigl(\wPcap{L'_K}{w v}\bigr)
\\
\simeq 
\Bigl( (U' \cap (wv)^{-1} U_B wv) \bsl U' / U'_K \Bigr)
 \Bigm/  \bigl(L' \cap (wv)^{-1} B wv \cap K \bigr),
\end{multline*}
where 
$ U'_{K} := U' \cap K $.
So Theorem~\ref{main.thm:classify.K.orbits.on.double.flag.variety} 
\eqref{main.thm:item:classification.via.unipotent.double.cosets} becomes
\begin{equation*}
K \bsl (\GKFl{B}{Q}) \simeq \coprod_{w \in \jwj{}{J'}} 
\coprod_{v\in \kgbw{w}} 
\Bigl( (U' \cap (wv)^{-1} U_B wv) \bsl U' / U'_K \Bigr)
 \Bigm/  \bigl(L' \cap (wv)^{-1} B wv \cap K \bigr) .
\end{equation*}

Since 
$ \GKFl{B}{Q} $ is of finite type if and only if 
$ G/Q $ is $ G $-spherical, 
we can concentrate on the existence of an open $ B $-orbit in $G/Q$.  
Take the longest element $ w_0 \in W$ so that 
\begin{equation*}
U' \cap w_0^{-1} U_B w_0 = \{ e \} 
\quad \text{ and } \quad 
L' \cap w_0^{-1} B w_0 = L' \cap B^- ,
\end{equation*}
where $ B^- $ denotes the opposite Borel subgroup
 corresponding to $ \Delta^- = - \Delta^+ $.  
Therefore, we get 
\begin{align*}
&(U' \cap (w_0 v)^{-1} U_B w_0 v) \bsl U' / U'_K = U' / U_K' 
\intertext{and }
&L' \cap (w_0 v)^{-1} B w_0 v \cap K = v^{-1} (L' \cap B^-)v \cap K .
\end{align*}
Since 
$ \kgbw{w_0} = (L' \cap B^-) \bsl L' / L'_K $ and 
$ (L' \cap B^-) \bsl L' $ is the full flag variety of $ L' $, 
 the subgroup $ v^{-1} (L' \cap B^-)v \; (v \in \kgbw{w_0}) $ runs over all
 the Borel subgroups of $ L' $ up to $ L'_K $-conjugacy.
However, for the existence of an open orbit, it is enough to consider 
a Borel subgroup $ v^{-1} (L' \cap B^-) v $ of $ L' $ such that
$ v^{-1} (L' \cap B^-) v L'_K$ is open in $ L' $.  
Let us denote such a Borel subgroup by 
$ \openBL = v^{-1} (L' \cap B^-) v $.  
Then we conclude that $\GKFl{B}{Q}$ is of finite type if and only if
 $ U'/U'_K $ has an open
 $ \openBL $-orbit.

\medskip

To describe such Borel subgroups, 
let us briefly recall some general facts about the double coset space $ B \bsl G / K $ 
and its relation to minimal $ \theta $-split parabolic subgroups.  
For these facts, we refer the readers to \cite[\S~1]{Vust.1974}.  
Also references \cite[\S~2]{Springer.1986} and \cite{Helminck.Wang.1993} will be useful.  

\smallskip

A \psg $ P $ of $ G $ is called $ \theta $-split if 
$ P $ and $ \theta(P) $ are opposite to each other.  
In other words, the intersection $ P \cap \theta(P) $ is a Levi component of $ P $, 
which is a $ \theta $-stable reductive subgroup of the same rank as $ G $.  
It is known that there exists uniquely a minimal $\theta$-split
 parabolic subgroup of $G$ up to $K$-conjugacy (\cite[Proposition~5]{Vust.1974}).

Let $ \Pmin $ be a minimal $ \theta $-split \psg of $ G $.  
Then $ \Pmin $ has a $ \theta $-stable Levi subgroup which contains 
a $ \theta $-stable maximal torus $ H $.  
Put $A := \exp \lie{h}^{-\theta} $ (maximal $ \theta $-split torus) and
 $M := Z_K(A)$, the centralizer of $A$ in $K$.
Then it follows that $\Pmin\cap K=M$ and
 $\Pmin=MAN$, where $N$ is the unipotent radical of $\Pmin$ (\cite[Proposition~2]{Vust.1974}). 

\begin{lemma}
\label{lemma:open.orbit.on.flag.variety}
\begin{thmenumerate}
\item
For $ v \in G $,  
 the set $ B v K $ is open in $ G $ if and only if
 there exists a minimal $\theta$-split parabolic subgroup $\Pmin$
 that contains  $ v^{-1} B v $.
\item
If $v^{-1} B v \subset \Pmin$ for $ v \in G $,
 then $v^{-1} B v\cap K=v^{-1} B v\cap M$
 and the identity component of $v^{-1} B v\cap M$ 
 is a Borel subgroup of the identity component of $M$.
\end{thmenumerate}
\end{lemma}

\begin{proof}
(1) This follows from the remark after Corollary to Theorem 1 in 
\cite{Vust.1974}.  

(2) The equation $v^{-1} B v\cap K=v^{-1} B v\cap M$ follows from
\begin{align*}
v^{-1} B v\cap K&=v^{-1} B v\cap \theta(v^{-1} B v) \cap K
\subset \Pmin \cap \theta(\Pmin) \cap K
=MA\cap K=M.
\end{align*}
For the last assertion, we work on the Lie algebra level.
Since $\adj(v^{-1})\lie{b}$ is a Borel subalgebra
 (i.e.\ a maximal solvable subalgebra) of $\lie{p}_{\text{min}}$, 
 it contains the nilradical $\lie{n}$ of $\lie{p}_{\text{min}}$. 
Hence we have
 $\adj(v^{-1})\lie{b}
 =(\adj(v^{-1})\lie{b}\cap (\lie{m}+\lie{a}))\oplus \lie{n}$ and 
 $\adj(v^{-1})\lie{b}\cap (\lie{m}+\lie{a})$
 is a Borel subalgebra of $\lie{m}+\lie{a}$.
Moreover since $\lie{a}$ is contained in the center of $\lie{m}+\lie{a}$,
 we have $\adj(v^{-1})\lie{b}=
 (\adj(v^{-1})\lie{b}\cap \lie{m})+\lie{a}+\lie{n}$
 and $\adj(v^{-1})\lie{b}\cap \lie{m}$ is a Borel subalgebra
 of $\lie{m}$.
\end{proof}

Now we return to the situation in the former paragraph.  
Since $\lie{u'}$ is $\theta$-stable,
 we have a decomposition
 $\lie{u'}=\lie{u}'_K\oplus (\lie{u'})^{-\theta}$.
Therefore, we get an $L'_K$-equivariant isomorphism 
$ (\lie{u'})^{-\theta}\simeq U'/ U'_K$ 
via the exponential map.

\begin{theorem}
\label{theorem:case.P=B}
Let $ P' $ be a $ \theta $-stable \psg of $ G $
 such that $ P'\cap K = Q $.  
We denote by $ P' = L' U' $ the standard Levi decomposition.  
Let $\Pmin'$ be a minimal $\theta$-split parabolic subgroup
 of $L'$ and put $M':=\Pmin' \cap K$.
Then the following three conditions are all equivalent.

\begin{thmenumerate}
\item
The double flag variety 
$ \GKFl{B}{Q} = G/B \times K/Q $ is of finite type.
\item
The homogeneous space $ G/Q $ is $ G $-spherical.
\item
The adjoint action of the identity component $M'_0$ of $M'$ on
$ (\lie{u'})^{-\theta} $ is spherical.
\end{thmenumerate}
\end{theorem}

\begin{proof}
The equivalence of (1) and (2) is clear.  

Let $ \openBL $ be a Borel subgroup of $ L' $ such that
 $ \openBL \cdot L'_K $ is open in $ L' $.  
By Lemma~\ref{lemma:open.orbit.on.flag.variety}, 
we may assume that $ \openBL $ is contained in $ \Pmin' $ and 
then the connected component of 
 $ \openBL \cap K $ is a Borel subgroup of $ M'_0 $. 
Since we have already seen above that $\GKFl{B}{Q}$
 is of finite type if and only if $U'/U'_K$ has an open
 $(\openBL \cap K)$-orbit, the equivalence of (1) and (3)
 follows from the isomorphism $U'/U'_K \simeq (\lie{u}')^{-\theta}$.
\end{proof}

\begin{corollary}
\label{corollary:P=B.to.Q=B_K}
If the double flag variety $ \GKFl{B}{Q} $ is of finite type, 
then the double flag variety 
 $ \GKFl{P'}{B_K} $ is also of finite type.  
\end{corollary}

\begin{proof}
By Theorem~\ref{theorem:case.P=B}, 
the adjoint action of $ M'_0 $ on 
 $ (\lie{u'})^{-\theta} $ is spherical.  
Since $ M' \subset L'_K $,
we conclude that $ \GKFl{P'}{B_K} $ is of finite type by 
Theorem~\ref{thm:finiteness.iff.MFaction}.
\end{proof}

\subsection{Triple flag varieties}
\label{subsection:triple.flag.variety}

Let us take three parabolic subgroups $ P_1, P_2 $ and $ P_3 $ of $ G $.  
If one considers $ \bbG = G \times G $ and an involution $ \theta(g_1, g_2) = (g_2, g_1) $ of $ \bbG $, 
the symmetric subgroup $ \bbK = \bbG^{\theta} $ is the diagonal subgroup 
$ \diag(G) \subset \bbG $.  
Thus $ (G \times G, \diag(G)) $ is a symmetric pair.  
Then $ \bbP = P_1 \times P_2 $ is a parabolic subgroup of $ \bbG $ and $ \bbQ = \diag(P_3) $ is a parabolic subgroup of $ \bbK $. Therefore 
our double flag variety becomes
\begin{equation*}
\bbG/ \bbP \times \bbK/ \bbQ
= (G \times G)/(P_1 \times P_2) \times (\diag(G)/\diag(P_3))
\simeq \GFl_{P_1} \times \GFl_{P_2} \times \GFl_{P_3},
\end{equation*}
which is a triple flag variety for $G$.  
So the double flag variety for symmetric pair 
 is a generalization of the triple flag variety.
We can take a parabolic subgroup 
 $ \bbP' = P_3 \times P_3 $ of $\bbG$
 so that $ \bbP' \cap \bbK = \diag(P_3) = \bbQ $ holds.

A triple flag variety $ \GFl_{P_1} \times \GFl_{P_2} \times \GFl_{P_3} $ 
is said to be of finite type if there are only
 finitely many $ G $-orbits in it.  
Note that this terminology agrees with that of
 the double flag variety for symmetric pair.  

If $ T $ is a maximal torus of $ G $, then $ \bbT = T \times T $ is a maximal torus of $ \bbG $.  
The root system of $ \bbG $ with respect to this $ \bbT $ is 
decomposed as
$ \Delta_{\bbG} = \Delta^{(1)} \sqcup \Delta^{(2)} $, 
where $ \Delta^{(1)} $ denotes the roots
 of the first factor and 
$ \Delta^{(2)} $ denotes the roots
 of the second factor.  
For a parabolic subgroup $ P_i $ of $ G $, there corresponds
 to a subset $ J_i $ of the set of simple roots $ \Pi $.  
We put $ \bbJ = J_1^{(1)} \sqcup J_2^{(2)} $.  
Here $ J_i^{(i)} $ is a copy of $ J_i $ in $ \Delta^{(i)} $ for $ i=1,2 $.

Let us specialize what we have already proved
 to the case of triple flag variety.  
In order to do this 
we will give a list of notations, which tells 
the correspondence of concepts.
\begin{align*}
&\bbP \bsl \bbG / \bbP' = (P_1 \bsl G / P_3) \times (P_2 \bsl G / P_3) 
\\
&\qquad\simeq \jbbwj{\bbJ}{\bbJ'} = \jwj{J_1}{J_3} \times \jwj{J_2}{J_3} 
\ni \bbw = (w_1, w_2),
\\
&\bbL' = L_3 \times L_3 \supset \bbL'_{\bbK} = \diag(L_3),
\\
&\bbU' = U_3 \times U_3,
\\
&\bbP^\bbw\cap \bbP' =
\bbw^{-1} \bbP \bbw \cap \bbP' 
= (w_1^{-1} P_1 w_1 \cap P_3) \times (w_2^{-1} P_2 w_2 \cap P_3), 
\\
&\bbP^\bbw\cap \bbL' =
\bbw^{-1} \bbP \bbw \cap \bbL' 
= (w_1^{-1} P_1 w_1 \cap L_3) \times (w_2^{-1} P_2 w_2 \cap L_3), 
\\
&(\bbP^\bbw \cap \bbL') \bsl \bbL' / \bbL'_{\bbK} 
= \bigl((P_1^{w_1}\cap L_3) \times (P_2^{w_2}\cap L_3) \bigr)
 \bsl (L_3 \times L_3) / \diag(L_3) \\
&\qquad\overset{(\star)}{\simeq} 
(P_1^{w_1}\cap L_3) \bsl L_3 / (P_2^{w_2}\cap L_3) 
= \kgbw{\bbw} \ni v.
\end{align*}
In the case of triple flag variety, 
the quotient of smaller symmetric space
 $(\bbP^\bbw \cap \bbL') \bsl \bbL' / \bbL'_{\bbK}$ 
is just a usual Bruhat decomposition.
So we can take a representative $ v $ of $ \kgbw{\bbw} $ from the Weyl group $ W_{J_3} $, 
and we further identify it with an element in $ N_T(L_3) \subset L_3 $.  
Note that, in the previous sections, 
$ \kgbw{\bbw} $ is considered as a subset of $ L_3 \times L_3 $.  
Through the bijection $ (\star) $ above, we get a correspondence between them.
\begin{equation*}
\bigl((P_1^{w_1}\cap{L_3}) \times (P_2^{w_2}\cap{L_3})\bigr)
 \cdot (v, e) \cdot \diag(L_3) 
\longleftrightarrow 
(P_1^{w_1}\cap{L_3}) \cdot v \cdot  (P_2^{w_2}\cap{L_3})
\quad (v \in W_{J_3}) .
\end{equation*}
Put $ \bbv = (v, e) \in L_3 \times L_3 $.  

Let us continue reinterpreting the notations:
\begin{align*}
\bbP^{\bbw\bbv}\cap \bbU' &=
\bbv^{-1} \bbw^{-1} \bbP \bbw \bbv \cap \bbU' 
= (v^{-1} w_1^{-1} P_1 w_1 v \cap U_3) \times (w_2^{-1} P_2 w_2 \cap U_3), 
\\
\bbP^{\bbw\bbv}\cap \bbL'_{\bbK} &=
\bbv^{-1} (\bbP^{\bbw}\cap \bbL') \bbv \cap \bbL'_{\bbK},  
\\
&= \diag(L_3) \cap \bigl( (v^{-1} w_1^{-1} P_1 w_1 v \cap L_3) \times (w_2^{-1} P_2 w_2 \cap L_3) \bigr),
\\
&= \diag (P_1^{w_1v}\cap P_2^{w_2}\cap L_3)
 \supset \diag(T).
\intertext{We therefore get}
(\bbP^{\bbw \bbv} \cap \bbU') \bsl \bbU' / \bbU'_{\bbK} 
&= ((P_1^{w_1v}\cap{U_3}) \times (P_2^{w_2}\cap U_3))
 \bsl (U_3 \times U_3) / \diag(U_3) \\
&\simeq (P_1^{w_1 v}\cap U_3) \bsl U_3 /
 (P_2^{w_2}\cap U_3).  
\end{align*}
On this last double coset space, 
the group 
$ P_1^{w_1 v} \cap P_2^{w_2} \cap L_3
\simeq 
\bbP^{\bbw \bbv}\cap{\bbL'_{\bbK}} 
$ 
acts by conjugation.

\begin{theorem}
\label{theorem:triple.classify}
Let $ P_1, P_2, P_3 $ be three parabolic subgroups of $ G $.  
Then the diagonal $ G $-orbits on the triple flag variety can be described as 
\begin{equation*}
G \bigm\backslash (G/P_1 \times G/P_2 \times G/P_3) 
\simeq \coprod_{\Ptilde{1}, \Ptilde{2}} 
\bigl( (\Ptilde{1}\cap {U_3}) \bsl U_3 / (\Ptilde{2} \cap {U_3}) \bigr) \bigm/ 
(\Ptilde{1} \cap \Ptilde{2} \cap L_3) ,
\end{equation*}
where 
$ (\Ptilde{1}, \Ptilde{2}) $ runs over all pairs 
$ (P_1^{w_1 v}, P_2^{w_2}) $
 for $w_1\in \jwj{J_1}{J_3}$,
 $w_2\in \jwj{J_2}{J_3}$, and $v\in \kgbw{( w_1, w_2 )}$.
\end{theorem}

Let us consider the special case where $ P_3 = B $ is a Borel subgroup.  
In this case $ \Ptilde{1} \cap \Ptilde{2} \cap L_3 = T $ is a maximal torus
 of $G$ and we have 
\begin{equation*}
G \bigm\backslash (G/P_1 \times G/P_2 \times G/B) 
\simeq \coprod_{\Ptilde{1}, \Ptilde{2}} 
\bigl( (\Ptilde{1}\cap{U_B}) \bsl U_B / (\Ptilde{2}\cap{U_B}) \bigr) \bigm/ 
T .
\end{equation*}

Also by applying Theorem~\ref{thm:finiteness.iff.MFaction}
 to the present case, we conclude that: 

\begin{corollary}
A triple flag variety $ G/P_1 \times G/P_2 \times G/B $ is of finite type 
if and only if 
$ \lie{u}_{J_1}\cap \lie{u}_{J_2} $ 
is  $L_{J_1 \cap J_2}$-spherical. 
\end{corollary}

\begin{remark}
This corollary can be deduced from \cite[Theorem~3]{Panyushev.1993}.
\end{remark}

The triple flag varieties $ G/P_1 \times G/P_2 \times G/B $
 of finite type
 (or equivalently the $G$-spherical double flag varieties
 $G/P_1 \times G/P_2$)
 were classified by Stembridge~\cite{Stembridge.2003}
 in connection with multiplicity-free tensor product
 of two irreducible $G$-modules.

\begin{theorem}[\cite{Stembridge.2003}]
\label{theorem:list.triple}
Let $G$ be a connected simple algebraic group.
Let $P_1, P_2$ be parabolic subgroups of $G$
 corresponding to
 sets of simple roots $J_1, J_2\subsetneq \Pi$, respectively.
Then the triple flag variety $ G/P_1 \times G/P_2 \times G/B $
 is of finite type 
 if and only if the pair
 $(\Pi\setminus J_1, \Pi\setminus J_2)$
 appears in Table~\ref{table:group.spherical}
 up to switching $J_1$ and $J_2$.
\end{theorem}

If we assume further that $P_2$ is a Borel subgroup,
 we see from Table~\ref{table:group.spherical} that:

\begin{corollary}
\label{corollary:PBB}
Let $G$ be a connected simple algebraic group
 and $P$ a parabolic subgroup.
Then the triple flag variety
 $G/P \times G/B \times G/B$ is of finite type
 if and only if $G=SL_n$ and $G/P$
 is isomorphic to the projective space of dimension $n-1$
 $($i.e.\ $P$ is a mirabolic subgroup of $G)$.
\end{corollary}

\begin{remark}
Recently Tanaka~\cite{Tanaka.2012} proved that $G/P_1 \times G/P_2$ is
 $G$-spherical if and only if the action of a compact real form of $G$ 
 on $G/P_1 \times G/P_2$ is strongly visible.
\end{remark}

\begin{longtable}{c|c}
\hline
\rule[0pt]{0pt}{12pt}
$\lie{g}$ & $(\Pi \setminus J_1,\ \Pi \setminus J_2)$
 up to switching $J_1$ and $J_2$
\\
\hline\hline
\rule[0pt]{0pt}{12pt}
\rule[-5pt]{0pt}{0pt}
$\lie{sl}_{n+1}$ & 
\begin{xy}
\ar@{-} (0,0) *++!D{\alpha_1} *{\circ}="A"; (10,0) *++!D{\alpha_2} 
 *{\circ}="B"
\ar@{-} "B"; (20,0)="C" 
\ar@{.} "C"; (30,0)="D" 
\ar@{-} "D"; (40,0) *++!D{\alpha_{n}} *{\circ}="E"
\end{xy}\\

\rule[0pt]{0pt}{12pt}
 & $(\{\alpha_i\}, \{\alpha_j\}) (\forall i,j)$,
 $(\{\alpha_1, \alpha_i\}, \{\alpha_j\}) (\forall i,j)$,  \\

\rule[0pt]{0pt}{12pt}
 & $(\{\alpha_i, \alpha_n\}, \{\alpha_j\}) (\forall i,j)$, 
$(\{\alpha_i, \alpha_{i+1}\}, \{\alpha_j\}) (\forall i,j)$, \\

\rule[0pt]{0pt}{12pt}
 & $(\{\alpha_i, \alpha_j\}, \{\alpha_2\}) (\forall i,j)$,
   $(\{\alpha_i, \alpha_j\}, \{\alpha_{n-1}\}) (\forall i,j)$, \\

\rule[0pt]{0pt}{12pt}
\rule[-5pt]{0pt}{0pt}
 & $(\{\alpha_1\}, \text{any})$, $(\{\alpha_n\}, \text{any})$ \\

\hline

\rule[0pt]{0pt}{12pt}
\rule[-5pt]{0pt}{0pt}
$\lie{so}_{2n+1}$ & 
\begin{xy}
\ar@{-} (0,0) *++!D{\alpha_1} *{\circ}="A"; (10,0) *++!D{\alpha_2} 
 *{\circ}="B"
\ar@{-} "B"; (20,0)="C" 
\ar@{.} "C"; (30,0)="D" 
\ar@{-} "D"; (40,0) *++!D{\alpha_{n-1}} *{\circ}="E"
\ar@{=>} "E"; (50,0) *++!D{\alpha_n} *{\circ}="F"
\end{xy}\\

\rule[0pt]{0pt}{12pt}
\rule[-5pt]{0pt}{0pt}
 & $(\{\alpha_1\}, \{\alpha_i\}) (\forall i)$,
 $(\{\alpha_n\}, \{\alpha_n\})$ \\

\hline

\rule[0pt]{0pt}{12pt}
\rule[-5pt]{0pt}{0pt}
$\lie{so}_{2n}$  &
\begin{xy}
\ar@{-} (0,0) *++!D{\alpha_1} *{\circ}="A"; (10,0) *++!D{\alpha_2} 
 *{\circ}="B"
\ar@{-} "B"; (20,0)="C" 
\ar@{.} "C"; (30,0)="D" 
\ar@{-} "D"; (40,0) *+!DR{\alpha_{n-2}} *{\circ}="E"
\ar@{-} "E"; (45,8.6)  *+!L{\alpha_{n-1}} *{\circ}
\ar@{-} "E"; (45,-8.6)  *+!L{\alpha_n} *{\circ}
\end{xy}\\

\rule[0pt]{0pt}{12pt}
 $n\geq 4$
 & $(\{\alpha_1\}, \{\alpha_i\}) (\forall i)$,
 $(\{\alpha_i\}, \{\alpha_j\}) (i=1,2,3,\ j=n-1,n)$, \\

\rule[0pt]{0pt}{12pt}
 & $(\{\alpha_{n-1}\}, \{\alpha_{n-1}\})$,
 $(\{\alpha_{n-1}\}, \{\alpha_n\})$,
 $(\{\alpha_{n}\}, \{\alpha_n\})$, \\

\rule[0pt]{0pt}{12pt}
 & $(\{\alpha_1\}, \{\alpha_i,\alpha_{n-1}\}) (\forall i)$,
 $(\{\alpha_1\}, \{\alpha_i,\alpha_n\}) (\forall i)$, \\

\rule[0pt]{0pt}{12pt}
 & $(\{\alpha_{n-1}\}, \{\alpha_1,\alpha_2\})$,
 $(\{\alpha_n\}, \{\alpha_1,\alpha_2\})$,\\

\rule[0pt]{0pt}{12pt}
 & $(\{\alpha_i\}, \{\alpha_j,\alpha_k\})
 (i=n-1,n,\ \{j,k\}\subset\{1,n-1,n\})$,\\

\rule[0pt]{0pt}{12pt}
\rule[-5pt]{0pt}{0pt}
 & $(\{\alpha_3\}, \{\alpha_2,\alpha_4\})$  if $n=4$,\ 
 $(\{\alpha_4\}, \{\alpha_2,\alpha_3\})$ if $n=4$\\

\hline

\rule[0pt]{0pt}{12pt}
\rule[-5pt]{0pt}{0pt}
$\lie{sp}_n$ &
\begin{xy}
\ar@{-} (0,0) *++!D{\alpha_1} *{\circ}="A"; (10,0) *++!D{\alpha_2} 
 *{\circ}="B"
\ar@{-} "B"; (20,0)="C" 
\ar@{.} "C"; (30,0)="D" 
\ar@{-} "D"; (40,0) *++!D{\alpha_{n-1}} *{\circ}="E"
\ar@{<=} "E"; (50,0) *++!D{\alpha_n} *{\circ}="F"
\end{xy}\\

\rule[0pt]{0pt}{12pt}
\rule[-5pt]{0pt}{0pt}
 & $(\{\alpha_1\}, \{\alpha_i\}) (\forall i)$,
 $(\{\alpha_n\}, \{\alpha_n\})$ \\

\hline

\rule[0pt]{0pt}{12pt}
\rule[-5pt]{0pt}{0pt}
$\lie{e}_6$ &
\begin{xy}
\ar@{-} (0,0) *++!D{\alpha_1} *{\circ}="A"; (10,0) *++!D{\alpha_3} 
 *{\circ}="B"
\ar@{-} "B"; (20,0)*++!U{\alpha_4}  *{\circ}="C"
\ar@{-} "C"; (30,0) *++!D{\alpha_5}  *{\circ}="D"
\ar@{-} "C"; (20,10) *++!D{\alpha_2}  *{\circ}="G"
\ar@{-} "D"; (40,0) *++!D{\alpha_6}  *{\circ}="E"
\end{xy}\\

\rule[0pt]{0pt}{12pt}
 & $(\{\alpha_i\}, \{\alpha_j\}) (i=1,6,\ j\neq 4)$,

\rule[0pt]{0pt}{12pt}
\rule[-5pt]{0pt}{0pt}
 $(\{\alpha_1\}, \{\alpha_1,\alpha_6\})$,
 $(\{\alpha_6\}, \{\alpha_1,\alpha_6\})$ \\

\hline

\rule[0pt]{0pt}{12pt}
\rule[-5pt]{0pt}{0pt}
$\lie{e}_7$ &
\begin{xy}
\ar@{-} (0,0) *++!D{\alpha_1} *{\circ}="A"; (10,0) *++!D{\alpha_3} 
 *{\circ}="B"
\ar@{-} "B"; (20,0)*++!U{\alpha_4}  *{\circ}="C"
\ar@{-} "C"; (30,0) *++!D{\alpha_5}  *{\circ}="D"
\ar@{-} "C"; (20,10) *++!D{\alpha_2}  *{\circ}="G"
\ar@{-} "D"; (40,0) *++!D{\alpha_6}  *{\circ}="E"
\ar@{-} "E"; (50,0) *++!D{\alpha_7}  *{\circ}="F"
\end{xy}\\

\rule[0pt]{0pt}{12pt}
\rule[-5pt]{0pt}{0pt}
& $(\{\alpha_7\}, \{\alpha_1\})$,
 $(\{\alpha_7\}, \{\alpha_2\})$,
 $(\{\alpha_7\}, \{\alpha_7\})$ \\
\hline
\caption{$G/P_1 \times G/P_2 \times G/B$ of finite type}
\label{table:group.spherical}
\vphantom{$\Biggm|$}
\end{longtable}


\section{Classification of $K$-spherical $G/P$}
\label{section:k.spherical}

In this section, we give a classification of the triples
 $(G, K, P)$ such that $G/P\times K/B_K$ is of finite type, where $B_K$
 is a Borel subgroup of $K$.
It is known that any symmetric pair $(\bbG, \bbK)$ with
 $\bbG$ connected, simply connected, and semisimple
 is a direct product of symmetric pairs $(G,K)$ such that
\begin{itemize}
\item $G$ is simple, or
\item $G=G'\times G'$, $K=\diag{G'}$, and $G'$ is simple.
\end{itemize}
For the latter case, 
 $G/P\times K/B_K$ can be written
 as the triple flag variety $G'/P'_1\times G'/P'_2 \times G'/B_{G'}$
 and the classification was already given
 (see Theorem~\ref{theorem:list.triple} and \cite{Stembridge.2003}).

In the rest of this section we assume that $G$ is simple.

\medskip
We first consider the case
 where $(G,K)$ is the complexification of a Hermitian symmetric pair,
 or equivalently the center of $K$ is one-dimensional.
In this case, $K$ equals the Levi component of
 a maximal parabolic subgroup of $G$.
Therefore we can choose a $\theta$-stable Borel subgroup $B$
 of $G$ and a simple root $\alpha_i\in\Pi$ such that
$K=L_{\Pi\setminus\{\alpha_i\}}$.
Then $K=P_{\Pi\setminus\{\alpha_i\}}\cap
 P_{\Pi\setminus\{\alpha_i\}}^-$,
 where $P_{\Pi\setminus\{\alpha_i\}}^-$ is the opposite parabolic subgroup
 of $P_{\Pi\setminus\{\alpha_i\}}$.

\begin{lemma}
\label{lemma:double.triple} 
Suppose that
 $(G,K)$ is the complexification of a Hermitian symmetric pair
 and $P$ is a parabolic subgroup of $G$.
Choose a $\theta$-stable Borel subgroup $B$
 and a simple root $\alpha_i$ such that
 $K=L_{\Pi\setminus\{\alpha_i\}}$.
Then $G/P \times K/B_K$ is of finite type if and only if
 $G/P \times G/P_{\opposite{\Pi\setminus\{\alpha_i\}}} \times G/B$
 is of finite type.
Here
 $\opposite{\Pi\setminus\{\alpha_i\}}:=-w_0(\Pi\setminus\{\alpha_i\})$
for the longest element $w_0 \in W$.
\end{lemma}

\begin{proof}
We follow an argument of \cite[Theorem 2]{NO.2011}.  
The opposite parabolic subgroup $P_{\Pi\setminus\{\alpha_i\}}^-$
 of $P_{\Pi\setminus\{\alpha_i\}}$ is conjugate to
 $P_{\opposite{\Pi\setminus\{\alpha_i\}}}$.
Hence
 $G/P \times G/P_{\opposite{\Pi\setminus\{\alpha_i\}}} \times G/B$
 is of finite type if and only if 
 $G/P \times G/P_{\Pi\setminus\{\alpha_i\}}^-$
 is $G$-spherical.
Since $BP_{\Pi\setminus\{\alpha_i\}}^-$ is open in $G$ and since
$B\cap P_{\Pi\setminus\{\alpha_i\}}^-
 =B\cap L_{\Pi\setminus\{\alpha_i\}} =B_K$,
a natural morphism $G/B_K \to G/B\times G/P_{\Pi\setminus\{\alpha_i\}}^-$
 is an open immersion.
Hence the conditions above are also equivalent to that 
$P$ has an open orbit in $G/B_K$.
\end{proof}

By Theorem~\ref{theorem:list.triple} and Lemma~\ref{lemma:double.triple}, 
 we get a list of $G/P \times K/B_K$ of finite type 
if $ (G, K) $ is a Hermitian symmetric pair.

\medskip

For the remaining pairs $(G, K)$, we use Theorem~\ref{thm:finiteness.iff.MFaction} 
and the classification of spherical linear actions by Benson and Ratcliff \cite{Benson.Ratcliff.1996}.
We carry out a classification according to the following procedure.
\begin{enumerate}
\item
For each symmetric pair $(G, K)$
 we fix a $\theta$-stable Borel subgroup $B$ of $G$
 and a $\theta$-stable Cartan subgroup $T$ in $B$.
\item
Take a standard parabolic subgroup $P$
 and determine
 $\lie{l} \cap \lie{k}$ and $\lie{u}\cap \lie{g}^{-\theta}$.
\item
Check whether the $(L \cap K)$-action on 
 $\lie{u}\cap \lie{g}^{-\theta}$ is spherical
 using the list of \cite{Benson.Ratcliff.1996}.
\end{enumerate}
In addition, the obvious dimension condition
 $\dim G/P + \dim K/B_K \leq \dim K$
 is helpful in some cases.
This is equivalent to $\dim L \geq \dim G-\dim K-\rank K$.
We note that the choice of 
 a $\theta$-stable Borel subgroup $B$ is not unique up to $K$-conjugacy
 in general and the Lie algebras
 $\lie{l} \cap \lie{k}$ and $\lie{u}\cap \lie{g}^{-\theta}$
 depend on this choice.
For our purpose, it is enough to check
 Theorem~\ref{thm:finiteness.iff.MFaction} 
\eqref{item:finiteness.iff.MFaction.BK}
for one choice of $B$.

Let us explain how the procedure above will be done
 in more detail.

For step (1) we describe $B$ and $T$ by using a Vogan diagram
 (see \cite[\S~VI.8]{Knapp.2002} for details).
The Vogan diagram is defined by incorporating the involution
 in the Dynkin diagram.
If $B$ is $\theta$-stable, then $\theta$ permutes simple roots.
The simple roots in the same $\theta$-orbit are connected by arrows.
If a simple root $\alpha$ is fixed by $\theta$, then it is painted or not 
 according to $\lie{g}_\alpha\subset\lie{g}^{-\theta}$ 
 or $\lie{g}_\alpha\subset\lie{k}=\lie{g}^{\theta}$.
We label the simple roots for $\lie{g}$ and $\lie{k}$ as
 $\Pi=\{\alpha_1,\alpha_2,\dots\}$ and $\Pi_K=\{\beta_1,\beta_2,\dots\}$,
 respectively.
Then a simple root $\beta_i$ for $\lie{k}$ is
 written as a restriction $\alpha|_{\lie{t}^{\theta}}$
 for some $\alpha\in\Delta^+=\Delta^+(\lie{g},\lie{t})$.
We note that if $\alpha,\alpha'\in\Delta^+$ and 
 $\alpha|_{\lie{t}^{\theta}}=\alpha'|_{\lie{t}^{\theta}}$,
 then $\alpha=\alpha'$ or $\theta\alpha=\alpha'$ holds.

For step (2) take a standard parabolic subgroup $P=P_J$
 with $J\subset \Pi$.
Put
\[
J_K=\{\beta_i:\exists \alpha\in\Delta^+\text{\ such that\ }
 \alpha|_{\lie{t}^\theta}=\beta_i \text{\ and\ } 
 \alpha,\theta\alpha\in\Delta(\lie{l},\lie{t})\}
\]
so that $L_J\cap K=L_{K,J_K}$, which is a Levi subgroup of $K$
 corresponding to $J_K\subset \Pi_K$.
We then describe the $(L\cap K)$-module
 $\lie{u}\cap \lie{g}^{-\theta}$
 by giving the highest weights of its irreducible constituents.
For $\alpha\in\Delta^+$, the restriction $\alpha|_{\lie{t}^\theta}$
 is a weight in $\lie{u}\cap\lie{g}^{-\theta}$ if and only if 
 one of the following two conditions holds:
\begin{itemize}
\item $\alpha=\theta\alpha\in\Delta(\lie{u},\lie{t})$
 and $\lie{g}_\alpha\subset \lie{g}^{-\theta}$,
\item $\alpha\neq \theta\alpha$ and
 $\alpha,\theta\alpha\in\Delta(\lie{u},\lie{t})$.
\end{itemize}
We write $\Lambda^+(\lie{u}\cap\lie{g}^{-\theta})$
 for the set of highest weights of irreducible constituents
 of the $(L\cap K)$-module $\lie{u}\cap\lie{g}^{-\theta}$.
Denote by $\omega_i\in(\lie{t}^\theta)^*$ the fundamental
 weight corresponding to $\beta_i$.
Then $\Lambda^+(\lie{u}\cap\lie{g}^{-\theta})$
 can be given in terms of $\omega_i$
 and step (3) can be carried out.

We give computations for each case in the following. 
We abbreviate $\alpha_k+\alpha_{k+1}+\cdots+\alpha_l$
 to $\alpha_{[k,l]}$.

\subsection{$(\lie{sl}_{n}, \lie{so}_{n})$}
\ 

Let $(\lie{g},\lie{k})=(\lie{sl}_{n},\lie{so}_{n})$.
We assume $n$ is even and put $m=\frac{n}{2}$.
The case where $n$ is odd can be treated similarly.
We fix a $\theta$-stable Borel subgroup $B$,
 a $\theta$-stable Cartan subgroup $T$
 and a labeling $\alpha_1,\alpha_2,\dots,\alpha_{n-1}$
 of simple roots corresponding to the following Vogan diagram.
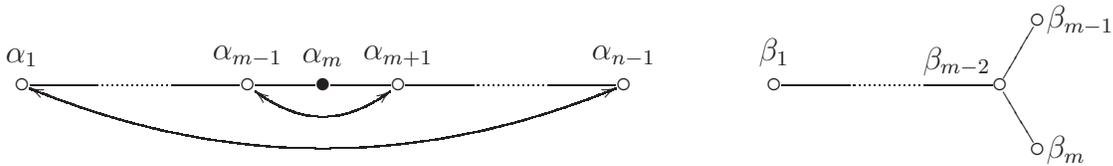
\begin{align*}
\begin{xy}
\ar@{-} (0,0) *++!D{\alpha_1} *{\circ}="A"; (10,0) 
\ar@{.} (10,0) ; (20,0) 
\ar@{-} (20,0) ; (30,0) *++!D{\alpha_{m-1}} *{\circ}="B"
\ar@{-} "B" ; (40,0) *++!D{\alpha_m} *{\bullet}="C"
\ar@{-} "C"; (50,0)*++!D{\alpha_{m+1}}  *{\circ}="D"
\ar@{-} "D"; (60,0) 
\ar@{.} (60,0); (70,0) 
\ar@{-} (70,0); (80,0) *++!D{\alpha_{n-1}}  *{\circ}="E"
\ar@/_1.0pc/@{<->} "B"; "D"
\ar@/_2.0pc/@{<->} "A"; "E"
\ar@{-} (100,0) *++!D{\beta_1} *{\circ}="F"; (110,0)="G" 
\ar@{.} "G"; (120,0)="H" 
\ar@{-} "H"; (130,0) *+!DR{\beta_{m-2}} *{\circ}="I"
\ar@{-} "I"; (135,8.6)  *+!L{\beta_{m-1}} *{\circ}
\ar@{-} "I"; (135,-8.6)  *+!L{\beta_m} *{\circ}
\end{xy}
\end{align*}
Let $\beta_i:=\alpha_i|_{\lie{t}^\theta}$ for $1\leq i\leq m-1$
 and $\beta_m:=(\alpha_{m-1}+\alpha_m)|_{\lie{t}^\theta}$.
Then $\Pi_K=\{\beta_1,\dots,\beta_m\}$ is a set of simple roots 
 for $K$ and the corresponding Dynkin diagram is given as above.

Suppose first that $J=\Pi\setminus\{\alpha_i\}$ with $1\leq i\leq m$.
Then $J_K=\Pi_K\setminus\{\beta_i\}$ if $i\neq m-1$ and
 $J_K=\Pi_K\setminus\{\beta_{m-1},\beta_m\}$ if $i=m-1$. 
We have 
 $\lie{l}\cap\lie{k}=\lie{l}_{K,J_K}
 \simeq \lie{gl}_i\oplus\lie{so}_{n-2i}$,
\[\Delta(\lie{u}\cap\lie{g}^{-\theta},\lie{t}^\theta)
=\{\alpha_{[k,l]}|_{\lie{t}^\theta}:
 k\leq i \text{\ and\ } n-i\leq l\},\quad 
\Lambda^+(\lie{u}\cap\lie{g}^{-\theta})
=\{\alpha_{[1,n-1]}|_{\lie{t}^\theta}\}=\{2\omega_1\}.\]
Hence $\lie{u}\cap\lie{g}^{-\theta} \simeq  S^2(\bbC^i)$,
 on which $L\cap K$ acts spherically.
The case $i>m$ is similar.

Suppose next that $J=\Pi\setminus\{\alpha_i,\alpha_j\}$
 with $1\leq i<j\leq n$.
We may assume that $i<m$ and $i\leq j'$, where $j':=\min\{j,n-j\}$.
Then $J_K=\Pi_K\setminus\{\beta_i,\beta_{j'}\}$ and
 $\lie{l}\cap\lie{k}=\lie{l}_{K,J_K}
 \simeq \lie{gl}_i\oplus\lie{gl}_{j'-i}\oplus\lie{so}_{n-2j'}$.
We have 
$\Delta(\lie{u}\cap\lie{g}^{-\theta},\lie{t}^\theta)=
\{\alpha_{[k,l]}|_{\lie{t}^\theta}:k\leq j' \text{\ and\ } n-j'\leq l\}$
if $j\leq m$ and 
$\Delta(\lie{u}\cap\lie{g}^{-\theta},\lie{t}^\theta)=
\{\alpha_{[k,l]}|_{\lie{t}^\theta}:k\leq j' \text{\ and\ } n-j'\leq l\}
\cup\{\alpha_{[k,l]}|_{\lie{t}^\theta}: k\leq i\text{\ and\ } j'\leq l\}$
if $j>m$.
Hence 
\begin{align*}
&\Lambda^+(\lie{u}\cap\lie{g}^{-\theta})=
\{\alpha_{[1,n-1]}|_{\lie{t}^{\theta}},\alpha_{[1,n-i-1]}|_{\lie{t}^{\theta}},
 \alpha_{[i+1,n-i-1]}|_{\lie{t}^{\theta}}\}
\\
&\qquad\qquad
=\{2\omega_1, \omega_1-\omega_i+\omega_{i+1},-2\omega_i+2\omega_{i+1}\},
\\
&\lie{u}\cap\lie{g}^{-\theta}\simeq S^2(\bbC^i)\oplus(\bbC^i\otimes \bbC^{j'-i})
 \oplus S^2(\bbC^{j'-i})
\end{align*}
if $j\leq m$,
\begin{align*}
&\Lambda^+(\lie{u}\cap\lie{g}^{-\theta})=
\{\alpha_{[1,n-1]}|_{\lie{t}^{\theta}},\alpha_{[1,n-i-1]}|_{\lie{t}^{\theta}},
 \alpha_{[i+1,n-i-1]}|_{\lie{t}^{\theta}},
 \alpha_{[1,j-1]}|_{\lie{t}^{\theta}}\}\\
&\qquad\qquad
 =\{2\omega_1, \omega_1-\omega_i+\omega_{i+1},-2\omega_i+2\omega_{i+1},
 \omega_1-\omega_{j'}+\omega_{j'+1}\},\\
&\lie{u}\cap\lie{g}^{-\theta}\simeq S^2(\bbC^i)\oplus(\bbC^i\otimes \bbC^{j'-i})
 \oplus S^2(\bbC^{j'-i})\oplus (\bbC^i\otimes \bbC^{n-2j'})
\end{align*}
if $m<j<n-i$, and 
\begin{align*}
&\Lambda^+(\lie{u}\cap\lie{g}^{-\theta})=
\{\alpha_{[1,n-1]}|_{\lie{t}^{\theta}},\alpha_{[1,j-1]}|_{\lie{t}^{\theta}}\}
 =\{2\omega_1, \omega_1-\omega_{j'}+\omega_{j'+1}\},\\
&\lie{u}\cap\lie{g}^{-\theta}\simeq S^2(\bbC^i)\oplus (\bbC^i\otimes \bbC^{n-2j'})
\end{align*}
if $j=n-i$.
We can see that none of them are $(L\cap K)$-spherical.

We therefore conclude that
 $G/P_J\times K/{B_K}$ is of finite type if and only if
 $|\Pi\setminus J|= 1$.

\subsection{$(\lie{sl}_{2n}, \lie{sp}_{n})$}\label{slsp}
\

Let $(\lie{g},\lie{k})=(\lie{sl}_{2n},\lie{sp}_{n})$.
We fix $B$, $T$ and simple roots $\alpha_1,\dots,\alpha_{2n-1}$
 corresponding to the following Vogan diagram.
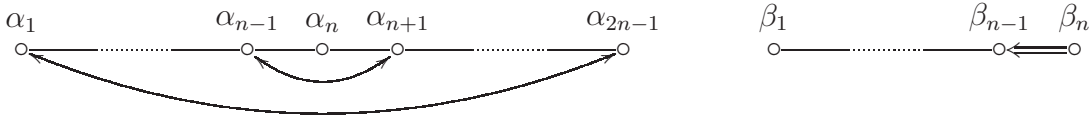
\begin{align*}
\begin{xy}
\ar@{-} (0,0) *++!D{\alpha_1} *{\circ}="A"; (10,0) 
\ar@{.} (10,0) ; (20,0) 
\ar@{-} (20,0) ; (30,0) *++!D{\alpha_{n-1}} *{\circ}="B"
\ar@{-} "B" ; (40,0) *++!D{\alpha_n} *{\circ}="C"
\ar@{-} "C"; (50,0)*++!D{\alpha_{n+1}}  *{\circ}="D"
\ar@{-} "D"; (60,0) 
\ar@{.} (60,0); (70,0) 
\ar@{-} (70,0); (80,0) *++!D{\alpha_{2n-1}}  *{\circ}="E"
\ar@/_1.0pc/@{<->} "B"; "D"
\ar@/_2.0pc/@{<->} "A"; "E"
\ar@{-} (100,0) *++!D{\beta_1} *{\circ}="F"; (110,0)="G" 
\ar@{.} "G"; (120,0)="H" 
\ar@{-} "H"; (130,0) *++!D{\beta_{n-1}} *{\circ}="I"
\ar@{<=} "I"; (140,0)  *++!D{\beta_n} *{\circ}
\end{xy}
\end{align*}
Let $\beta_i:=\alpha_i|_{\lie{t}^\theta}$ for $1\leq i\leq n$.
Then $\Pi_K=\{\beta_1,\dots,\beta_n\}$ is a set of simple roots 
 for $K$ and the corresponding Dynkin diagram is given as above.

Suppose that $J=\Pi\setminus\{\alpha_i,\alpha_j\}$
 with $1\leq i<j\leq 2n-1$.
We assume that $i<n$ and $i\leq j'$, where $j':=\min\{j,2n-j\}$.
Then $J_K=\Pi_K\setminus\{\beta_i,\beta_{j'}\}$ and
 $\lie{l}\cap\lie{k}=\lie{l}_{K,J_K}
 \simeq \lie{gl}_i\oplus\lie{gl}_{j'-i}\oplus\lie{sp}_{n-j'}$.
We have 
$\Delta(\lie{u}\cap\lie{g}^{-\theta},\lie{t}^\theta)=
\{\alpha_{[k,l]}|_{\lie{t}^\theta}:k+l<2n\text{\ and\ } 2n-j'\leq l\}$
if $j\leq n$ and 
$\Delta(\lie{u}\cap\lie{g}^{-\theta},\lie{t}^\theta)=
\{\alpha_{[k,l]}|_{\lie{t}^\theta}:k+l<2n\text{\ and\ } 2n-j'\leq l\}
\cup\{\alpha_{[k,l]}|_{\lie{t}^\theta}:
 k+l< 2n,\ k\leq i\text{\ and\ } j'\leq l\}$
if $j>n$.
As in the case $(\lie{sl}_{n},\lie{so}_{n})$ above, we have 
\begin{align*}
\lie{u}\cap\lie{g}^{-\theta}
 \simeq
\begin{cases}
\bigwedge^2\bbC^i\oplus \bigwedge^2\bbC^{j'-i}
 \oplus(\bbC^i\otimes \bbC^{j'-i})
\quad
\text{\ if $j\leq n$},\\
\bigwedge^2\bbC^i\oplus \bigwedge^2\bbC^{j'-i}
 \oplus(\bbC^i\otimes \bbC^{j'-i})
 \oplus (\bbC^i\otimes \bbC^{2n-2j'})
\quad
\text{\ if $n<j<2n-i$}, \\
\bigwedge^2\bbC^i \oplus(\bbC^i\otimes \bbC^{2n-2j'})
\quad
\text{\ if $j=2n-i$}.
\end{cases}
\end{align*}
This is $(L\cap K)$-spherical if and only if $i=1$ or $i+1=j$.

Suppose that $J=\Pi\setminus\{\alpha_1,\alpha_i,\alpha_{i+1}\}$
 with $2\leq i\leq 2n-3$.
Put $i':=\min\{i,2n-i-1\}$. 
Then $J_K=\Pi_K\setminus\{\beta_1,\beta_{i'},\beta_{i'+1}\}$ and
 $\lie{l}\cap\lie{k}=\lie{l}_{K,J_K}
 \simeq \lie{gl}_1\oplus\lie{gl}_{i'-1}\oplus\lie{gl}_1
 \oplus\lie{sp}_{n-i'-1}$.
We have 
$\Delta(\lie{u}\cap\lie{g}^{-\theta},\lie{t}^\theta)=
\{\alpha_{[k,l]}|_{\lie{t}^\theta}:k+l<2n\text{\ and\ } 2n-i'-1\leq l\}$
if $i\leq n-1$ and 
$\Delta(\lie{u}\cap\lie{g}^{-\theta},\lie{t}^\theta)=
\{\alpha_{[k,l]}|_{\lie{t}^\theta}:k+l<2n\text{\ and\ } 2n-i'-1\leq l\}
\cup\{\alpha_{[1,l]}|_{\lie{t}^\theta}: i'\leq l< 2n-1\}$
if $i\geq n$.
Then 
\begin{align*}
&\Lambda^+(\lie{u}\cap\lie{g}^{-\theta})
=\{\alpha_{[1,2n-2]}|_{\lie{t}^{\theta}},
 \alpha_{[1,2n-3]}|_{\lie{t}^{\theta}}, 
 \alpha_{[2,2n-3]}|_{\lie{t}^{\theta}} \}
=\{\omega_2,\omega_1-\omega_2+\omega_3,-\omega_1+\omega_3\},\\
\quad 
&\lie{u}\cap\lie{g}^{-\theta}
\simeq\bbC\oplus \bbC \oplus\bbC
\end{align*}
if $i=2$,
\begin{align*}
&\Lambda^+(\lie{u}\cap\lie{g}^{-\theta})
\supset\{\alpha_{[1,2n-2]}|_{\lie{t}^{\theta}},
 \alpha_{[1,2n-i'-1]}|_{\lie{t}^{\theta}}, 
 \alpha_{[2,2n-3]}|_{\lie{t}^{\theta}},
 \alpha_{[2,2n-i'-1]}|_{\lie{t}^{\theta}} \} \\
&\qquad\qquad\quad
=\{\omega_2,\omega_1-\omega_{i'}+\omega_{i'+1},-\omega_1+\omega_3,
 -\omega_1+\omega_2-\omega_{i'}+\omega_{i'+1}\},\\
\quad 
&\lie{u}\cap\lie{g}^{-\theta}
\supset \bbC^{i'-1}\oplus \bbC\oplus\bigwedge^2\bbC^{i'-1}\oplus\bbC^{i'-1}
\end{align*}
if $3\leq i\leq 2n-4$, and
\begin{align*}
&\Lambda^+(\lie{u}\cap\lie{g}^{-\theta})
\supset\{\alpha_{[1,2n-2]}|_{\lie{t}^{\theta}},
 \alpha_{[2,2n-3]}|_{\lie{t}^{\theta}},
 \alpha_{[1,2]}|_{\lie{t}^{\theta}}\}
=\{\omega_2,-\omega_1+\omega_3, \omega_1+\omega_2-\omega_3\},\\
\quad 
&\lie{u}\cap\lie{g}^{-\theta}
\supset\bbC\oplus \bbC \oplus\bbC
\end{align*}
if $i=2n-3$.
We can see that $\lie{u}\cap\lie{g}^{-\theta}$ is $(L\cap K)$-spherical
 if and only if $i=2$.

Suppose that $J=\Pi\setminus\{\alpha_1,\alpha_i,\alpha_{2n-1}\}$
 with $2\leq i\leq n$.
Then $J_K=\Pi_K\setminus\{\beta_1,\beta_i\}$ and
 $\lie{l}\cap\lie{k}=\lie{l}_{K,J_K}
 \simeq \lie{gl}_1\oplus\lie{gl}_{i-1}\oplus\lie{sp}_{n-i}$.
We have 
\begin{align*}
&\Lambda^+(\lie{u}\cap\lie{g}^{-\theta})
=\{\alpha_{[1,2n-2]}|_{\lie{t}^{\theta}},
 \alpha_{[1,2n-i-1]}|_{\lie{t}^{\theta}},
 \alpha_{[2,2n-3]}|_{\lie{t}^{\theta}},
 \alpha_{[1,i-1]}|_{\lie{t}^{\theta}}\}\\
&\qquad\qquad
=\{\omega_2,\omega_1-\omega_i+\omega_{i+1},-\omega_1+\omega_3,
 \omega_1+\omega_{i-1}-\omega_i\},\\
&\lie{u}\cap\lie{g}^{-\theta}
\simeq\bbC^{i-1}\oplus \bbC^{2n-2i}\oplus\bigwedge^2\bbC^{i-1}\oplus(\bbC^{i-1})^*
\end{align*}
if $i<n$, 
\begin{align*}
&\Lambda^+(\lie{u}\cap\lie{g}^{-\theta})
=\{\alpha_{[1,2n-2]}|_{\lie{t}^{\theta}},
 \alpha_{[2,2n-3]}|_{\lie{t}^{\theta}},
 \alpha_{[1,n-1]}|_{\lie{t}^{\theta}}\}
=\{\omega_2,-\omega_1+\omega_3,
 \omega_1+\omega_{n-1}-\omega_n\},\\
&\lie{u}\cap\lie{g}^{-\theta}
\simeq\bbC^{n-1}\oplus \bigwedge^2\bbC^{n-1}\oplus(\bbC^{n-1})^*
\end{align*}
if $i=n\geq 3$, and  
\begin{align*}
\Lambda^+(\lie{u}\cap\lie{g}^{-\theta})
=\{\alpha_{[1,2]}|_{\lie{t}^{\theta}}, \alpha_1|_{\lie{t}^{\theta}}\}
=\{\omega_2, 2\omega_1-\omega_2\},\quad 
\lie{u}\cap\lie{g}^{-\theta}
\simeq\bbC\oplus\bbC^*
\end{align*}
if $i=n=2$.
Note that in both cases the factor $ \lie{gl}_1 $ (central torus) acts on 
$ \bbC^{i - 1} \oplus (\bbC^{i - 1})^{\ast} $ by the same scalar, which can be read off from 
the explicit description of $ \Lambda^+(\lie{u}\cap\lie{g}^{-\theta}) $.  
Therefore $ \lie{u}\cap\lie{g}^{-\theta} $ is $(L\cap K)$-spherical if and only if $i=2$.

These observations imply that
 $G/P_J\times K/{B_K}$ is of finite type if and only if
\begin{align*}
\Pi\setminus J
=
&
  \{\alpha_i\} (1\leq i\leq 2n-1), \;\; 
  \{\alpha_1,\alpha_i\} \; (2\leq i\leq 2n-1),
\\
&
 \{\alpha_i,\alpha_{2n-1}\} \; (1\leq i\leq 2n-2), \;\;
 \{\alpha_i,\alpha_{i+1}\} \; (1\leq i\leq 2n-2),
\\
&
 \{\alpha_1,\alpha_2,\alpha_3\}, \;\;
 \{\alpha_{2n-3},\alpha_{2n-2},\alpha_{2n-1}\}, \;\;
 \{\alpha_1,\alpha_2,\alpha_{2n-1}\}, \;\;
 \{\alpha_1,\alpha_{2n-2},\alpha_{2n-1}\}.
\end{align*}

\subsection{$(\lie{so}_{2n+1}, \lie{so}_p\oplus\lie{so}_q)$}\label{soso}
\ 

Let $(\lie{g},\lie{k})=(\lie{so}_{2n+1},\lie{so}_p\oplus\lie{so}_q)$
 with $p+q=2n+1$. 
We may assume that $p$ is even and $q$ is odd.
Put $p'=\frac{p}{2}$ and $q'=\frac{q-1}{2}$.
We fix $B$, $T$ and simple roots $\alpha_1,\dots,\alpha_n$
 corresponding to the following Vogan diagram.
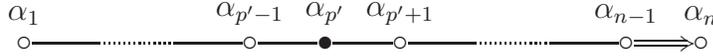
\begin{align*}
\begin{xy}
\ar@{-} (0,0) *++!D{\alpha_1} *{\circ}="A"; (10,0) 
\ar@{.} (10,0) ; (20,0) 
\ar@{-} (20,0) ; (30,0) *++!D{\alpha_{p'-1}} *{\circ}="B"
\ar@{-} "B" ; (40,0) *++!D{\alpha_{p'}} *{\bullet}="C"
\ar@{-} "C"; (50,0)*++!D{\alpha_{p'+1}}  *{\circ}="D"
\ar@{-} "D"; (60,0) 
\ar@{.} (60,0); (70,0) 
\ar@{-} (70,0); (80,0) *++!D{\alpha_{n-1}}  *{\circ}="E"
\ar@{=>} "E"; (90,0) *++!D{\alpha_n} *{\circ}
\end{xy}
\end{align*}
Let $\beta_i:=\alpha_i$ for $i\neq p'$
 and $\beta_{p'}:=\alpha_{p'-1}+2(\alpha_{p'}+\cdots+\alpha_n)$.
Then $\Pi_K=\{\beta_1,\dots,\beta_n\}$ is a set of simple roots 
 for $K$ and the corresponding Dynkin diagram is 
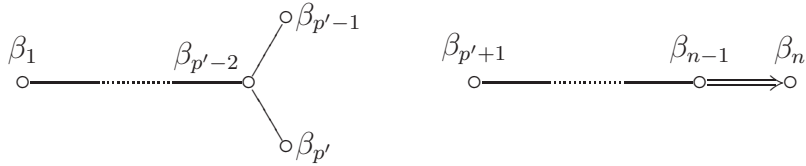
\begin{align*}
\begin{xy}
\ar@{-} (0,0) *++!D{\beta_1} *{\circ}="A"; (10,0)="B" 
\ar@{.} "B"; (20,0) 
\ar@{-} (20,0); (30,0) *+!DR{\beta_{p'-2}} *{\circ}="C"
\ar@{-} "C"; (35,8.6)  *+!L{\beta_{p'-1}} *{\circ}="D"
\ar@{-} "C"; (35,-8.6)  *+!L{\beta_{p'}} *{\circ}
\ar@{-} (60,0) *++!D{\beta_{p'+1}} *{\circ}="F"; (70,0)="G" 
\ar@{.} "G"; (80,0)="H" 
\ar@{-} "H"; (90,0) *++!D{\beta_{n-1}} *{\circ}="I"
\ar@{=>} "I"; (102,0)  *++!D{\beta_{n}} *{\circ}
\end{xy}
\end{align*}

We first consider the case $p,q\geq 3$.
Suppose that $J=\Pi\setminus\{\alpha_i\}$ with $1\leq i \leq p'$.
Then $J_K=\Pi_K\setminus\{\beta_i\}$ if $i\neq p'-1$ and
 $J_K=\Pi_K\setminus\{\beta_{p'-1},\beta_{p'}\}$ if $i=p'-1$.
Hence $\lie{l}\cap\lie{k}=\lie{l}_{K,J_K}
 \simeq \lie{gl}_{i}\oplus\lie{so}_{p-2i}\oplus\lie{so}_{q}$.
We have 
$\Delta(\lie{u}\cap\lie{g}^{-\theta},\lie{t}^\theta)=
\{\alpha_{[k,l]}:k\leq i\text{\ and\ } p'\leq l\}
\cup\{\alpha_{[k,l]}+2\alpha_{[l+1,n]}: k\leq i\text{\ and\ } p'\leq l\}$.
Hence $\lie{u}\cap\lie{g}^{-\theta}\simeq \bbC^i\otimes \bbC^q$,
which is $(L\cap K)$-spherical if and only if $i=1$.

Suppose that $J=\Pi\setminus\{\alpha_i\}$ with $p'< i \leq n$.
Then $J_K=\Pi_K\setminus\{\beta_{p'},\beta_i\}$
 and $\lie{l}\cap\lie{k}=\lie{l}_{K,J_K}
 \simeq \lie{gl}_{p'}\oplus\lie{gl}_{i-p'}\oplus\lie{so}_{2n-2i+1}$.
We have 
$\Delta(\lie{u}\cap\lie{g}^{-\theta},\lie{t}^\theta)=
\{\alpha_{[k,l]}:k\leq p'\text{\ and\ } i\leq l\}
\cup\{\alpha_{[k,l]}+2\alpha_{[l+1,n]}: k\leq p'\leq l\}$.
Hence $\lie{u}\cap\lie{g}^{-\theta}\simeq 
 (\bbC^{p'}\otimes \bbC^{i-p'})\oplus (\bbC^{p'}\otimes \bbC^{2n-2i+1})$,
which is $(L\cap K)$-spherical if and only if $i=n$.

Suppose that $J=\Pi\setminus\{\alpha_1,\alpha_n\}$.
Then $J_K=\Pi_K\setminus\{\beta_1,\beta_{p'},\beta_n\}$
 and $\lie{l}\cap\lie{k}=\lie{l}_{K,J_K}
 \simeq \lie{gl}_{1}\oplus\lie{gl}_{p'-1} \oplus\lie{gl}_{q'}$.
We have
\begin{align*}
&\Delta(\lie{u}\cap\lie{g}^{-\theta},\lie{t}^\theta)=
\{\alpha_{[1,k]}:p'\leq k\}\cup\{\alpha_{[k,n]}:k\leq p'\}
\cup\{\alpha_{[k,l]}+2\alpha_{[l+1,n]}: k\leq p'\leq l\}.
\end{align*}
Hence 
\begin{align*}
&\Lambda^+(\lie{u}\cap\lie{g}^{-\theta})
=\{\alpha_{[1,p']}+2\alpha_{[p'+1,n]},\alpha_{[1,n]},\alpha_{[1,n-1]},
 \alpha_{[2,p']}+2\alpha_{[p'+1,n]},\alpha_{[2,n]}\},\\
&\qquad\qquad=\{\omega_1+\omega_{p'+1},\omega_1,
 \omega_1+\omega_{n-1}-2\omega_n,-\omega_1+\omega_2+\omega_{p'+1},
 -\omega_1+\omega_2\}\\
&\lie{u}\cap\lie{g}^{-\theta}
\simeq \bbC^{q'}\oplus \bbC\oplus (\bbC^{q'})^*
 \oplus(\bbC^{p'-1}\otimes \bbC^{q'})\oplus \bbC^{p'-1},
\end{align*}
which is not $(L\cap K)$-spherical.

Hence $G/P_J\times K/{B_K}$ is of finite type if and only if
 $\Pi\setminus J=\{\alpha_1\}$ or $\{\alpha_n\}$ for $p,q\geq 3$.

Since $p=2$ is a Hermitian case, $q=1$ is the only remaining case.
Let $(\lie{g},\lie{k})=(\lie{so}_{2n+1},\lie{so}_{2n})$
 and $J=\emptyset$.
Then $\lie{l}\cap\lie{k}=\lie{t}^\theta$ and 
$\Delta(\lie{u}\cap\lie{g}^{-\theta},\lie{t}^\theta)
 =\{\alpha_{[i,n]}:1\leq i\leq n-1\}$.
Since the weights in $\lie{u}\cap\lie{g}^{-\theta}$
 are linearly independent, 
 $L\cap K$ acts spherically on $\lie{u}\cap\lie{g}^{-\theta}$.
Hence $G/P_J\times K/{B_K}$ is of finite type for any $J\subset \Pi$.

\subsection{$(\lie{so}_{2n}, \lie{so}_p\oplus\lie{so}_q)$}
\ 

This case can be treated is a way similar to the case
 $(\lie{g},\lie{k})=(\lie{so}_{2n+1},\lie{so}_p\oplus\lie{so}_q)$ above.

\subsection{$(\lie{sp}_{n}, \lie{sp}_p\oplus\lie{sp}_q)$}\label{spsp}
\ 

Let $(\lie{g},\lie{k})=(\lie{sp}_n,\lie{sp}_p\oplus\lie{sp}_q)$ with $p+q=n$. 
We fix $B$, $T$ and simple roots $\alpha_1,\dots,\alpha_n$
 corresponding to the following Vogan diagram.
\begin{align*}
\begin{xy}
\ar@{-} (0,0) *++!D{\alpha_1} *{\circ}="A"; (10,0) 
\ar@{.} (10,0) ; (20,0) 
\ar@{-} (20,0) ; (30,0) *++!D{\alpha_{p-1}} *{\circ}="B"
\ar@{-} "B" ; (40,0) *++!D{\alpha_{p}} *{\bullet}="C"
\ar@{-} "C"; (50,0)*++!D{\alpha_{p+1}}  *{\circ}="D"
\ar@{-} "D"; (60,0) 
\ar@{.} (60,0); (70,0) 
\ar@{-} (70,0); (80,0) *++!D{\alpha_{n-1}}  *{\circ}="E"
\ar@{<=} "E"; (90,0) *++!D{\alpha_n} *{\circ}
\end{xy}
\end{align*}
Let $\beta_i:=\alpha_i$ for $i\neq p$
 and $\beta_p:=2(\alpha_p+\cdots+\alpha_{n-1})+\alpha_n$.
Then $\Pi_K=\{\beta_1,\dots,\beta_n\}$ is a set of simple roots 
 for $K$ and the corresponding Dynkin diagram is 
\begin{align*}
\begin{xy}
\ar@{-} (0,0) *++!D{\beta_1} *{\circ}="A"; (10,0)="B" 
\ar@{.} "B"; (20,0) 
\ar@{-} (20,0); (30,0)  *++!D{\beta_{p-1}} *{\circ}="D"
\ar@{<=} "D"; (40,0)  *++!D{\beta_p} *{\circ}
\ar@{-} (60,0) *++!D{\beta_{p+1}} *{\circ}="F"; (70,0)="G" 
\ar@{.} "G"; (80,0)="H" 
\ar@{-} "H"; (90,0) *++!D{\beta_{n-1}} *{\circ}="I"
\ar@{<=} "I"; (100,0)  *++!D{\beta_n} *{\circ}
\end{xy}
\end{align*}

Suppose that $J=\Pi\setminus\{\alpha_i\}$ with $1\leq i \leq p$.
Then $J_K=\Pi_K\setminus\{\beta_i\}$ and
 $\lie{l}\cap\lie{k}=\lie{l}_{K,J_K}
 \simeq \lie{gl}_{i}\oplus\lie{sp}_{p-i}\oplus\lie{sp}_{q}$.
We have 
$\lie{u}\cap\lie{g}^{-\theta}\simeq \bbC^i\otimes \bbC^{2q}$,
which is $(L\cap K)$-spherical if and only if $i\leq 3$ or $q\leq 2$.

Suppose that $J=\Pi\setminus\{\alpha_i\}$ with $p<i\leq n$.
Then $J_K=\Pi_K\setminus\{\beta_p,\beta_i\}$
 and $\lie{l}\cap\lie{k}=\lie{l}_{K,J_K}
 \simeq \lie{gl}_{p}\oplus\lie{gl}_{i-p}\oplus\lie{sp}_{n-i}$.
We have 
$\lie{u}\cap\lie{g}^{-\theta}
\simeq (\bbC^p\otimes \bbC^{i-p})\oplus (\bbC^{p}\otimes \bbC^{2(n-i)})$,
which is $(L\cap K)$-spherical if and only if
 $i-p=n-i=1$, $p\leq 2$, or $i=n$.

Suppose that $J=\Pi\setminus\{\alpha_i,\alpha_j\}$ with $1\leq i<j\leq p$.
Then $J_K=\Pi_K\setminus\{\beta_i,\beta_j\}$
 and $\lie{l}\cap\lie{k}=\lie{l}_{K,J_K}
 \simeq \lie{gl}_{i}\oplus\lie{gl}_{j-i}\oplus\lie{sp}_{p-j}
 \oplus\lie{sp}_{q}$.
We have 
$\lie{u}\cap\lie{g}^{-\theta}
\simeq (\bbC^i\otimes \bbC^{2q})\oplus (\bbC^{j-i}\otimes \bbC^{2q})$,
which is $(L\cap K)$-spherical if and only if $i=j-i=1$ or $q=1$.

Suppose that $J=\Pi\setminus\{\alpha_i,\alpha_j\}$ with
 $1\leq i\leq p<j\leq n$.
Then $J_K=\Pi_K\setminus\{\beta_i,\beta_p,\beta_j\}$
 and $\lie{l}\cap\lie{k}=\lie{l}_{K,J_K}
 \simeq \lie{gl}_{i}\oplus\lie{gl}_{p-i}\oplus\lie{gl}_{j-p}
 \oplus\lie{sp}_{n-j}$.
We have 
\[
\begin{aligned}
\lie{u}\cap\lie{g}^{-\theta}
\simeq 
&
(\bbC^i\otimes \bbC^{j-p}) 
\oplus (\bbC^i\otimes \bbC^{2(n-j)})
\oplus  (\bbC^i\otimes (\bbC^{j-p})^*) 
\\
&
\quad
\oplus (\bbC^{p-i}\otimes \bbC^{j-p})
\oplus (\bbC^{p-i}\otimes \bbC^{2(n-j)}),
\end{aligned}
\]
which is $(L\cap K)$-spherical if and only if $i=p=1$ or $j=n=p+1$.

Suppose that $J=\Pi\setminus\{\alpha_i,\alpha_j\}$ with $p<i<j\leq n$.
Then $J_K=\Pi_K\setminus\{\beta_p,\beta_i,\beta_j\}$
 and $\lie{l}\cap\lie{k}=\lie{l}_{K,J_K}
 \simeq \lie{gl}_{p}\oplus\lie{gl}_{i-p}\oplus\lie{gl}_{j-i}
 \oplus\lie{sp}_{n-j}$.
We have 
\[\lie{u}\cap\lie{g}^{-\theta}
\simeq (\bbC^p\otimes \bbC^{i-p}) \oplus (\bbC^p\otimes \bbC^{j-i})
\oplus  (\bbC^p\otimes (\bbC^{j-i})^*) \oplus (\bbC^p\otimes \bbC^{2(n-j)})\]
which is $(L\cap K)$-spherical if and only if $p=1$.

From these observations
 we see that for $p,q\geq 3$, 
 $G/P_J\times K/{B_K}$ is of finite type if and only if 
 $\Pi\setminus J=\{\alpha_1\}$, $\{\alpha_2\}$, $\{\alpha_3\}$,
 $\{\alpha_n\}$, or $\{\alpha_1,\alpha_2\}$.
Add to this,
 $\Pi\setminus J=\{\alpha_i\}$ for $1\leq i\leq n$
 are also the cases if $\min\{p,q\}\leq 2$.
For $\min\{p,q\}=1$, $\lie{u}\cap\lie{g}^{-\theta}$ is $(L\cap K)$-spherical
 if $|\Pi\setminus J|=2$.
To prove that these are all the cases we need to check
 the case $p=1$ and $|\Pi\setminus J|=3$.

Suppose that $p=1$ and $J=\Pi\setminus\{\alpha_i,\alpha_j,\alpha_k\}$
 with $1\leq i<j<k\leq n$.
Then $J_K=\Pi_K\setminus\{\beta_1,\beta_i,\beta_j,\beta_k\}$
 and $\lie{l}\cap\lie{k}=\lie{l}_{K,J_K}
 \simeq \lie{gl}_{1}\oplus\lie{gl}_{i-1}\oplus\lie{gl}_{j-i}
 \oplus\lie{gl}_{k-j}\oplus\lie{sp}_{n-j}$.
We have 
\[\lie{u}\cap\lie{g}^{-\theta}
\simeq \bbC^{i-1} \oplus \bbC^{j-i}\oplus \bbC^{k-j} \oplus \bbC^{2(n-k)}
 \oplus (\bbC^{k-j})^* \oplus (\bbC^{j-i})^*,\]
which is not $(L\cap K)$-spherical.

\subsection{$(\lie{g}_{2}, \lie{sl}_2\oplus\lie{sl}_2)$}
\ 

Let $(\lie{g}, \lie{k})=(\lie{g}_2, \lie{sl}_2\oplus\lie{sl}_2)$.
Then the dimension condition is $\dim L\geq \dim G-\dim K-\rank K= 14-6-2=6$.
But this does not hold for a proper parabolic subgroup $P\subset G$.

\subsection{$(\lie{f}_{4}, \lie{sp}_3\oplus\lie{sp}_1)$}
\ 

Let $(\lie{g}, \lie{k})=(\lie{f}_4, \lie{sp}_3\oplus\lie{sp}_1)$.
Then the dimension condition is $\dim L\geq \dim G-\dim K-\rank K= 52-24-4=24$.
But this does not hold for a proper parabolic subgroup $P\subset G$.

\subsection{$(\lie{f}_{4}, \lie{so}_9)$} \label{f4so}
\ 

Let $(\lie{g}, \lie{k})=(\lie{f}_4, \lie{so}_9)$.
We fix $B$, $T$ and simple roots $\alpha_1,\alpha_2,\alpha_3,\alpha_4$
 corresponding to the following Vogan diagram.
\begin{align*}
\begin{xy}
\ar@{-} (0,0) *++!D{\alpha_1} *{\circ}="A";
  (10,0) *++!D{\alpha_2} *{\circ}="B"
\ar@{=>} "B" ; (20,0)  *++!D{\alpha_3} *{\circ}="C"
\ar@{-} "C" ; (30,0) *++!D{\alpha_4} *{\bullet}
\ar@{-} (50,0) *++!D{\beta_1} *{\circ}="E";
  (60,0) *++!D{\beta_2} *{\circ}="F"
\ar@{-} "F" ; (70,0)  *++!D{\beta_3} *{\circ}="G"
\ar@{=>} "G" ; (80,0) *++!D{\beta_4} *{\circ}
\end{xy}
\end{align*}
Let $\beta_1:=\alpha_2+2\alpha_3+2\alpha_4$,
 $\beta_2:=\alpha_1$, $\beta_3:=\alpha_2$, $\beta_4:=\alpha_3$.
Then $\Pi_K=\{\beta_1,\beta_2,\beta_3,\beta_4\}$ is a set of simple roots 
 for $K$ and the corresponding Dynkin diagram is given as above.

The dimension condition is $\dim L\geq \dim G-\dim K-\rank K= 52-36-4=12$.
This holds for $J=\Pi\setminus\{\alpha_i\}$ with $1\leq i\leq 4$
 or $J=\Pi\setminus\{\alpha_1,\alpha_4\}$.

Suppose that $J=\Pi\setminus\{\alpha_1,\alpha_4\}$.
Then $J_K=\Pi_K\setminus\{\beta_1,\beta_2\}$ and
 $\lie{l}\cap\lie{k}=\lie{l}_{K,J_K}
 \simeq \lie{gl}_1\oplus\lie{gl}_1\oplus\lie{so}_{5}
 (\simeq \lie{gl}_1\oplus\lie{gl}_1\oplus\lie{sp}_{2})$.
We have $\Delta(\lie{u}\cap\lie{g}^{-\theta},\lie{t}^{\theta})
=\{\sum_{i=1}^4 m_i\alpha_i\in\Delta^+:m_4=1\}$. Hence
\begin{align*}
&\Lambda^+(\lie{u}\cap\lie{g}^{-\theta})=
\{\alpha_2+2\alpha_3+\alpha_4,\alpha_1+2\alpha_2+3\alpha_3+\alpha_4\}
=\{\omega_1-\omega_2+\omega_4,\omega_4\}
\end{align*}
and $\lie{u}\cap\lie{g}^{-\theta} \simeq \bbC^4\oplus \bbC^4$,
which is $(L\cap K)$-spherical.

Suppose that $J=\Pi\setminus\{\alpha_2\}$.
Then $J_K=\Pi_K\setminus\{\beta_1,\beta_3\}$
 and $\lie{l}\cap\lie{k}=\lie{l}_{K,J_K}
 \simeq \lie{gl}_2\oplus\lie{gl}_2$.
We have  $\Delta(\lie{u}\cap\lie{g}^{-\theta},\lie{t}^{\theta})
=\{\sum_{i=1}^4 m_i\alpha_i\in\Delta^+:m_2>0,\ m_4=1\}$.
Hence
\begin{align*}
&\Lambda^+(\lie{u}\cap\lie{g}^{-\theta})=
\{\alpha_1+\alpha_2+2\alpha_3+\alpha_4,\alpha_1+2\alpha_2+3\alpha_3+\alpha_4\}
=\{\omega_2-\omega_3+\omega_4,\omega_4\}
\end{align*}
and
$\lie{u}\cap\lie{g}^{-\theta} \simeq(\bbC^2\otimes \bbC^2)\oplus \bbC^2$,
which is $(L\cap K)$-spherical.

Suppose that $J=\Pi\setminus\{\alpha_3\}$.
Then $J_K=\Pi_K\setminus\{\beta_1,\beta_4\}$
 and $\lie{l}\cap\lie{k}=\lie{l}_{K,J_K}
 \simeq \lie{gl}_1\oplus\lie{gl}_3$.
We have  $\Delta(\lie{u}\cap\lie{g}^{-\theta},\lie{t}^{\theta})
=\{\sum_{i=1}^4 m_i\alpha_i\in\Delta^+:m_3>0,\ m_4=1\}$.
Hence
\begin{align*}
&\Lambda^+(\lie{u}\cap\lie{g}^{-\theta})=
\{\alpha_1+\alpha_2+\alpha_3+\alpha_4,\alpha_1+2\alpha_2+2\alpha_3+\alpha_4,
\alpha_1+2\alpha_2+3\alpha_3+\alpha_4\}\\
&\qquad\qquad\quad=\{\omega_2-\omega_4,\omega_3-\omega_4,\omega_4\}
\end{align*}
and $\lie{u}\cap\lie{g}^{-\theta} \simeq\bbC^3\oplus (\bbC^3)^*\oplus \bbC$,
which is $(L\cap K)$-spherical.

Hence $G/P_J\times K/{B_K}$ is of finite type if and only if
 $|\Pi\setminus J|=1$ or $\Pi\setminus J=\{\alpha_1,\alpha_4\}$.

\subsection{$(\lie{e}_6, \lie{sp}_4)$}
\ 

Let $(\lie{g}, \lie{k})=(\lie{e}_6, \lie{sp}_4)$.
We fix $B$, $T$ and simple roots $\alpha_1,\dots,\alpha_6$
 corresponding to the following Vogan diagram.
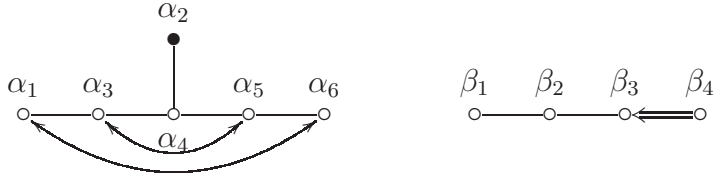
\begin{align*}
\begin{xy}
\ar@{-} (0,0) *++!D{\alpha_1} *{\circ}="A"; (10,0) *++!D{\alpha_3} 
 *{\circ}="B"
\ar@{-} "B"; (20,0)*++!U{\alpha_4}  *{\circ}="C"
\ar@{-} "C"; (30,0) *++!D{\alpha_5}  *{\circ}="D"
\ar@{-} "C"; (20,10) *++!D{\alpha_2}  *{\bullet}="G"
\ar@{-} "D"; (40,0) *++!D{\alpha_6}  *{\circ}="E"
\ar@/_1.2pc/@{<->} "B"; "D"
\ar@/_1.8pc/@{<->} "A"; "E"
\ar@{-} (60,0) *++!D{\beta_1} *{\circ}="F"; (70,0) *++!D{\beta_2} 
 *{\circ}="G"
\ar@{-} "G"; (80,0)*++!D{\beta_3}  *{\circ}="H"
\ar@{<=} "H"; (90,0) *++!D{\beta_4}  *{\circ}="I"
\end{xy}
\end{align*}
Let $\beta_1:=(\alpha_2+\alpha_3+\alpha_4)|_{\lie{t}^\theta}$,
 $\beta_2:=\alpha_1|_{\lie{t}^\theta}$, $\beta_3:=\alpha_3|_{\lie{t}^\theta}$,
 and $\beta_4:=\alpha_4|_{\lie{t}^\theta}$.
Then $\Pi_K=\{\beta_1,\beta_2,\beta_3,\beta_4\}$ is a set of simple roots 
 for $K$ and the corresponding Dynkin diagram is given as above.

The dimension condition is $\dim L\geq \dim G-\dim K-\rank K= 78-36-4=38$.
This holds only if $\Pi\setminus J=\{\alpha_1\}$ or $\{\alpha_6\}$.

Suppose that $J=\Pi\setminus\{\alpha_1\}$.
Then $J_K=\Pi_K\setminus\{\beta_2\}$ and
 $\lie{l}\cap\lie{k}=\lie{l}_{K,J_K}\simeq \lie{gl}_2\oplus\lie{sp}_{2}
 (\simeq \lie{gl}_2\oplus\lie{so}_{5})$.
We have 
\begin{align*}
\Delta(\lie{u}\cap\lie{g}^{-\theta},\lie{t}^{\theta})
&=\{
(\alpha_1+\alpha_2+\alpha_3+\alpha_4+\alpha_5+\alpha_6)|_{\lie{t}^\theta}, 
(\alpha_1+\alpha_2+\alpha_3+2\alpha_4+\alpha_5+\alpha_6)|_{\lie{t}^\theta},\\
&(\alpha_1+\alpha_2+2\alpha_3+2\alpha_4+\alpha_5+\alpha_6)|_{\lie{t}^\theta},
(\alpha_1+\alpha_2+2\alpha_3+2\alpha_4+2\alpha_5+\alpha_6)
 |_{\lie{t}^\theta},\\
&(\alpha_1+\alpha_2+2\alpha_3+3\alpha_4+2\alpha_5+\alpha_6)|_{\lie{t}^\theta}
\}.
\end{align*}
Hence $\Lambda^+(\lie{u}\cap\lie{g}^{-\theta})
 =\{(\alpha_1+\alpha_2+2\alpha_3+3\alpha_4+2\alpha_5+\alpha_6)
 |_{\lie{t}^\theta}\}=\{\omega_4\}$
 and $\lie{u}\cap\lie{g}^{-\theta} \simeq \bbC^5$,
which is $(L\cap K)$-spherical.

Therefore, $G/P_J\times K/{B_K}$ is of finite type if and only if
 $\Pi\setminus J=\{\alpha_1\}$ or $\{\alpha_6\}$.

\subsection{$(\lie{e}_6, \lie{sl}_6 \oplus\lie{sl}_2)$} \label{e6sl}
\ 

Let $(\lie{g}, \lie{k})=(\lie{e}_6, \lie{sl}_6\oplus\lie{sl}_2)$.
We fix $B$, $T$ and simple roots $\alpha_1,\dots,\alpha_6$
 corresponding to the following Vogan diagram.
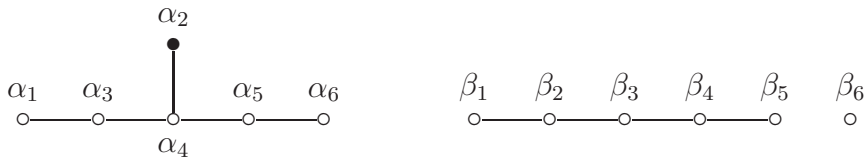
\begin{align*}
\begin{xy}
\ar@{-} (0,0) *++!D{\alpha_1} *{\circ}="A"; (10,0) *++!D{\alpha_3} 
 *{\circ}="B"
\ar@{-} "B"; (20,0)*++!U{\alpha_4}  *{\circ}="C"
\ar@{-} "C"; (30,0) *++!D{\alpha_5}  *{\circ}="D"
\ar@{-} "C"; (20,10) *++!D{\alpha_2}  *{\bullet}="G"
\ar@{-} "D"; (40,0) *++!D{\alpha_6}  *{\circ}="E"
\ar@{-} (60,0) *++!D{\beta_1} *{\circ}="F"; (70,0) *++!D{\beta_2} 
 *{\circ}="G"
\ar@{-} "G"; (80,0)*++!D{\beta_3}  *{\circ}="H"
\ar@{-} "H"; (90,0)*++!D{\beta_4}  *{\circ}="I"
\ar@{-} "I"; (100,0)*++!D{\beta_5}  *{\circ}="J"
\ar@{} "J"; (110,0) *++!D{\beta_6} *{\circ}
\end{xy}
\end{align*}
Let $\beta_1:=\alpha_1$, $\beta_2:=\alpha_3$, $\beta_3:=\alpha_4$,
 $\beta_4:=\alpha_5$, $\beta_5:=\alpha_6$, and
 $\beta_6:=\alpha_1+2\alpha_2+2\alpha_3+3\alpha_4+2\alpha_5+\alpha_6$.
Then $\Pi_K=\{\beta_1,\dots,\beta_6\}$ is a set of simple roots 
 for $K$ and the corresponding Dynkin diagram is given as above.

The dimension condition is $\dim L\geq \dim G-\dim K-\rank K= 78-38-6=34$.
This holds only if $\Pi\setminus J=\{\alpha_1\}$, $\{\alpha_2\}$
 or $\{\alpha_6\}$.

Suppose that $J=\Pi\setminus\{\alpha_1\}$.
Then $J_K=\Pi_K\setminus\{\beta_1,\beta_6\}$ and
 $\lie{l}\cap\lie{k}=\lie{l}_{K,J_K}\simeq \lie{gl}_5\oplus\lie{gl}_1$.
We have 
$\Delta(\lie{u}\cap\lie{g}^{-\theta},\lie{t}^{\theta})
=\{\sum_{i=1}^6 m_i\alpha_i\in\Delta^+:m_1=m_2=1\}$.
Hence $\Lambda^+(\lie{u}\cap\lie{g}^{-\theta})
 =\{\alpha_1+\alpha_2+2\alpha_3+3\alpha_4+2\alpha_5+\alpha_6\}
 =\{\omega_3+\omega_6\}$
and 
$\lie{u}\cap\lie{g}^{-\theta}\simeq\bigwedge^2 \bbC^5$,
which is $(L\cap K)$-spherical.

Suppose that $J=\Pi\setminus\{\alpha_2\}$.
Then $J_K=\Pi_K\setminus\{\beta_6\}$ and
 $\lie{l}\cap\lie{k}=\lie{l}_{K,J_K}\simeq \lie{sl}_6\oplus\lie{gl}_1$.
We have 
$\Delta(\lie{u}\cap\lie{g}^{-\theta},\lie{t}^{\theta})
=\{\sum_{i=1}^6 m_i\alpha_i\in\Delta^+:m_2=1\}$ and
$\lie{u}\cap\lie{g}^{-\theta}\simeq \bigwedge^3 \bbC^6$,
which is not $(L\cap K)$-spherical.

Therefore, $G/P_J\times K/{B_K}$ is of finite type if and only if
 $\Pi\setminus J=\{\alpha_1\}$ or $\{\alpha_6\}$.

\subsection{$(\lie{e}_6, \lie{f}_4)$} \label{e6f4}
\ 

Let $(\lie{g}, \lie{k})=(\lie{e}_6, \lie{f}_4)$.
We fix $B$, $T$ and simple roots $\alpha_1,\dots,\alpha_6$
 corresponding to the following Vogan diagram.
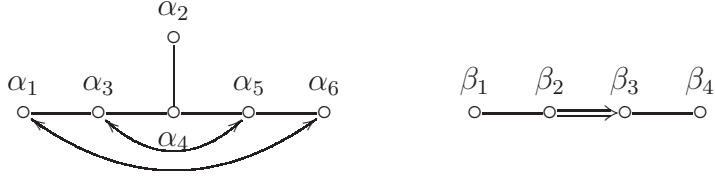
\begin{align*}
\begin{xy}
\ar@{-} (0,0) *++!D{\alpha_1} *{\circ}="A"; (10,0) *++!D{\alpha_3} 
 *{\circ}="B"
\ar@{-} "B"; (20,0)*++!U{\alpha_4}  *{\circ}="C"
\ar@{-} "C"; (30,0) *++!D{\alpha_5}  *{\circ}="D"
\ar@{-} "C"; (20,10) *++!D{\alpha_2}  *{\circ}="G"
\ar@{-} "D"; (40,0) *++!D{\alpha_6}  *{\circ}="E"
\ar@/_1.2pc/@{<->} "B"; "D"
\ar@/_1.8pc/@{<->} "A"; "E"
\ar@{-} (60,0) *++!D{\beta_1} *{\circ}="F";
  (70,0) *++!D{\beta_2} *{\circ}="G"
\ar@{=>} "G" ; (80,0)  *++!D{\beta_3} *{\circ}="H"
\ar@{-} "H" ; (90,0) *++!D{\beta_4} *{\circ}
\end{xy}
\end{align*}
Let $\beta_1:=\alpha_2|_{\lie{t}^\theta}$,
 $\beta_2:=\alpha_4|_{\lie{t}^\theta}$, $\beta_3:=\alpha_3|_{\lie{t}^\theta}$,
 and $\beta_4:=\alpha_1|_{\lie{t}^\theta}$.
Then $\Pi_K=\{\beta_1,\beta_2,\beta_3,\beta_4\}$ is a set of simple roots 
 for $K$ and the corresponding Dynkin diagram is given as above.

The dimension condition is $\dim L\geq \dim G-\dim K-\rank K= 78-52-4=22$.
This is satisfied for $\Pi\setminus J= \{\alpha_i\} (i\neq 4)$, 
 $\{\alpha_1, \alpha_2\}$, $\{\alpha_2, \alpha_6\}$,
 $\{\alpha_1, \alpha_3\}$, $\{\alpha_5, \alpha_6\}$,
 $\{\alpha_1, \alpha_6\}$.

Suppose that $J=\Pi\setminus\{\alpha_1,\alpha_2\}$.
Then $J_K=\Pi_K\setminus\{\beta_1,\beta_4\}$ and
 $\lie{l}\cap\lie{k}=\lie{l}_{K,J_K}\simeq
 \lie{gl}_1\oplus\lie{sp}_2\oplus\lie{gl}_1$.
We have 
\begin{align*}
\Delta(\lie{u}\cap\lie{g}^{-\theta},\lie{t}^{\theta})
&=\{(\alpha_2+\alpha_3+\alpha_4)|_{\lie{t}^\theta}, 
(\alpha_1+\alpha_2+\alpha_3+\alpha_4)|_{\lie{t}^\theta},
\\
&
(\alpha_1+\alpha_2+\alpha_3+\alpha_4+\alpha_5)|_{\lie{t}^\theta},
(\alpha_1+\alpha_2+\alpha_3+2\alpha_4+\alpha_5)|_{\lie{t}^\theta},
\\
&
(\alpha_1+\alpha_2+2\alpha_3+2\alpha_4+\alpha_5)|_{\lie{t}^\theta},
(\alpha_1+\alpha_2+2\alpha_3+2\alpha_4+\alpha_5+\alpha_6)|_{\lie{t}^\theta}
\}.
\end{align*}
Hence 
\begin{align*}
\Lambda^+(\lie{u}\cap\lie{g}^{-\theta})
&=\{(\alpha_2+\alpha_3+\alpha_4)|_{\lie{t}^\theta},
 (\alpha_1+\alpha_2+2\alpha_3+2\alpha_4+\alpha_5)|_{\lie{t}^\theta},
\\
&\qquad
 (\alpha_1+\alpha_2+2\alpha_3+2\alpha_4+\alpha_5+\alpha_6)|_{\lie{t}^\theta}\}
\\
&=\{\omega_1-\omega_4,\omega_3-\omega_4,\omega_4\}
\end{align*}
and $\lie{u}\cap\lie{g}^{-\theta} \simeq \bbC\oplus\bbC^4\oplus \bbC$,
which is $(L\cap K)$-spherical.

Suppose that $J=\Pi\setminus\{\alpha_1,\alpha_3\}$.
Then $J_K=\Pi_K\setminus\{\beta_3,\beta_4\}$ and
 $\lie{l}\cap\lie{k}=\lie{l}_{K,J_K}\simeq \lie{gl}_3\oplus\lie{gl}_1$.
We have 
\begin{align*}
\Delta(\lie{u}\cap\lie{g}^{-\theta},\lie{t}^{\theta})
&=\{(\alpha_1+\alpha_3+\alpha_4+\alpha_5)|_{\lie{t}^\theta}, 
(\alpha_1+\alpha_2+\alpha_3+\alpha_4+\alpha_5)|_{\lie{t}^\theta},
\\
&
(\alpha_1+\alpha_2+\alpha_3+2\alpha_4+\alpha_5)|_{\lie{t}^\theta},
(\alpha_1+\alpha_2+2\alpha_3+2\alpha_4+\alpha_5)|_{\lie{t}^\theta},
\\
&
(\alpha_1+\alpha_2+2\alpha_3+2\alpha_4+\alpha_5+\alpha_6)|_{\lie{t}^\theta}
\}.
\end{align*}
Hence 
\begin{align*}
\Lambda^+(\lie{u}\cap\lie{g}^{-\theta})
&=\{(\alpha_1+\alpha_2+\alpha_3+2\alpha_4+\alpha_5)|_{\lie{t}^\theta},
 (\alpha_1+\alpha_2+2\alpha_3+2\alpha_4+\alpha_5)|_{\lie{t}^\theta},\\
&\qquad
 (\alpha_1+\alpha_2+2\alpha_3+2\alpha_4+\alpha_5+\alpha_6)|_{\lie{t}^\theta}\}
\\
&=\{\omega_2-\omega_3,\omega_3-\omega_4,\omega_4\}
\end{align*}
and 
$\lie{u}\cap\lie{g}^{-\theta} \simeq (\bbC^3)^* \oplus \bbC \oplus \bbC$,
which is $(L\cap K)$-spherical.

Suppose that $J=\Pi\setminus\{\alpha_1,\alpha_6\}$.
Then $J_K=\Pi_K\setminus\{\beta_4\}$ and
 $\lie{l}\cap\lie{k}=\lie{l}_{K,J_K}\simeq \lie{so}_7\oplus\lie{gl}_1$.
We have
\begin{align*}
\Lambda^+(\lie{u}\cap\lie{g}^{-\theta})
&=\{(\alpha_1+\alpha_2+2\alpha_3+2\alpha_4+\alpha_5)|_{\lie{t}^\theta},
 (\alpha_1+\alpha_2+2\alpha_3+2\alpha_4+\alpha_5+\alpha_6)|_{\lie{t}^\theta}\}
\\
&=\{\omega_3-\omega_4,\omega_4\}
\end{align*} 
and
$\lie{u}\cap\lie{g}^{-\theta}\simeq \bbC^8 \oplus \bbC$,
which is not $(L\cap K)$-spherical.

Hence $G/P_J\times K/{B_K}$ is of finite type if and only if
 $\Pi\setminus J= \{\alpha_i\} (i\neq 4)$, 
 $\{\alpha_1, \alpha_2\}$, $\{\alpha_2, \alpha_6\}$, 
 $\{\alpha_1, \alpha_3\}$, or $\{\alpha_5, \alpha_6\}$.

\subsection{$(\lie{e}_7, \lie{sl}_8)$}
\ 

Let $(\lie{g}, \lie{k})=(\lie{e}_7, \lie{sl}_8)$.
We fix $B$, $T$ and simple roots $\alpha_1,\dots,\alpha_7$
 corresponding to the following Vogan diagram.
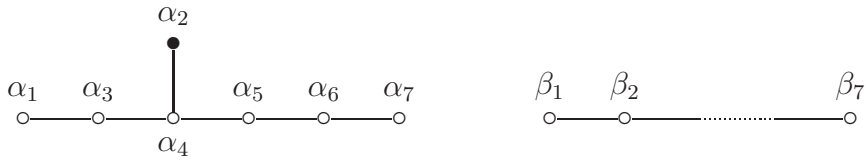
\begin{align*}
\begin{xy}
\ar@{-} (0,0) *++!D{\alpha_1} *{\circ}="A"; (10,0) *++!D{\alpha_3} 
 *{\circ}="B"
\ar@{-} "B"; (20,0)*++!U{\alpha_4}  *{\circ}="C"
\ar@{-} "C"; (30,0) *++!D{\alpha_5}  *{\circ}="D"
\ar@{-} "C"; (20,10) *++!D{\alpha_2}  *{\bullet}
\ar@{-} "D"; (40,0) *++!D{\alpha_6}  *{\circ}="E"
\ar@{-} "E"; (50,0) *++!D{\alpha_7}  *{\circ}
\ar@{-} (70,0) *++!D{\beta_1} *{\circ}="F";
  (80,0) *++!D{\beta_2} *{\circ}="G"
\ar@{-} "G" ; (90,0) 
\ar@{.} (90,0) ; (100,0) 
\ar@{-} (100,0) ; (110,0) *++!D{\beta_7} *{\circ}
\end{xy}
\end{align*}
Let $\beta_1:=\alpha_1$,
 $\beta_i:=\alpha_{i+1}$ for $2\leq i\leq 6$,
 and $\beta_7:=\alpha_1+2\alpha_2+2\alpha_3+3\alpha_4+2\alpha_5+\alpha_6$.
Then $\Pi_K=\{\beta_1,\dots,\beta_7\}$ is a set of simple roots 
 for $K$ and the corresponding Dynkin diagram is given as above.

The dimension condition is $\dim L\geq \dim G-\dim K-\rank K= 133-63-7=63$,
 which implies $\Pi\setminus J= \{\alpha_1\}$ or $\{\alpha_7\}$.

Suppose that $J=\Pi\setminus\{\alpha_1\}$.
Then $J_K=\Pi_K\setminus\{\beta_1,\beta_7\}$ and
 $\lie{l}\cap\lie{k}=\lie{l}_{K,J_K}\simeq
 \lie{gl}_1\oplus\lie{sl}_6\oplus\lie{gl}_1$.
We have 
$\Delta(\lie{u}\cap\lie{g}^{-\theta},\lie{t}^{\theta})
=\{\sum_{i=1}^7 m_i\alpha_i\in\Delta^+:m_1=m_2=1\}$.
Hence 
\begin{align*}
\Lambda^+(\lie{u}\cap\lie{g}^{-\theta})
=\{\alpha_1+\alpha_2+2\alpha_3+3\alpha_4+3\alpha_5+2\alpha_6+\alpha_7\}
=\{\omega_4\}
\end{align*}
and 
$\lie{u}\cap\lie{g}^{-\theta} \simeq \bigwedge^3 \bbC^6$,
which is not $(L\cap K)$-spherical.

Suppose that $J=\Pi\setminus\{\alpha_7\}$.
Then $J_K=\Pi_K\setminus\{\beta_6\}$ and
 $\lie{l}\cap\lie{k}=\lie{l}_{K,J_K}\simeq
 \lie{sl}_6\oplus\lie{gl}_1\oplus\lie{sl}_2$.
We have $\Delta(\lie{u}\cap\lie{g}^{-\theta},\lie{t}^{\theta})
=\{\sum_{i=1}^7 m_i\alpha_i\in\Delta^+:m_2=m_7=1\}$.
Hence 
$\Lambda^+(\lie{u}\cap\lie{g}^{-\theta})=\{\omega_4\}$
and 
$\lie{u}\cap\lie{g}^{-\theta} \simeq \bigwedge^2 \bbC^6$,
which is $(L\cap K)$-spherical.

Therefore, $G/P_J\times K/{B_K}$ is of finite type if and only if
 $\Pi\setminus J= \{\alpha_7\}$.

\subsection{$(\lie{e}_7, \lie{so}_{12}\oplus\lie{sl}_2)$} \label{e7so}
\ 

Let $(\lie{g}, \lie{k})=(\lie{e}_7, \lie{so}_{12}\oplus\lie{sl}_2)$.
We fix $B$, $T$ and simple roots $\alpha_1,\dots,\alpha_7$
 corresponding to the following Vogan diagram.
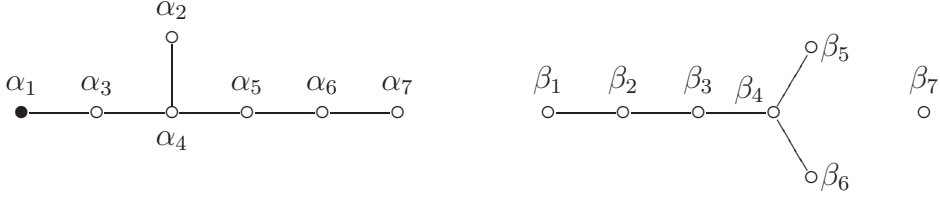
\begin{align*}
\begin{xy}
\ar@{-} (0,0) *++!D{\alpha_1} *{\bullet}="A"; (10,0) *++!D{\alpha_3} 
 *{\circ}="B"
\ar@{-} "B"; (20,0)*++!U{\alpha_4}  *{\circ}="C"
\ar@{-} "C"; (30,0) *++!D{\alpha_5}  *{\circ}="D"
\ar@{-} "C"; (20,10) *++!D{\alpha_2}  *{\circ}
\ar@{-} "D"; (40,0) *++!D{\alpha_6}  *{\circ}="E"
\ar@{-} "E"; (50,0) *++!D{\alpha_7}  *{\circ}
\ar@{-} (70,0) *++!D{\beta_1} *{\circ}="F"; (80,0) *++!D{\beta_2} 
 *{\circ}="H"
\ar@{-} "H" ; (90,0)  *++!D{\beta_3}  *{\circ}="I"
\ar@{-} "I" ; (100,0) *+!DR{\beta_4} *{\circ}="G" 
\ar@{-} "G" ; (105,8.6)  *+!L{\beta_5} *{\circ}
\ar@{-} "G" ; (105,-8.6)  *+!L{\beta_6} *{\circ}
\ar@{} "G" ; (120,0) *++!D{\beta_7} *{\circ}
\end{xy}
\end{align*}
Let $\beta_i:=\alpha_{8-i}$ for $1\leq i\leq 6$,
 and $\beta_7:=2\alpha_1+2\alpha_2+3\alpha_3
 +4\alpha_4+3\alpha_5+2\alpha_6+\alpha_7$.
Then $\Pi_K=\{\beta_1,\dots,\beta_7\}$ is a set of simple roots 
 for $K$ and the corresponding Dynkin diagram is given as above.

The dimension condition is $\dim L\geq \dim G-\dim K-\rank K= 133-69-7=57$,
 which implies $\Pi\setminus J= \{\alpha_1\}$ or $\{\alpha_7\}$.

Suppose that $J=\Pi\setminus\{\alpha_1\}$.
Then $J_K=\Pi_K\setminus\{\beta_7\}$ and
 $\lie{l}\cap\lie{k}=\lie{l}_{K,J_K}\simeq
 \lie{so}_{12}\oplus\lie{gl}_1$.
We have $\Delta(\lie{u}\cap\lie{g}^{-\theta},\lie{t}^{\theta})
=\{\sum_{i=1}^7 m_i\alpha_i\in\Delta^+:m_1=1\}$.
Hence 
\begin{align*}
\Lambda^+(\lie{u}\cap\lie{g}^{-\theta})
=\{\alpha_1+2\alpha_2+3\alpha_3+4\alpha_4+3\alpha_5+2\alpha_6+\alpha_7\}
=\{\omega_5+\omega_7\}
\end{align*}
and 
$\lie{u}\cap\lie{g}^{-\theta} \simeq \bbC^{32}$,
which is not $(L\cap K)$-spherical.

Suppose that $J=\Pi\setminus\{\alpha_7\}$.
Then $J_K=\Pi_K\setminus\{\beta_1,\beta_7\}$ and
 $\lie{l}\cap\lie{k}=\lie{l}_{K,J_K}\simeq
 \lie{gl}_1\oplus\lie{so}_{10}\oplus\lie{gl}_1$.
We have $\Delta(\lie{u}\cap\lie{g}^{-\theta},\lie{t}^{\theta})
=\{\sum_{i=1}^7 m_i\alpha_i\in\Delta^+:m_1=m_7=1\}$.
Hence
$\Lambda^+(\lie{u}\cap\lie{g}^{-\theta})=\{\omega_5+\omega_7\}$
and 
$\lie{u}\cap\lie{g}^{-\theta} \simeq \bbC^{16}$,
which is $(L\cap K)$-spherical.

Therefore, $G/P_J\times K/{B_K}$ is of finite type if and only if
 $\Pi\setminus J= \{\alpha_7\}$.

\subsection{$(\lie{e}_8, \lie{so}_{16})$}
\ 

Let $(\lie{g}, \lie{k})=(\lie{e}_8, \lie{so}_{16})$.
We fix $B$, $T$ and simple roots $\alpha_1,\dots,\alpha_8$
 corresponding to the following Vogan diagram.
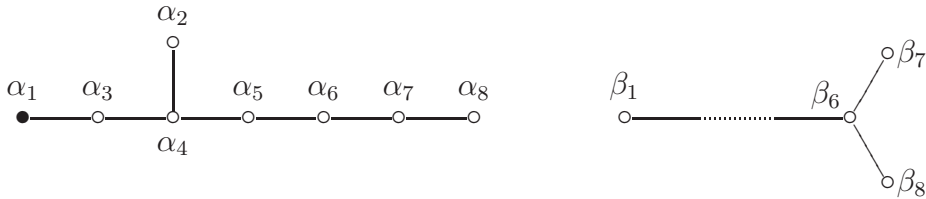
\begin{align*}
\begin{xy}
\ar@{-} (0,0) *++!D{\alpha_1} *{\bullet}="A"; (10,0) *++!D{\alpha_3} 
 *{\circ}="B"
\ar@{-} "B"; (20,0)*++!U{\alpha_4}  *{\circ}="C"
\ar@{-} "C"; (30,0) *++!D{\alpha_5}  *{\circ}="D"
\ar@{-} "C"; (20,10) *++!D{\alpha_2}  *{\circ}
\ar@{-} "D"; (40,0) *++!D{\alpha_6}  *{\circ}="E"
\ar@{-} "E"; (50,0) *++!D{\alpha_7}  *{\circ}="F"
\ar@{-} "F"; (60,0) *++!D{\alpha_8}  *{\circ}
\ar@{-} (80,0) *++!D{\beta_1} *{\circ}; (90,0) 
\ar@{.} (90,0); (100,0)
\ar@{-} (100,0) ; (110,0) *+!DR{\beta_6} *{\circ}="G" 
\ar@{-} "G" ; (115,8.6)  *+!L{\beta_7} *{\circ}
\ar@{-} "G" ; (115,-8.6)  *+!L{\beta_8} *{\circ}
\end{xy}
\end{align*}
Let $\beta_1:=2\alpha_1+2\alpha_2+3\alpha_3+4\alpha_4+3\alpha_5
 +2\alpha_6+\alpha_7$
 and $\beta_i:=\alpha_{10-i}$ for $2\leq i\leq 8$.
Then $\Pi_K=\{\beta_1,\dots,\beta_8\}$ is a set of simple roots 
 for $K$ and the corresponding Dynkin diagram is given as above.

The dimension condition is $\dim L\geq \dim G-\dim K-\rank K= 248-120-8=120$,
 which implies $\Pi\setminus J= \{\alpha_8\}$.

Suppose that $J=\Pi\setminus\{\alpha_8\}$.
Then $J_K=\Pi_K\setminus\{\beta_2\}$ and
 $\lie{l}\cap\lie{k}=\lie{l}_{K,J_K}\simeq
 \lie{gl}_2\oplus\lie{so}_{12}$.
We have $\Delta(\lie{u}\cap\lie{g}^{-\theta},\lie{t}^{\theta})
=\{\sum_{i=1}^8 m_i\alpha_i\in\Delta^+:m_1=m_8=1\}$, 
\begin{align*}
\Lambda^+(\lie{u}\cap\lie{g}^{-\theta})
=\{\alpha_1+3\alpha_2+3\alpha_3+5\alpha_4+4\alpha_5+3\alpha_6+2\alpha_7
 +\alpha_8\}=\{\omega_8\},
\end{align*}
and 
$\lie{u}\cap\lie{g}^{-\theta} \simeq \bbC^{32}$,
which is not $(L\cap K)$-spherical.

Hence there is no $G/P_J\times K/{B_K}$ of finite type.

\subsection{$(\lie{e}_8, \lie{e}_{7}\oplus\lie{sl}_2)$}
\ 

Let $(\lie{g}, \lie{k})=(\lie{e}_8, \lie{e}_{7}\oplus\lie{sl}_2)$.
We fix $B$, $T$ and simple roots $\alpha_1,\dots,\alpha_8$
 corresponding to the following Vogan diagram.
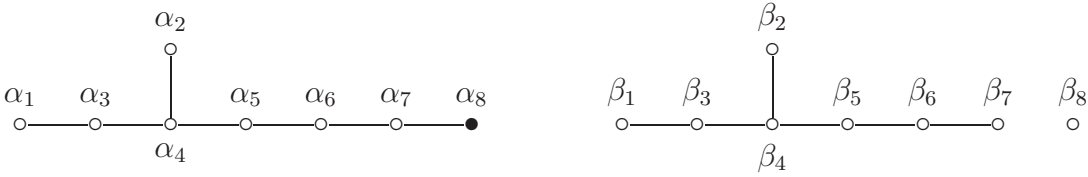
\begin{align*}
\begin{xy}
\ar@{-} (0,0) *++!D{\alpha_1} *{\circ}="A"; (10,0) *++!D{\alpha_3} 
 *{\circ}="B"
\ar@{-} "B"; (20,0)*++!U{\alpha_4}  *{\circ}="C"
\ar@{-} "C"; (30,0) *++!D{\alpha_5}  *{\circ}="D"
\ar@{-} "C"; (20,10) *++!D{\alpha_2}  *{\circ}
\ar@{-} "D"; (40,0) *++!D{\alpha_6}  *{\circ}="E"
\ar@{-} "E"; (50,0) *++!D{\alpha_7}  *{\circ}="F"
\ar@{-} "F"; (60,0) *++!D{\alpha_8}  *{\bullet}
\ar@{-} (80,0) *++!D{\beta_1} *{\circ}="G"; (90,0) *++!D{\beta_3} 
 *{\circ}="H"
\ar@{-} "H"; (100,0)*++!U{\beta_4}  *{\circ}="I"
\ar@{-} "I"; (110,0) *++!D{\beta_5}  *{\circ}="J"
\ar@{-} "I"; (100,10) *++!D{\beta_2}  *{\circ}
\ar@{-} "J"; (120,0) *++!D{\beta_6}  *{\circ}="K"
\ar@{-} "K"; (130,0) *++!D{\beta_7}  *{\circ}="L"
\ar@{} "L"; (140,0) *++!D{\beta_8}  *{\circ}
\end{xy}
\end{align*}
Let $\beta_i:=\alpha_i$ for $1\leq i\leq 7$ and
 $\beta_8:=2\alpha_1+3\alpha_2+4\alpha_3+6\alpha_4+5\alpha_5
 +4\alpha_6+3\alpha_7+2\alpha_8$.
Then $\Pi_K=\{\beta_1,\dots,\beta_8\}$ is a set of simple roots 
 for $K$ and the corresponding Dynkin diagram is given as above.

The dimension condition is $\dim L\geq \dim G-\dim K-\rank K= 248-136-8=104$,
 which implies $\Pi\setminus J= \{\alpha_8\}$.

Suppose that $J=\Pi\setminus\{\alpha_8\}$.
Then $J_K=\Pi_K\setminus\{\beta_8\}$ and
 $\lie{l}\cap\lie{k}=\lie{l}_{K,J_K}\simeq
 \lie{e}_7\oplus\lie{gl}_1$.
We have $\Delta(\lie{u}\cap\lie{g}^{-\theta},\lie{t}^{\theta})
=\{\sum_{i=1}^8 m_i\alpha_i\in\Delta^+:m_8=1\}$,
\begin{align*}
\Lambda^+(\lie{u}\cap\lie{g}^{-\theta})
=\{2\alpha_1+3\alpha_2+4\alpha_3+6\alpha_4+5\alpha_5+4\alpha_6+3\alpha_7
 +\alpha_8\}=\{\omega_7+\omega_8\},
\end{align*}
and 
$\lie{u}\cap\lie{g}^{-\theta} \simeq \bbC^{56}$,
 which is not $(L\cap K)$-spherical.

Hence there is no $G/P_J\times K/{B_K}$ of finite type.

\bigskip

We thus conclude that:

\begin{theorem}
\label{theorem:k.spherical}
Let $G$ be a connected simple algebraic group
 and $(G,K)$ a symmetric pair.
Let $P$ be a parabolic subgroup of $G$
 corresponding to $J \subsetneq \Pi$.
Then the double flag variety $ G/P \times K/B_K $
 is of finite type 
 if and only if the triple
 $(\lie{g}, \lie{k}, \Pi\setminus J)$
 appears in Table~\ref{table:k.spherical}.
\end{theorem}

\begin{remark}
\label{remark:choice.of.k}
For $\lie{g}\simeq \lie{so}_{4n}$,
 a symmetric subalgebra $\lie{k}$ that is isomorphic to
 $\lie{sl}_{2n}\oplus \bbC$ is not unique up to inner automorphisms of $\lie{g}$.
For $(\lie{g},\lie{k})=(\lie{so}_{2n}, \lie{sl}_{n}\oplus \bbC)$
 in Table~\ref{table:k.spherical},
 we take $(\lie{g},\lie{k})$
 and a positive system $\Delta^+$ in such a way that
 the Vogan diagram becomes
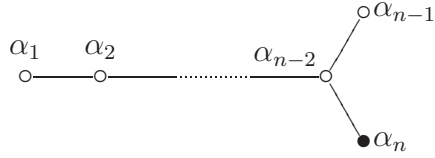
\begin{align*}
\begin{xy}
\ar@{-} (0,0) *++!D{\alpha_1} *{\circ}="A"; (10,0) *++!D{\alpha_2} 
 *{\circ}="B"
\ar@{-} "B"; (20,0)="C" 
\ar@{.} "C"; (30,0)="D" 
\ar@{-} "D"; (40,0) *+!DR{\alpha_{n-2}} *{\circ}="E"
\ar@{-} "E"; (45,8.6)  *+!L{\alpha_{n-1}} *{\circ}
\ar@{-} "E"; (45,-8.6)  *+!L{\alpha_n} *{\bullet}
\end{xy}
\end{align*}
In particular, 
 $K$ is the Levi component $L_{\Pi\setminus\{\alpha_n\}}$
 of the parabolic subgroup $P_{\Pi\setminus\{\alpha_n\}}$.

Similarly, the subalgebra $\lie{k}$ of $\lie{g}$
 is not unique for
 $(\lie{g},\lie{k})\simeq(\lie{so}_8, \lie{so}_7)$,
 $(\lie{so}_8, \lie{so}_6 \oplus \bbC)$,
 $(\lie{so}_8, \lie{so}_5 \oplus \lie{so}_3)$.
In Table~\ref{table:k.spherical},
 we take $(\lie{g},\lie{k})$
 and positive systems $\Delta^+$ in such a way that
 the Vogan diagrams become
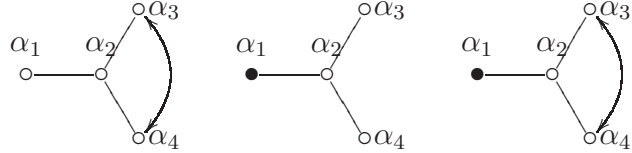
\begin{align*}
\begin{xy}
\ar@{-} (0,0) *++!D{\alpha_1} *{\circ}="A"; (10,0) *++!D{\alpha_2} 
 *{\circ}="B"
\ar@{-} "B"; (15,8.6)  *+!L{\alpha_{3}} *{\circ}="C"
\ar@{-} "B"; (15,-8.6)  *+!L{\alpha_4} *{\circ}="D"
\ar@/^1pc/@{<->} "C"; "D"
\ar@{-} (30,0) *++!D{\alpha_1} *{\bullet}="A"; (40,0) *++!D{\alpha_2} 
 *{\circ}="B"
\ar@{-} "B"; (45,8.6)  *+!L{\alpha_{3}} *{\circ}="C"
\ar@{-} "B"; (45,-8.6)  *+!L{\alpha_4} *{\circ}="D"
\ar@{-} (60,0) *++!D{\alpha_1} *{\bullet}="A"; (70,0) *++!D{\alpha_2} 
 *{\circ}="B"
\ar@{-} "B"; (75,8.6)  *+!L{\alpha_{3}} *{\circ}="C"
\ar@{-} "B"; (75,-8.6)  *+!L{\alpha_4} *{\circ}="D"
\ar@/^1pc/@{<->} "C"; "D"
\end{xy}
\end{align*}
for $(\lie{g},\lie{k})=(\lie{so}_8, \lie{so}_7)$,
 $(\lie{so}_8, \lie{so}_6 \oplus \bbC)$,
 $(\lie{so}_8, \lie{so}_5 \oplus \lie{so}_3)$, respectively.
\end{remark}

\begin{longtable}{c|c|c}
\hline
\rule[0pt]{0pt}{12pt}
$\lie{g}$ & $\lie{k}$ & $\Pi \setminus J\ (P=P_J)$
\\
\hline\hline
\multicolumn{1}{c}{
\rule[0pt]{0pt}{12pt}
\rule[-5pt]{0pt}{0pt}
$\lie{sl}_{n+1}$} & \multicolumn{1}{c}{} & 
\multicolumn{1}{c}{
\begin{xy}
\ar@{-} (0,0) *++!D{\alpha_1} *{\circ}="A"; (10,0) *++!D{\alpha_2} 
 *{\circ}="B"
\ar@{-} "B"; (20,0)="C" 
\ar@{.} "C"; (30,0)="D" 
\ar@{-} "D"; (40,0) *++!D{\alpha_{n}} *{\circ}="E"
\end{xy}
}\\
\hline

\rule[0pt]{0pt}{12pt}
\rule[-5pt]{0pt}{0pt}
$\lie{sl}_{n+1}$ & $\lie{so}_{n+1}$ & $\{\alpha_i\} (\forall i)$ \\
\hline

\rule[0pt]{0pt}{12pt}
$\lie{sl}_{2m}$ & $\lie{sp}_m$ &
 $\{\alpha_i\} (\forall i)$,
 $\{\alpha_i, \alpha_{i+1}\} (\forall i)$,\\
$2m=n+1$ &&
 $\{\alpha_1, \alpha_i\} (\forall i)$, 
 $\{\alpha_i, \alpha_n\}  (\forall i)$, \\
&& $\{\alpha_1, \alpha_2, \alpha_3\}$,
 $\{\alpha_{n-2}, \alpha_{n-1}, \alpha_n\}$,\\
\rule[-5pt]{0pt}{0pt}
&& $\{\alpha_1, \alpha_2, \alpha_n\}$,
 $\{\alpha_1, \alpha_{n-1}, \alpha_n\}$
 \\
\hline

\rule[0pt]{0pt}{12pt}
$\lie{sl}_{p+q}$ & $\lie{sl}_p\oplus\lie{sl}_q\oplus\bbC$&
 $\{\alpha_i\} (\forall i)$,
 $\{\alpha_i, \alpha_{i+1}\} (\forall i)$,\\
$p+q=n+1$ & $1\leq p\leq q$ &
 $\{\alpha_1, \alpha_i\} (\forall i)$, 
 $\{\alpha_i, \alpha_n\} (\forall i)$, \\
&& $\{\alpha_i, \alpha_j\} (\forall i, j)$ if $p=2$,\\
\rule[-5pt]{0pt}{0pt}
&& any subset of $\Pi$ if $p=1$\\
\hline\hline

\multicolumn{1}{c}{
\rule[0pt]{0pt}{12pt}
\rule[-5pt]{0pt}{0pt}
$\lie{so}_{2n+1}$} & \multicolumn{1}{c}{} & 
\multicolumn{1}{c}{
\begin{xy}
\ar@{-} (0,0) *++!D{\alpha_1} *{\circ}="A"; (10,0) *++!D{\alpha_2} 
 *{\circ}="B"
\ar@{-} "B"; (20,0)="C" 
\ar@{.} "C"; (30,0)="D" 
\ar@{-} "D"; (40,0) *++!D{\alpha_{n-1}} *{\circ}="E"
\ar@{=>} "E"; (50,0) *++!D{\alpha_n} *{\circ}="F"
\end{xy}
}\\
\hline

\rule[0pt]{0pt}{12pt}
$\lie{so}_{p+q}$ & $\lie{so}_p \oplus \lie{so}_q$
 & $\{\alpha_1\}$, $\{\alpha_n\}$, \\
$p+q=2n+1$ & $1\leq p\leq q$ &
 $\{\alpha_i\} (\forall i)$ if $p=2$,\\
\rule[-5pt]{0pt}{0pt}
&& any subset of $\Pi$ if $p=1$\\
\hline\hline

\multicolumn{1}{c}{
\rule[0pt]{0pt}{12pt}
\rule[-5pt]{0pt}{0pt}
$\lie{so}_{2n}$} &\multicolumn{1}{c}{} & 
\multicolumn{1}{c}{
\begin{xy}
\ar@{-} (0,0) *++!D{\alpha_1} *{\circ}="A"; (10,0) *++!D{\alpha_2} 
 *{\circ}="B"
\ar@{-} "B"; (20,0)="C" 
\ar@{.} "C"; (30,0)="D" 
\ar@{-} "D"; (40,0) *+!DR{\alpha_{n-2}} *{\circ}="E"
\ar@{-} "E"; (45,8.6)  *+!L{\alpha_{n-1}} *{\circ}
\ar@{-} "E"; (45,-8.6)  *+!L{\alpha_n} *{\circ}
\end{xy}
}\\
\hline

\rule[0pt]{0pt}{12pt}
$\lie{so}_{p+q}$ & \footnote[1]{See Remark~\ref{remark:choice.of.k}}$\lie{so}_p \oplus \lie{so}_q$ & $\{\alpha_1\}$,
 $\{\alpha_{n-1}\}$, $\{\alpha_n\}$, \\
$p+q=2n$ & $1\leq p\leq q$ &
$\{\alpha_i\} (\forall i)$ if $p=2$,\\
$n\geq 4$ && $\{\alpha_i, \alpha_{n-1}\} (\forall i)$ if $p=2$,\\
&& $\{\alpha_i, \alpha_{n}\} (\forall i)$ if $p=2$,\\
\rule[-5pt]{0pt}{0pt}
&& any subset of $\Pi$ if $p=1$\\
\hline

\rule[0pt]{0pt}{12pt}
\rule[-5pt]{0pt}{0pt}
$\lie{so}_{2n}$ & \footnotemark[1]
 $\lie{sl}_n\oplus\bbC$ & $\{\alpha_1\}$,
 $\{\alpha_2\}$, $\{\alpha_3\}$, $\{\alpha_{n-1}\}$, $\{\alpha_n\}$, \\
$n\geq 4$ && $\{\alpha_1,\alpha_2\}$, $\{\alpha_1,\alpha_{n-1}\}$,
 $\{\alpha_1,\alpha_n\}$, $\{\alpha_{n-1},\alpha_n\}$,\\
&& $\{\alpha_2, \alpha_3\}$ if $n=4$\\
\hline\hline

\multicolumn{1}{c}{
\rule[0pt]{0pt}{12pt}
\rule[-5pt]{0pt}{0pt}
$\lie{sp}_n$} &\multicolumn{1}{c}{} & 
\multicolumn{1}{c}{
\begin{xy}
\ar@{-} (0,0) *++!D{\alpha_1} *{\circ}="A"; (10,0) *++!D{\alpha_2} 
 *{\circ}="B"
\ar@{-} "B"; (20,0)="C" 
\ar@{.} "C"; (30,0)="D" 
\ar@{-} "D"; (40,0) *++!D{\alpha_{n-1}} *{\circ}="E"
\ar@{<=} "E"; (50,0) *++!D{\alpha_n} *{\circ}="F"
\end{xy}
}\\
\hline

\rule[0pt]{0pt}{12pt}
\rule[-5pt]{0pt}{0pt}
$\lie{sp}_n$ & $\lie{sl}_n\oplus \bbC$ & $\{\alpha_1\}$,
 $\{\alpha_n\}$ \\
\hline

\rule[0pt]{0pt}{12pt}
$\lie{sp}_{p+q}$ & $\lie{sp}_p \oplus \lie{sp}_q$ & $\{\alpha_1\}$,
 $\{\alpha_2\}$, $\{\alpha_3\}$,
 $\{\alpha_n\}$, $\{\alpha_1,\alpha_2\}$, \\
$p+q=n$ & $1\leq p\leq q$ &
 $\{\alpha_i\} (\forall i)$ if $p\leq 2$,\\
\rule[-5pt]{0pt}{0pt}
&& $\{\alpha_i, \alpha_j\} (\forall i, j)$ if $p=1$\\
\hline\hline

\multicolumn{1}{c}{
\rule[0pt]{0pt}{12pt}
\rule[-5pt]{0pt}{0pt}
$\lie{f}_4$} & \multicolumn{1}{c}{} & 
\multicolumn{1}{c}{
\begin{xy}
\ar@{-} (0,0) *++!D{\alpha_1} *{\circ}="A"; (10,0) *++!D{\alpha_2} 
 *{\circ}="B"
\ar@{=>} "B"; (20,0)*++!D{\alpha_3}  *{\circ}="C"
\ar@{-} "C"; (30,0) *++!D{\alpha_4}  *{\circ}="D"
\end{xy}
}\\
\hline

\rule[0pt]{0pt}{12pt}
\rule[-5pt]{0pt}{0pt}
$\lie{f}_4$ & $\lie{so}_9$ & $\{\alpha_1\}$,
 $\{\alpha_2\}$, $\{\alpha_3\}$, $\{\alpha_4\}$,
 $\{\alpha_1,\alpha_4\}$ \\
\hline\hline

\multicolumn{1}{c}{
\rule[0pt]{0pt}{12pt}
\rule[-5pt]{0pt}{0pt}
$\lie{e}_6$} &\multicolumn{1}{c}{} & 
\multicolumn{1}{c}{
\begin{xy}
\ar@{-} (0,0) *++!D{\alpha_1} *{\circ}="A"; (10,0) *++!D{\alpha_3} 
 *{\circ}="B"
\ar@{-} "B"; (20,0)*++!U{\alpha_4}  *{\circ}="C"
\ar@{-} "C"; (30,0) *++!D{\alpha_5}  *{\circ}="D"
\ar@{-} "C"; (20,10) *++!D{\alpha_2}  *{\circ}="G"
\ar@{-} "D"; (40,0) *++!D{\alpha_6}  *{\circ}="E"
\end{xy}
}\\
\hline

\rule[0pt]{0pt}{12pt}
\rule[-5pt]{0pt}{0pt}
$\lie{e}_6$ & $\lie{sp}_4$ & $\{\alpha_1\}$, $\{\alpha_6\}$ \\
\hline

\rule[0pt]{0pt}{12pt}
\rule[-5pt]{0pt}{0pt}
$\lie{e}_6$ & $\lie{sl}_6\oplus \lie{sl}_2$ & $\{\alpha_1\}$, $\{\alpha_6\}$ \\
\hline

\rule[0pt]{0pt}{12pt}
\rule[-5pt]{0pt}{0pt}
$\lie{e}_6$ & $\lie{so}_{10}\oplus \bbC$
 & $\{\alpha_1\}$, $\{\alpha_2\}$,
 $\{\alpha_3\}$, $\{\alpha_5\}$, $\{\alpha_6\}$,
 $\{\alpha_1,\alpha_6\}$ \\
\hline

\rule[0pt]{0pt}{12pt}
\rule[-5pt]{0pt}{0pt}
$\lie{e}_6$ & $\lie{f}_4$ & $\{\alpha_1\}$, $\{\alpha_2\}$,
 $\{\alpha_3\}$, $\{\alpha_5\}$, $\{\alpha_6\}$,\\
& & 
 $\{\alpha_1,\alpha_2\}$, $\{\alpha_2,\alpha_6\}$, $\{\alpha_1,\alpha_3\}$,
 $\{\alpha_5,\alpha_6\}$\\
\hline\hline

\multicolumn{1}{c}{
\rule[0pt]{0pt}{12pt}
\rule[-5pt]{0pt}{0pt}
$\lie{e}_7$} &\multicolumn{1}{c}{} & 
\multicolumn{1}{c}{
\begin{xy}
\ar@{-} (0,0) *++!D{\alpha_1} *{\circ}="A"; (10,0) *++!D{\alpha_3} 
 *{\circ}="B"
\ar@{-} "B"; (20,0)*++!U{\alpha_4}  *{\circ}="C"
\ar@{-} "C"; (30,0) *++!D{\alpha_5}  *{\circ}="D"
\ar@{-} "C"; (20,10) *++!D{\alpha_2}  *{\circ}="G"
\ar@{-} "D"; (40,0) *++!D{\alpha_6}  *{\circ}="E"
\ar@{-} "E"; (50,0) *++!D{\alpha_7}  *{\circ}="F"
\end{xy}}\\
\hline

\rule[0pt]{0pt}{12pt}
\rule[-5pt]{0pt}{0pt}
$\lie{e}_7$ & $\lie{sl}_8$ & $\{\alpha_7\}$ \\
\hline

\rule[0pt]{0pt}{12pt}
\rule[-5pt]{0pt}{0pt}
$\lie{e}_7$ & $\lie{so}_{12}\oplus \lie{sl}_2$ & $\{\alpha_7\}$ \\
\hline

\rule[0pt]{0pt}{12pt}
\rule[-5pt]{0pt}{0pt}
$\lie{e}_7$ & $\lie{e}_6\oplus \bbC$ & $\{\alpha_1\}$, $\{\alpha_2\}$,
 $\{\alpha_7\}$ \\
\hline

\caption{$K$-spherical $G/P$}
\label{table:k.spherical}
\vphantom{$\Biggm|$}
\end{longtable}


\section{Classification of $G$-spherical $G/Q$}

In this section, we give a classification of the triples
 $(G, K, Q)$ such that $G/B\times K/Q$ is of finite type.
As in the previous section,
 it is enough to consider the two cases:
\begin{itemize}
\item $G$ is simple;
\item $G=G'\times G'$, $K=\diag{G'}$, and $G'$ is simple.
\end{itemize}
In the latter case, 
 $G/B\times K/Q$ can be written
 as the triple flag variety $G'/B_{G'}\times G'/B_{G'} \times G'/Q$
 and the classification was already given
 (see Corollary~\ref{corollary:PBB}).

In what follows we assume that $G$ is simple.

\medskip

We first consider the case
 where $(G,K)$ is the complexification of a Hermitian symmetric pair.
We choose a $\theta$-stable Borel subgroup $B$
 of $G$ and a simple root $\beta\in\Pi$ such that
$K=L_{\Pi\setminus\{\beta\}}$.
Since $\rank G=\rank K$, we have $ T = T_K $.
Therefore the set of simple roots $\Pi_K$ for $K$ can be 
 regarded as a subset of $\Pi$ and then 
 $\Pi=\Pi_K \cup \{\beta\}$.

\begin{lemma}
\label{lemma:double.triple2} 
Suppose that
 $(G,K)$ is the complexification of a Hermitian symmetric pair
 and $Q_{J_K}$ is the parabolic subgroup of $K$
 corresponding to a subset $J_K\subset \Pi_K(\subset \Pi)$.
Choose a $\theta$-stable Borel subgroup $B$
 and a simple root $\beta$ such that
 $K=L_{\Pi\setminus\{\beta\}}$.
Then $G/B \times K/Q_{J_K}$ is of finite type if and only if
 $G/B \times G/P_{J_K} \times G/P_{\opposite{\Pi\setminus\{\beta\}}}$
 is of finite type.
\end{lemma}

\begin{proof}
The opposite parabolic subgroup $P_{\Pi\setminus\{\beta\}}^-$
 is conjugate to $P_{\opposite{\Pi\setminus\{\beta\}}}$ and 
 hence
$ G/P_{\Pi\setminus\{\beta\}}^-
 \simeq G/P_{\opposite{\Pi\setminus\{\beta\}}}$.
Since $P_{\Pi\setminus\{\beta\}}^- \cdot P_{J_K}$ is open in $G$ and
\begin{equation*}
P_{\Pi\setminus\{\beta\}}^- \cap P_{J_K}
= P_{\Pi\setminus\{\beta\}}^- \cap P_{\Pi\setminus\{\beta\}} \cap P_{J_K}
= K \cap P_{J_K}
= Q_{J_K},
\end{equation*}
 $G/Q_{J_K}$ is openly embedded into
 $G/P_{\Pi\setminus\{\beta\}}^- \times G/P_{J_K}$.
Hence $G/P_{J_K} \times G/P_{\opposite{\Pi\setminus\{\beta\}}}$
 is $G$-spherical if and only if $G/Q_{J_K}$ is $G$-spherical. 
(See also \cite{KNOT.2012}.)
\end{proof}

By Theorem~\ref{theorem:list.triple} and Lemma~\ref{lemma:double.triple2}, 
 we get a list of $G/B \times K/Q$ of finite type.

\medskip

For the remaining pairs $(G, K)$, we
 use Theorem~\ref{theorem:case.P=B}
 and the classification of spherical linear actions in \cite{Benson.Ratcliff.1996}.
We carry out a classification according to the following procedure.
\begin{enumerate}
\item
For each triple $(G, K, Q)$
 we choose a $\theta$-stable parabolic subgroup $P'$ of $G$
 such that $P'\cap K=Q$.
\item
Determine the Lie algebras $\lie{l}'$,
 $\lie{l}' \cap \lie{k}$ and $(\lie{u'})^{-\theta}$.
\item
Determine $\lie{m}'$ (see Theorem~\ref{theorem:case.P=B})
 and check whether the $M'_0$-action on 
 $(\lie{u'})^{-\theta}$ is spherical
 using the list of \cite{Benson.Ratcliff.1996}.
\end{enumerate}
The dimension condition
 $\dim G/B + \dim K/Q \leq \dim K$
 is helpful in some cases.
This is equivalent to $\dim L'_K \geq \dim G-\dim K-\rank G$,
 where $L'_K$ is a Levi component of $Q$.
Also we use Corollary~\ref{corollary:P=B.to.Q=B_K} to exclude
 unsuitable cases.

For step (1) we choose a Borel subgroup $B_K$ of $K$ and
 a Cartan subgroup $T_K$ in $B_K$.
We label the simple roots as $\Pi_K=\{\beta_1,\beta_2,\cdots\}$
 by describing the Dynkin diagram for $K$ and choose a parabolic subgroup
 $Q\supset B_K$ in terms of $\Pi_K$.
Then we take a $\theta$-stable Borel subgroup $B$ of $G$
 and a $\theta$-stable Cartan subgroup $T$ in $B$ satisfying:
\begin{itemize}
\item $B\cap K=B_K$,
\item $T\cap K=T_K$, and 
\item there exists a $\theta$-stable parabolic subgroup $P'$
 containing $B$ such that $P'\cap K=Q$.
\end{itemize}
Because of the third condition here,
 we need to choose different $B$ depending on $Q$.
We write the Vogan diagram corresponding to $(G,B,K)$
 and label the simple roots of $G$ as $\Pi=\{\alpha_1,\alpha_2,\dots\}$.
The standard parabolic subgroup $P'$ is given as $P'=P_{J'}$ for
 $J'\subset \Pi$.

For step (2) we have $L'=L_{J'}$ and $L'\cap K=L'_K$.
The determination of $(\lie{u'})^{-\theta}$ 
 will be done as in \S~\ref{section:k.spherical}.

For step (3) the `$M$-part' of a symmetric pair is well-known
 and found for example in \cite[Appendix C]{Knapp.2002}.
We restrict the $(L'\cap K)$-module $(\lie{u'})^{-\theta}$ to
 the subgroup $M'_0$ and see whether it is spherical.

\subsection{$(\lie{sl}_n, \lie{so}_n)$}
\ 

Let $(\lie{g}, \lie{k})=(\lie{sl}_n, \lie{so}_{n})$.
Then the dimension condition is
\[
\begin{aligned}
\dim L'_K &\geq \dim G-\dim K-\rank G=(n^2-1)-\frac{n(n-1)}{2}-(n-1)
\\
&=\frac{n(n-1)}{2}=\dim K.
\end{aligned}
\]
This does not hold for a proper parabolic subgroup $Q\subset K$.

\subsection{$(\lie{sl}_{2n}, \lie{sp}_n)$}
\ 

Let $(\lie{g}, \lie{k})=(\lie{sl}_{2n}, \lie{sp}_{n})$ for $n\geq 2$.
We fix the numbering $\beta_1,\dots,\beta_n \in \Pi_K$ in such a way that
 the Dynkin diagram of $K$ becomes
\begin{align*}
\begin{xy}
\ar@{-} (0,0) *++!D{\beta_1} *{\circ}="A"; (10,0) *++!D{\beta_2} 
 *{\circ}="B"
\ar@{-} "B"; (20,0)="C" 
\ar@{.} "C"; (30,0)="D" 
\ar@{-} "D"; (40,0) *++!D{\beta_{n-1}} *{\circ}="E"
\ar@{<=} "E"; (50,0) *++!D{\beta_n} *{\circ}="F"
\end{xy}
\end{align*}
Suppose that $J_K=\Pi_K\setminus\{\beta_i\}$ with $1\leq i\leq n$
 and $Q=Q_{J_K}$.
We take $B$ and $T$ as in \S~\ref{slsp}, so the Vogan diagram is
\begin{align*}
\begin{xy}
\ar@{-} (0,0) *++!D{\alpha_1} *{\circ}="A"; (10,0) 
\ar@{.} (10,0) ; (20,0) 
\ar@{-} (20,0) ; (30,0) *++!D{\alpha_{n-1}} *{\circ}="B"
\ar@{-} "B" ; (40,0) *++!D{\alpha_n} *{\circ}="C"
\ar@{-} "C"; (50,0)*++!D{\alpha_{n+1}}  *{\circ}="D"
\ar@{-} "D"; (60,0) 
\ar@{.} (60,0); (70,0) 
\ar@{-} (70,0); (80,0) *++!D{\alpha_{2n-1}}  *{\circ}="E"
\ar@/_1.0pc/@{<->} "B"; "D"
\ar@/_2.0pc/@{<->} "A"; "E"
\end{xy}
\end{align*}
and $\beta_j=\alpha_j|_{\lie{t}^\theta}$.
By putting $J':=\Pi\setminus \{\alpha_i,\alpha_{2n-i}\}$ and $P'=P_{J'}$,
 we get $P'\cap K=Q$.
In \S~\ref{slsp}
 we saw that $(\lie{u}')^{-\theta}$ is $(L'\cap K)$-spherical
 only if $i=1$ or $n$.

If $i=1$, we have $(\lie{l}',\lie{l}'\cap \lie{k})
 \simeq(\lie{gl}_1\oplus\lie{sl}_{2n-2}\oplus\lie{gl}_1,
 \lie{gl}_1\oplus\lie{sp}_{n-1})$ and hence $\lie{m}'=\lie{gl}_1\oplus
 \underbrace{\lie{sl}_2\oplus\cdots\oplus\lie{sl}_2}_{n-1}$.
Since $(\lie{u}')^{-\theta}\simeq \bbC^{2n-2}$ is a natural representation
 of the $\lie{sp}_{n-1}$-component in $\lie{l}'\cap \lie{k}$,
its restriction to $\underbrace{\lie{sl}_2\oplus\cdots\oplus\lie{sl}_2}_{n-1}$
 decomposes as
$(\lie{u}')^{-\theta}\simeq \underbrace{\bbC^2\oplus\cdots\oplus\bbC^2}_{n-1}$.
Here each $\lie{sl}_2$ acts naturally on one $\bbC^2$ 
 and trivially on the others.
Hence this is $M'_0$-spherical.

If $i=n$,  we have $(\lie{l}',\lie{l}'\cap \lie{k})
 =(\lie{sl}_n\oplus\lie{gl}_{1}\oplus\lie{sl}_n,
 \lie{gl}_1\oplus\lie{sl}_{n})$ and hence $\lie{m}'=\lie{gl}_1\oplus
 \underbrace{\lie{gl}_1\oplus\cdots\oplus\lie{gl}_1}_{n-1}$.
Therefore, $\lie{m'}=\lie{t}^\theta$ and 
 $M'_0$-module is spherical if and only if the weights are linearly
 independent.
Since $(\lie{u}')^{-\theta}\simeq \bigwedge^2\bbC^n$ as an $(L'\cap K)$-module,
 it is spherical as an $M'_0$-module if and only if $n\leq 3$.

Suppose that $J_K=\Pi_K\setminus\{\beta_1,\beta_n\}$.
Then we can take $B$ as above and then putting 
 $J'=\{\alpha_1,\alpha_n,\alpha_{2n-1}\}$, we have $P'\cap K=Q$.
According to \S~\ref{slsp}, $(\lie{u}')^{-\theta}$ is $(L'\cap K)$-spherical 
 only if $n=2$.
For $n=2$, since $(\lie{sl}_4,\lie{sp}_2)\simeq(\lie{so}_6,\lie{so}_5)$,
 any double flag variety is of finite type
 by Theorem~\ref{theorem:k.spherical}.

Consequently, $G/B\times K/Q_{J_K}$ is of finite type if and only if
 at least one of the following three conditions holds:
\begin{itemize}
\item $J_K=\Pi_K\setminus\{\beta_1\}$,
\item $n=3$ and $J_K=\Pi_K\setminus\{\beta_3\}$,
\item $n=2$.
\end{itemize}

\subsection{$(\lie{so}_n, \lie{so}_{p}\oplus\lie{so}_q)$}
\ 

Let $(\lie{g}, \lie{k})=(\lie{so}_n, \lie{so}_{p}\oplus\lie{so}_q)$
 with $p+q=n$.
If $p=1$ or $q=1$, then Theorem~\ref{theorem:k.spherical} implies
 that any double flag variety is of finite type.
If $p=2$ or $q=2$, the pair $(\lie{g}, \lie{k})$ is of Hermitian type.
Let $p,q\geq 3$. 
We assume moreover that $n$ is odd as the case where $n$ is even
 can be treated similarly.
Then we may assume that $p$ is even and $q$ is odd.
Put $m:=\frac{n-1}{2}$, $p'=\frac{p}{2}$, and $q'=\frac{q-1}{2}$.
We fix the numbering $\beta_1,\dots,\beta_m\in\Pi_K$ in such a way that
 the Dynkin diagram of $K$ becomes
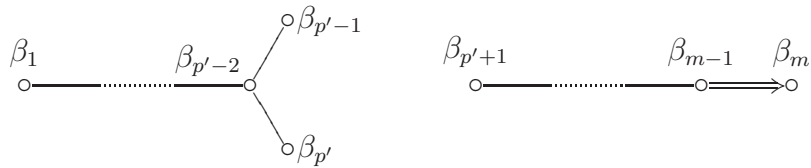
\begin{align*}
\begin{xy}
\ar@{-} (0,0) *++!D{\beta_1} *{\circ}="A"; (10,0)="B" 
\ar@{.} "B"; (20,0) 
\ar@{-} (20,0); (30,0) *+!DR{\beta_{p'-2}} *{\circ}="C"
\ar@{-} "C"; (35,8.6)  *+!L{\beta_{p'-1}} *{\circ}="D"
\ar@{-} "C"; (35,-8.6)  *+!L{\beta_{p'}} *{\circ}
\ar@{-} (60,0) *++!D{\beta_{p'+1}} *{\circ}="F"; (70,0)="G" 
\ar@{.} "G"; (80,0)="H" 
\ar@{-} "H"; (90,0) *++!D{\beta_{m-1}} *{\circ}="I"
\ar@{=>} "I"; (102,0)  *++!D{\beta_{m}} *{\circ}
\end{xy}
\end{align*}

Suppose that $J_K=\Pi_K\setminus\{\beta_i\}$ with $1\leq i\leq p'$
 and $Q=Q_{J_K}$.
We take $B$ and $T$ as in \S~\ref{soso}, so the Vogan diagram is
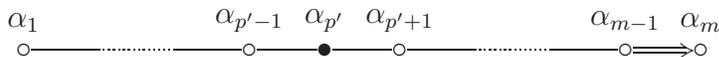
\begin{align*}
\begin{xy}
\ar@{-} (0,0) *++!D{\alpha_1} *{\circ}="A"; (10,0) 
\ar@{.} (10,0) ; (20,0) 
\ar@{-} (20,0) ; (30,0) *++!D{\alpha_{p'-1}} *{\circ}="B"
\ar@{-} "B" ; (40,0) *++!D{\alpha_{p'}} *{\bullet}="C"
\ar@{-} "C"; (50,0)*++!D{\alpha_{p'+1}}  *{\circ}="D"
\ar@{-} "D"; (60,0) 
\ar@{.} (60,0); (70,0) 
\ar@{-} (70,0); (80,0) *++!D{\alpha_{m-1}}  *{\circ}="E"
\ar@{=>} "E"; (90,0) *++!D{\alpha_m} *{\circ}
\end{xy}
\end{align*}
and we have
$\beta_j=\alpha_j$ for $j\neq p'$
 and $\beta_{p'}=\alpha_{p'-1}+2(\alpha_{p'}+\cdots+\alpha_m)$.
By putting $J':=\Pi\setminus \{\alpha_i\}$ and $P'=P_{J'}$,
 we get $P'\cap K=Q$.
In \S~\ref{soso} we saw that $(\lie{u}')^{-\theta}$ is $(L'\cap K)$-spherical
 only if $i=1$.
For $i=1$, we have $(\lie{l}',\lie{l}'\cap \lie{k})
 \simeq(\lie{gl}_1\oplus\lie{so}_{n-2},
 \lie{gl}_1\oplus\lie{so}_{p-2}\oplus\lie{so}_q)$
 and hence $\lie{m}'\simeq\lie{gl}_1\oplus \lie{so}_{|p-q-2|}$.
Here $\lie{so}_{|p-q-2|}$ is contained in the $\lie{so}_{p-2}$-component
 of $\lie{l}'\cap\lie{k}$ if $p-2\geq q$
 and in the $\lie{so}_q$-component if $p-2<q$.
Since $(\lie{u}')^{-\theta}\simeq \bbC^q$ is a natural representation
 of the $\lie{so}_q$-component in $\lie{l}'\cap \lie{k}$,
its restriction to $\lie{so}_{|p-q-2|}$ decomposes as
\begin{align*}
(\lie{u}')^{-\theta}\simeq 
\begin{cases}
\underbrace{\bbC\oplus\cdots\oplus\bbC}_{q} \text{\ if $p-2\geq q$},\\
\bbC^{-p+q+2}\oplus\underbrace{\bbC\oplus\cdots\oplus\bbC}_{p-2} \text{\ if $p-2<q$},
\end{cases}
\end{align*}
which is not $M'_0$-spherical.

Suppose that $J_K=\Pi_K\setminus\{\beta_i\}$ with $p'+1\leq i\leq m$
 and $Q=Q_{J_K}$.
We take $B$ and $T$ corresponding to the following Vogan diagram
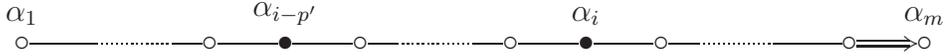
\begin{align*}
\begin{xy}
\ar@{-} (0,0) *++!D{\alpha_1} *{\circ}="A"; (10,0) 
\ar@{.} (10,0) ; (20,0) 
\ar@{-} (20,0) ; (25,0) *{\circ}="B"
\ar@{-} "B" ; (35,0) *++!D{\alpha_{i-p'}} *{\bullet}="C"
\ar@{-} "C"; (45,0) *{\circ}="D"
\ar@{-} "D"; (50,0) 
\ar@{.} (50,0) ; (60,0) 
\ar@{-} (60,0) ; (65,0) *{\circ}="E"
\ar@{-} "E" ; (75,0) *++!D{\alpha_i} *{\bullet}="F"
\ar@{-} "F" ; (85,0)  *{\circ}="G"
\ar@{-} "G" ; (90,0)
\ar@{.} (90,0); (100,0)
\ar@{-} (100,0) ; (110,0) *{\circ}="H"
\ar@{=>} "H"; (120,0) *++!D{\alpha_m} *{\circ}
\end{xy}
\end{align*}
and then 
 $\beta_j=\alpha_{i+j-p'}$ for $1\leq j\leq p'-1$, 
 $\beta_{p'}=\alpha_{i-1}+2(\alpha_i+\cdots+\alpha_m)$,
 $\beta_j=\alpha_{j-p'}$ for $p'+1\leq j\leq i-1$, 
 $\beta_i=\alpha_{i-p'}+\cdots+\alpha_i$, 
 and  $\beta_j=\alpha_j$ for $i+1 \leq j\leq m$.
By putting $J':=\Pi\setminus \{\alpha_{i-p'}\}$ and $P'=P_{J'}$,
 we get $P'\cap K=Q$.
Theorem~\ref{theorem:k.spherical}
 implies that $(\lie{u}')^{-\theta}$ is $(L'\cap K)$-spherical
 only if $i-p'=1$.
For $i-p'=1$, we have $(\lie{l}',\lie{l}'\cap \lie{k})
 \simeq (\lie{gl}_1\oplus\lie{so}_{n-2},
 \lie{gl}_1\oplus\lie{so}_{p}\oplus\lie{so}_{q-2})$
 and hence $\lie{m}'\simeq\lie{gl}_1\oplus \lie{so}_{|p-q+2|}$.
As in the previous case, we see that 
 $(\lie{u}')^{-\theta}$ is not $M'_0$-spherical.

Hence there is no $G/B\times K/Q_{J_K}$ of finite type if $p,q\geq 3$.

\subsection{$(\lie{sp}_n, \lie{sp}_{p}\oplus\lie{sp}_q)$}
\ 

Let $(\lie{g}, \lie{k})=(\lie{sp}_n, \lie{sp}_{p}\oplus\lie{sp}_q)$
 with $p+q=n$.
We fix the numbering $\beta_1,\dots,\beta_n\in\Pi_K$ in such a way that
 the Dynkin diagram of $K$ becomes
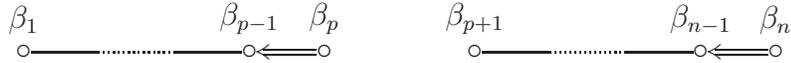
\begin{align*}
\begin{xy}
\ar@{-} (0,0) *++!D{\beta_1} *{\circ}="A"; (10,0)="C" 
\ar@{.} "C"; (20,0)="D" 
\ar@{-} "D"; (30,0) *++!D{\beta_{p-1}} *{\circ}="E"
\ar@{<=} "E"; (40,0) *++!D{\beta_p} *{\circ}="F"
\ar@{-} (60,0) *++!D{\beta_{p+1}} *{\circ}="G"; (70,0)="H" 
\ar@{.} "H"; (80,0)="I" 
\ar@{-} "I"; (90,0) *++!D{\beta_{n-1}} *{\circ}="J"
\ar@{<=} "J"; (100,0) *++!D{\beta_n} *{\circ}
\end{xy}
\end{align*}

Suppose that $J_K=\Pi_K\setminus\{\beta_i\}$ with $1\leq i\leq p$
 and $Q=Q_{J_K}$.
We take $B$ and $T$ as in \S~\ref{spsp}, so the Vogan diagram is
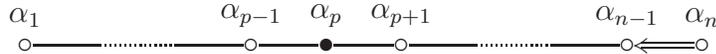
\begin{align*}
\begin{xy}
\ar@{-} (0,0) *++!D{\alpha_1} *{\circ}="A"; (10,0) 
\ar@{.} (10,0) ; (20,0) 
\ar@{-} (20,0) ; (30,0) *++!D{\alpha_{p-1}} *{\circ}="B"
\ar@{-} "B" ; (40,0) *++!D{\alpha_{p}} *{\bullet}="C"
\ar@{-} "C"; (50,0)*++!D{\alpha_{p+1}}  *{\circ}="D"
\ar@{-} "D"; (60,0) 
\ar@{.} (60,0); (70,0) 
\ar@{-} (70,0); (80,0) *++!D{\alpha_{n-1}}  *{\circ}="E"
\ar@{<=} "E"; (90,0) *++!D{\alpha_n} *{\circ}
\end{xy}
\end{align*}
and we have $\beta_j=\alpha_j$ for $j\neq p$
 and $\beta_p=2(\alpha_p+\cdots+\alpha_{n-1})+\alpha_n$.
By putting $J':=\Pi\setminus \{\alpha_i\}$ and $P'=P_{J'}$,
 we get $P'\cap K=Q$.
Put $r:=\min\{p-i,q\}$.
We have $(\lie{l}',\lie{l}'\cap \lie{k}) \simeq(\lie{gl}_i\oplus\lie{sp}_{n-i},
 \lie{gl}_i\oplus\lie{sp}_{p-i}\oplus\lie{sp}_q)$
 and hence $\lie{m}'\simeq\lie{gl}_i\oplus
 \underbrace{\lie{sl}_2\oplus\cdots\lie{sl}_2}_r \oplus \lie{sp}_{|p-q-i|}$.
Here $\lie{sl}_2$s are diagonally embedded in $\lie{sp}_{p-i}\oplus\lie{sp}_q$
 and $\lie{sp}_{|p-q-i|}$ is contained in $\lie{sp}_{p-i}$
 if $p-i\geq q$ and in $\lie{sp}_q$ if $p-i<q$.
According to \S~\ref{spsp},
 $(\lie{u}')^{-\theta}\simeq \bbC^i\otimes\bbC^{2q}$ as an $(L'\cap K)$-module,
 where $\bbC^i$ and $\bbC^{2q}$ are natural representation of
 $\lie{gl}_i$ and $\lie{sp}_{q}$, respectively.
Suppose $p-i\geq q$. Then 
\begin{align*}
(\lie{u}')^{-\theta}\simeq 
\underbrace{(\bbC^i\otimes\bbC^2)\oplus\cdots\oplus(\bbC^i\otimes\bbC^2)}_{q}
\end{align*}
as an $\lie{m}'$-module, which is spherical if and only if $i=1$ or $q=1$.
Suppose $p-i< q$. Then 
\begin{align*}
(\lie{u}')^{-\theta}\simeq 
\underbrace{(\bbC^i\otimes\bbC^2)\oplus\cdots\oplus(\bbC^i\otimes\bbC^2)}_{p-i}
 \oplus (\bbC^i\otimes\bbC^{2(-p+q+i)})
\end{align*}
as an $\lie{m}'$-module.
This is spherical if and only if at least one of the following
 three conditions holds:
\begin{itemize}
\item $i=1$,
\item $p-i=0$ and $q\leq 2$,
\item $p-i=0$ and $i\leq 3$.
\end{itemize}

Suppose that $J_K=\Pi_K\setminus\{\beta_i,\beta_j\}$ with $1\leq i<j\leq p$
 and $Q=Q_{J_K}$.
We take $B$ and $T$ as above.
By putting $J':=\Pi\setminus \{\alpha_i,\alpha_j\}$ and $P'=P_{J'}$,
 we get $P'\cap K=Q$.
Put $r:=\min\{p-j,q\}$.
We have $(\lie{l}',\lie{l}'\cap \lie{k}) \simeq
(\lie{gl}_i\oplus\lie{gl}_{j-i}\oplus\lie{sp}_{n-j},
 \lie{gl}_i\oplus\lie{gl}_{j-i}\oplus\lie{sp}_{p-j}\oplus\lie{sp}_q)$
 and hence $\lie{m}'\simeq\lie{gl}_i\oplus\lie{gl}_{j-i}\oplus
 \underbrace{\lie{sl}_2\oplus\cdots\lie{sl}_2}_r \oplus \lie{sp}_{|p-q-j|}$.
According to \S~\ref{spsp},
 $(\lie{u}')^{-\theta}\simeq (\bbC^i\otimes\bbC^{2q})
 \oplus (\bbC^{j-i}\otimes\bbC^{2q})$
 as an $(L'\cap K)$-module.
Suppose $p-j\geq q$. Then 
\begin{align*}
(\lie{u}')^{-\theta}\simeq 
\underbrace{(\bbC^i\otimes\bbC^2)\oplus\cdots\oplus(\bbC^i\otimes\bbC^2)}_{q}
\oplus
\underbrace{(\bbC^{j-i}\otimes\bbC^2)\oplus\cdots\oplus(\bbC^{j-i}\otimes\bbC^2)}_{q}
\end{align*}
as an $\lie{m}'$-module, which is spherical if and only if $q=1$.
Suppose $p-j< q$. Then 
\begin{align*}
(\lie{u}')^{-\theta}\simeq 
&\underbrace{(\bbC^i\otimes\bbC^2)\oplus\cdots\oplus(\bbC^i\otimes\bbC^2)}_{p-j}
\oplus (\bbC^i\otimes\bbC^{2(-p+q+j)})\\
& \oplus
\underbrace{(\bbC^{j-i}\otimes\bbC^2)\oplus\cdots\oplus(\bbC^{j-i}\otimes\bbC^2)}_{p-j}
\oplus (\bbC^{j-i}\otimes\bbC^{2(-p+q+j)})
\end{align*}
as an $\lie{m}'$-module, which is spherical if and only if $q=1$
 or $(i,j,p)=(1,2,2)$.

Suppose that $J_K=\Pi_K\setminus\{\beta_i,\beta_j\}$ with 
 $1\leq i\leq p < j\leq n$ and $p-i\leq n-j$.
We take $B$ and $T$ corresponding to the following Vogan diagram
\begin{align*}
\begin{xy}
\ar@{-} (0,0) *++!D{\alpha_1} *{\circ}="A"; (10,0) 
\ar@{.} (10,0) ; (20,0) 
\ar@{-} (20,0) ; (25,0) *{\circ}="B"
\ar@{-} "B" ; (35,0) *++!D{\alpha_{j-p}} *{\bullet}="C"
\ar@{-} "C"; (45,0) *{\circ}="D"
\ar@{-} "D"; (50,0) 
\ar@{.} (50,0) ; (60,0) 
\ar@{-} (60,0) ; (65,0) *{\circ}="E"
\ar@{-} "E" ; (75,0) *++!D{\alpha_j} *{\bullet}="F"
\ar@{-} "F" ; (85,0)  *{\circ}="G"
\ar@{-} "G" ; (90,0)
\ar@{.} (90,0); (100,0)
\ar@{-} (100,0) ; (110,0) *{\circ}="H"
\ar@{<=} "H"; (120,0) *++!D{\alpha_n} *{\circ}
\end{xy}
\end{align*}
and we have $\beta_k=\alpha_{k+j-p}$ for $1\leq k\leq p-1$,
 $\beta_p=2(\alpha_{j}+\cdots+\alpha_{n-1})+\alpha_n$,
 $\beta_k=\alpha_{k-p}$ for $p+1\leq k\leq j-1$,
 $\beta_j=\alpha_{j-p}+\cdots+\alpha_j$,
 and $\beta_k=\alpha_k$ for $j+1\leq k\leq n$.
By putting $J':=\Pi\setminus \{\alpha_{j-p},\alpha_{i+j-p}\}$
 and $P'=P_{J'}$, we get $P'\cap K=Q$.
Then $(\lie{l}',\lie{l}'\cap \lie{k})
 \simeq (\lie{gl}_i\oplus\lie{gl}_{j-p}\oplus\lie{sp}_{n-(i+j-p)},
 \lie{gl}_i\oplus\lie{gl}_{j-p}\oplus\lie{sp}_{p-i}\oplus\lie{sp}_{n-j})$
 and $\lie{m}'\simeq\lie{gl}_i\oplus \lie{gl}_{j-p}\oplus
 \underbrace{\lie{sl}_2\oplus\cdots\oplus\lie{sl}_2}_{p-i}
 \oplus\lie{sp}_{q+i-j}$. 
Then we can compute that
\[(\lie{u}')^{-\theta}\simeq (\bbC^i\otimes \bbC^{j-p})
 \oplus ((\bbC^i)^*\otimes \bbC^{j-p})
 \oplus (\bbC^i \otimes\bbC^{2(n-j)}) \oplus (\bbC^{j-p}\otimes \bbC^{2(p-i)})\]
as an $(L'\cap K)$-module.
When restricted to $\lie{m}'$ it decomposes as 
\begin{align*}
&(\bbC^i\otimes \bbC^{j-p})
 \oplus ((\bbC^i)^*\otimes \bbC^{j-p})
 \oplus \underbrace{(\bbC^i\otimes \bbC^2)\oplus \cdots 
 \oplus(\bbC^i\otimes\bbC^2)}_{p-i}\\
&\quad  \oplus (\bbC^i \otimes\bbC^{2(q+i-j)}) 
 \oplus \underbrace{(\bbC^{j-p}\otimes \bbC^2)\oplus \cdots
 \oplus(\bbC^{j-p}\otimes\bbC^2)}_{p-i},
\end{align*}
which is spherical if and only if
\begin{itemize}
\item $p-i=0$ and $i=1$, or 
\item $p-i=0$, $q+i-j=0$, and $j-p=1$.
\end{itemize}
This is equivalent to $\min\{p,q\}=1$ under our assumption.

Consequently, $G/B\times K/Q_{J_K}$ is of finite type if and only if
 at least one of the following holds:
\begin{itemize}
\item $\Pi_K\setminus J_K=\{\beta_1\}$,
\item $\Pi_K\setminus J_K=\{\beta_p\}$ and $q\leq 2$,
\item $\Pi_K\setminus J_K=\{\beta_p\}$ and $p\leq 3$,
\item $\Pi_K\setminus J_K=\{\beta_q\}$ and $p\leq 2$,
\item $\Pi_K\setminus J_K=\{\beta_q\}$ and $q\leq 3$,
\item $\Pi_K\setminus J_K=\{\beta_1,\beta_2\}$ and $p=2$,
\item $\Pi_K\setminus J_K=\{\beta_{p+1},\beta_{p+2}\}$ and $q=2$,
\item $|\Pi_K\setminus J_K|\leq 2$ and $\min\{p,q\}=1$.
\end{itemize}

\subsection{$ \lie{g}_2, \lie{e}_8 $ and  $(\lie{f}_{4}, \lie{sp}_3\oplus\lie{sp}_1)$}
\

For $(\lie{g}, \lie{k})=(\lie{g}_2, \lie{sl}_2\oplus\lie{sl}_2)$,
 $(\lie{f}_{4}, \lie{sp}_3\oplus\lie{sp}_1)$,
 $(\lie{e}_8, \lie{so}_{16})$, and $(\lie{e}_8, \lie{e}_7\oplus\lie{sl}_2)$,
 there does not exist the double flag variety $G/P\times K/B_K$
 of finite type for $P\subsetneq G$.
Therefore, $G/B\times K/Q$ cannot be of finite type for $Q\subsetneq K$
 by Corollary~\ref{corollary:P=B.to.Q=B_K}.

\subsection{$(\lie{f}_{4}, \lie{so}_9)$}
\ 

Let $(\lie{g}, \lie{k})=(\lie{f}_4, \lie{so}_9)$.
We fix the numbering $\beta_1,\beta_2,\beta_3,\beta_4\in\Pi_K$
 in such a way that the Dynkin diagram of $K$ becomes
\begin{align*}
\begin{xy}
\ar@{-} (0,0) *++!D{\beta_1} *{\circ}="A"; (10,0) *++!D{\beta_2} *{\circ}="D"
\ar@{-} "D"; (20,0) *++!D{\beta_3} *{\circ}="E"
\ar@{=>} "E"; (30,0) *++!D{\beta_4} *{\circ}="F"
\end{xy}
\end{align*}

Suppose that $J_K=\Pi_K\setminus\{\beta_1,\beta_2\}$ and $Q=Q_{J_K}$.
We take $B$ and $T$ as in \S~\ref{f4so}, so the Vogan diagram is
\begin{align*}
\begin{xy}
\ar@{-} (0,0) *++!D{\alpha_1} *{\circ}="A";
  (10,0) *++!D{\alpha_2} *{\circ}="B"
\ar@{=>} "B" ; (20,0)  *++!D{\alpha_3} *{\circ}="C"
\ar@{-} "C" ; (30,0) *++!D{\alpha_4} *{\bullet}
\end{xy}
\end{align*}
 and we have $\beta_1=\alpha_2+2\alpha_3+2\alpha_4$,
 $\beta_2=\alpha_1$, $\beta_3=\alpha_2$, $\beta_4=\alpha_3$.
By putting $J':=\Pi\setminus \{\alpha_1,\alpha_4\}$ and $P'=P_{J'}$,
 we get $P'\cap K=Q$.
We saw in \S~\ref{f4so} that $(\lie{u}')^{-\theta}$ is $(L'\cap K)$-spherical.
Since $\lie{l}'=\lie{l}'\cap \lie{k}\simeq
 \lie{gl}_1\oplus\lie{sp}_2\oplus\lie{gl}_1$
 we have $\lie{l}'\cap \lie{k}=\lie{m}'$.
Hence $(\lie{u}')^{-\theta}$ is $M'_0$-spherical as well.

Suppose that $J_K=\Pi_K\setminus\{\beta_3\}$ and $Q=Q_{J_K}$.
We take $B$ and $T$ corresponding to the following Vogan diagram 
\begin{align*}
\begin{xy}
\ar@{-} (0,0) *++!D{\alpha_1} *{\circ}="A";
  (10,0) *++!D{\alpha_2} *{\circ}="B"
\ar@{=>} "B" ; (20,0)  *++!D{\alpha_3} *{\bullet}="C"
\ar@{-} "C" ; (30,0) *++!D{\alpha_4} *{\circ}
\end{xy}
\end{align*}
 and we have $\beta_1=\alpha_2$, $\beta_2=\alpha_1$,
 $\beta_3=\alpha_2+2\alpha_3$, $\beta_4=\alpha_4$.
By putting $J':=\Pi\setminus \{\alpha_3\}$ and $P'=P_{J'}$,
 we get $P'\cap K=Q$.
Theorem~\ref{theorem:k.spherical}
 implies that $(\lie{u}')^{-\theta}$ is $(L'\cap K)$-spherical.
Since $\lie{l}'=\lie{l}'\cap \lie{k}\simeq
 \lie{gl}_3\oplus\lie{sl}_2$ we have $\lie{l}'\cap \lie{k}=\lie{m}'$
 and $(\lie{u}')^{-\theta}$ is $M'_0$-spherical.

Suppose that $J_K=\Pi_K\setminus\{\beta_4\}$ and $Q=Q_{J_K}$.
We take $B$ and $T$ corresponding to the following Vogan diagram 
\begin{align*}
\begin{xy}
\ar@{-} (0,0) *++!D{\alpha_1} *{\circ}="A";
  (10,0) *++!D{\alpha_2} *{\circ}="B"
\ar@{=>} "B" ; (20,0)  *++!D{\alpha_3} *{\bullet}="C"
\ar@{-} "C" ; (30,0) *++!D{\alpha_4} *{\bullet}
\end{xy}
\end{align*}
 and we have $\beta_1=\alpha_2+2\alpha_3$, $\beta_2=\alpha_1$,
 $\beta_3=\alpha_2$, $\beta_4=\alpha_3+\alpha_4$.
By putting $J':=\Pi\setminus \{\alpha_4\}$ and $P'=P_{J'}$,
 we get $P'\cap K=Q$.
We have $(\lie{l}',\lie{l}'\cap \lie{k})\simeq
 (\lie{so}_7\oplus\lie{gl}_1,\lie{so}_6\oplus\lie{gl}_1)$
 and hence $\lie{m}'\simeq \lie{so}_5\oplus\lie{gl}_1$.
We can compute that $(\lie{u}')^{-\theta}\simeq \bbC^4\oplus \bbC$ as
 an $(\lie{l}'\cap \lie{k})$-module, where $\lie{so}_6$ acts 
 on the first factor $\bbC^4$ as a spin representation
 and $\lie{gl}_1$ acts as non-zero scalar on the second factor $\bbC$.
Therefore, $(\lie{u}')^{-\theta}$ is $M'_0$-spherical.

If $J_K=\Pi_K\setminus\{\beta_i,\beta_j\}$ and $\{i,j\}\neq \{1,2\}$,
 then $\dim L'_K\leq 10$.
Hence the dimension condition 
 $\dim L'_K\geq \dim G-\dim K-\rank G=52-36-4=12$
 does not hold.

We conclude that 
 $G/B\times K/Q_{J_K}$ is of finite type if and only if
$|\Pi_K\setminus J_K|=1$ or $\Pi_K\setminus J_K=\{\beta_1,\beta_2\}$.

\subsection{$(\lie{e}_{6}, \lie{sp}_4)$}
\ 

Let $(\lie{g}, \lie{k})=(\lie{e}_6, \lie{sp}_4)$.
Then the dimension condition is
 $\dim L'_K\geq \dim G-\dim K-\rank G= 78-36-6=36=\dim K$.
This does not hold for $Q\subsetneq K$.

\subsection{$(\lie{e}_{6}, \lie{sl}_6\oplus\lie{sl}_2)$}
\ 

Let $(\lie{g}, \lie{k})=(\lie{e}_6, \lie{sl}_6\oplus\lie{sl}_2)$.
We fix the numbering $\beta_1,\dots,\beta_6\in\Pi_K$ in such a way that
 the Dynkin diagram of $K$ becomes
\begin{align*}
\begin{xy}
\ar@{-} (0,0) *++!D{\beta_1} *{\circ}="F"; (10,0) *++!D{\beta_2} 
 *{\circ}="G"
\ar@{-} "G"; (20,0)*++!D{\beta_3}  *{\circ}="H"
\ar@{-} "H"; (30,0)*++!D{\beta_4}  *{\circ}="I"
\ar@{-} "I"; (40,0)*++!D{\beta_5}  *{\circ}="J"
\ar@{} "J"; (50,0) *++!D{\beta_6} *{\circ}
\end{xy}
\end{align*}

Suppose that $J_K=\Pi_K\setminus\{\beta_6\}$ and $Q=Q_{J_K}$.
We take $B$ and $T$ as in \S~\ref{e6sl}, so the Vogan diagram is
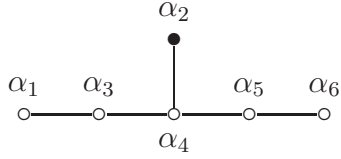
\begin{align*}
\begin{xy}
\ar@{-} (0,0) *++!D{\alpha_1} *{\circ}="A"; (10,0) *++!D{\alpha_3} 
 *{\circ}="B"
\ar@{-} "B"; (20,0)*++!U{\alpha_4}  *{\circ}="C"
\ar@{-} "C"; (30,0) *++!D{\alpha_5}  *{\circ}="D"
\ar@{-} "C"; (20,10) *++!D{\alpha_2}  *{\bullet}="G"
\ar@{-} "D"; (40,0) *++!D{\alpha_6}  *{\circ}="E"
\end{xy}
\end{align*}
and we have $\beta_1=\alpha_1$, $\beta_2=\alpha_3$, $\beta_3=\alpha_4$,
 $\beta_4=\alpha_5$, $\beta_5=\alpha_6$, and
 $\beta_6=\alpha_1+2\alpha_2+2\alpha_3+3\alpha_4+2\alpha_5+\alpha_6$.
By putting $J':=\Pi\setminus \{\alpha_2\}$ and $P'=P_{J'}$,
 we get $P'\cap K=Q$.
We saw in \S~\ref{e6sl} that $(\lie{u}')^{-\theta}$ is not $(L'\cap K)$-spherical.

If $\Pi_K\setminus J_K=\{\beta_i\}$ for $1\leq i\leq 5$,
 then $\dim L'_K\leq 28$.
It does not satisfy the dimension condition 
 $\dim L'_K\geq \dim G-\dim K-\rank G= 78-38-6=34$.

Hence there is no $G/B\times K/Q_{J_K}$ of finite type.

\subsection{$(\lie{e}_{6}, \lie{f}_4)$}
\ 

Let $(\lie{g}, \lie{k})=(\lie{e}_6, \lie{f}_4)$.
We fix the numbering $\beta_1,\beta_2,\beta_3,\beta_4\in\Pi_K$
 in such a way that the Dynkin diagram of $K$ becomes
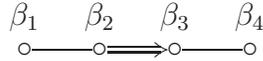
\begin{align*}
\begin{xy}
\ar@{-} (0,0) *++!D{\beta_1} *{\circ}="F"; (10,0) *++!D{\beta_2} 
 *{\circ}="G"
\ar@{=>} "G"; (20,0)*++!D{\beta_3}  *{\circ}="H"
\ar@{-} "H"; (30,0)*++!D{\beta_4}  *{\circ}="I"
\end{xy}
\end{align*}

Suppose that $J_K=\Pi_K\setminus\{\beta_1\}$ and $Q=Q_{J_K}$.
We take $B$ and $T$ as in \S~\ref{e6f4}, so the Vogan diagram is
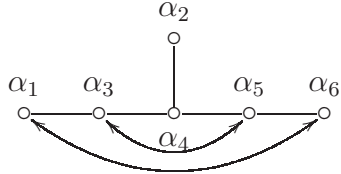
\begin{align*}
\begin{xy}
\ar@{-} (0,0) *++!D{\alpha_1} *{\circ}="A"; (10,0) *++!D{\alpha_3} 
 *{\circ}="B"
\ar@{-} "B"; (20,0)*++!U{\alpha_4}  *{\circ}="C"
\ar@{-} "C"; (30,0) *++!D{\alpha_5}  *{\circ}="D"
\ar@{-} "C"; (20,10) *++!D{\alpha_2}  *{\circ}="G"
\ar@{-} "D"; (40,0) *++!D{\alpha_6}  *{\circ}="E"
\ar@/_1.2pc/@{<->} "B"; "D"
\ar@/_1.8pc/@{<->} "A"; "E"
\end{xy}
\end{align*}
and we have $\beta_1=\alpha_2|_{\lie{t}^\theta}$,
 $\beta_2=\alpha_4|_{\lie{t}^\theta}$, $\beta_3=\alpha_3|_{\lie{t}^\theta}$,
 and $\beta_4=\alpha_1|_{\lie{t}^\theta}$.
By putting $J':=\Pi\setminus \{\alpha_2\}$ and $P'=P_{J'}$,
 we get $P'\cap K=Q$.
We have $(\lie{l}',\lie{l}'\cap \lie{k}) \simeq
(\lie{sl}_6\oplus\lie{gl}_1, \lie{sp}_3\oplus\lie{gl}_1)$
 and hence
 $\lie{m}'\simeq\lie{sl}_2\oplus\lie{sl}_2\oplus\lie{sl}_2\oplus \lie{gl}_1$.
We can compute 
 $(\lie{u}')^{-\theta}\simeq \bbC^6$
 is a natural representation of $\lie{sp}_3$ in $\lie{l}'\cap \lie{k}$.
Hence $(\lie{u}')^{-\theta}$ decomposes as $\bbC^2\oplus \bbC^2\oplus\bbC^2$
 as an $\lie{m}'$-module, where each $\lie{sl}_2$ in $\lie{m}'$
 acts naturally on one $\bbC^2$ and trivially on the other two.
This is a spherical action.

Suppose that $J_K=\Pi_K\setminus\{\beta_4\}$ and $Q=Q_{J_K}$.
We take $B$ and $T$ as above. Putting
 $J':=\Pi\setminus \{\alpha_1,\alpha_6\}$ and $P'=P_{J'}$,
 we get $P'\cap K=Q$.
We saw in \S~\ref{e6f4} that $(\lie{u}')^{-\theta}$ is not $(L'\cap K)$-spherical.

If $\Pi_K\setminus J_K=\{\beta_i\}$ for $i=2$ or $3$,
 then $\dim L'_K=12$.
It does not satisfy the dimension condition 
 $\dim L'_K\geq \dim G-\dim K-\rank G= 78-52-6=20$.

Hence $G/B\times K/Q_{J_K}$ is of finite type
 if and only if $J_K=\Pi_K\setminus\{\beta_1\}$.

\subsection{$(\lie{e}_{7}, \lie{sl}_8)$}
\ 

Let $(\lie{g}, \lie{k})=(\lie{e}_7, \lie{sl}_8)$.
Then the dimension condition is
 $\dim L'_K\geq \dim G-\dim K-\rank G= 133-63-7=63=\dim K$.
This does not hold for $Q\subsetneq K$.

\subsection{$(\lie{e}_{7}, \lie{so}_{12}\oplus\lie{sl}_2)$}
\

Let $(\lie{g}, \lie{k})=(\lie{e}_7, \lie{so}_{12}\oplus\lie{sl}_2)$.
We fix the numbering $\beta_1,\dots,\beta_7\in\Pi_K$ in such a way that
 the Dynkin diagram of $K$ becomes
\begin{align*}
\begin{xy}
\ar@{-} (0,0) *++!D{\beta_1} *{\circ}="F"; (10,0) *++!D{\beta_2} 
 *{\circ}="H"
\ar@{-} "H" ; (20,0)  *++!D{\beta_3}  *{\circ}="I"
\ar@{-} "I" ; (30,0) *+!DR{\beta_4} *{\circ}="G" 
\ar@{-} "G" ; (35,8.6)  *+!L{\beta_5} *{\circ}
\ar@{-} "G" ; (35,-8.6)  *+!L{\beta_6} *{\circ}
\ar@{} "G" ; (50,0) *++!D{\beta_7} *{\circ}
\end{xy}
\end{align*}

Suppose that $J_K=\Pi_K\setminus\{\beta_7\}$ and $Q=Q_{J_K}$.
We take $B$ and $T$ as in \S~\ref{e7so}, so the Vogan diagram is
\begin{align*}
\begin{xy}
\ar@{-} (0,0) *++!D{\alpha_1} *{\bullet}="A"; (10,0) *++!D{\alpha_3} 
 *{\circ}="B"
\ar@{-} "B"; (20,0)*++!U{\alpha_4}  *{\circ}="C"
\ar@{-} "C"; (30,0) *++!D{\alpha_5}  *{\circ}="D"
\ar@{-} "C"; (20,10) *++!D{\alpha_2}  *{\circ}
\ar@{-} "D"; (40,0) *++!D{\alpha_6}  *{\circ}="E"
\ar@{-} "E"; (50,0) *++!D{\alpha_7}  *{\circ}
\end{xy}
\end{align*}
 and we have $\beta_i=\alpha_{8-i}$ for $1\leq i\leq 6$,
 $\beta_7=2\alpha_1+2\alpha_2+3\alpha_3
 +4\alpha_4+3\alpha_5+2\alpha_6+\alpha_7$.
By putting $J':=\Pi\setminus \{\alpha_1\}$ and $P'=P_{J'}$,
 we get $P'\cap K=Q$.
We saw in \S~\ref{e7so} that $(\lie{u}')^{-\theta}$ is not $(L'\cap K)$-spherical.

If $\Pi_K\setminus J_K=\{\beta_i\}$ for $i\neq 7$,
 then $\dim L'_K\leq 49$.
It does not satisfy the dimension condition 
 $\dim L'_K\geq \dim G-\dim K-\rank G=133-69-7=57$.

Hence there is no $G/B\times K/Q_{J_K}$ of finite type.

\bigskip

We thus conclude that:

\begin{theorem}
\label{theorem:classification.P=B}
Let $G$ be a connected simple algebraic group
 and $(G,K)$ a symmetric pair.
Let $Q$ be a parabolic subgroup of $K$
 corresponding to $J_K \subsetneq \Pi_K$.
Then the double flag variety $ G/B \times K/Q $
 is of finite type 
 if and only if the triple
 $(\lie{g}, \lie{k}, \Pi_K \setminus J_K)$
 appears in Table~\ref{table:g.spherical}.
\end{theorem}

\begin{remark}
\label{remark:labeling}
In Table~\ref{table:g.spherical},
 the labeling $\beta_1,\beta_2,\dots$ of the simple roots for $\lie{k}$
 is not unique up to isomorphisms in the cases
 where $(\lie{g},\lie{k})=(\lie{sl}_{p+q+2}, \lie{sl}_{p+1}
 \oplus \lie{sl}_{q+1} \oplus \bbC)$ with $1\leq p\leq q$,
 $(\lie{so}_{2n+2}, \lie{so}_{2n} \oplus \bbC)$ with $n=4$,
 and $(\lie{so}_{2n+1}, \lie{so}_{2n})$ with $n=4$.
In order to fix this, we describe the $[\lie{k},\lie{k}]$-module
 $\lie{g}^{-\theta}$ in terms of $\beta_i$.
We denote by $\omega_i \in (\lie{t}^\theta\cap[\lie{k},\lie{k}])^*$
 the fundamental weight for $[\lie{k},\lie{k}]$
 corresponding to $\beta_i$.
Write $V(\lambda)$ for the irreducible $[\lie{k},\lie{k}]$-module
 with highest weight $\lambda$.

For $(\lie{g},\lie{k})=(\lie{sl}_{p+q+2}, \lie{sl}_{p+1}
 \oplus \lie{sl}_{q+1} \oplus \bbC)$ with $0\leq p\leq q$
 in Table~\ref{table:g.spherical},
 we label $\beta_i$ so that
 $\lie{g}^{-\theta}\simeq V(\omega_1+\omega_{p+q})
 \oplus V(\omega_p+\omega_{p+1})$ if $p>0$ and 
 $\lie{g}^{-\theta}\simeq V(\omega_1) \oplus V(\omega_q)$ if $p=0$.

For $(\lie{g},\lie{k})=(\lie{so}_{2n+2}, \lie{so}_{2n} \oplus \bbC)$
 in Table~\ref{table:g.spherical},
 we label $\beta_i$ so that
 $\lie{g}^{-\theta}\simeq V(\omega_1)\oplus V(\omega_1)$.

For $(\lie{g},\lie{k})=(\lie{so}_{2n+1}, \lie{so}_{2n})$
 in Table~\ref{table:g.spherical},
 we label $\beta_i$ so that
 $\lie{g}^{-\theta}\simeq V(\omega_1)$.
\end{remark}

\begin{longtable}{c|c|c}
\hline
\rule[0pt]{0pt}{12pt}
$\lie{g}$ & $\lie{k}$ & $\Pi_K \setminus J_K\ (Q=Q_{J_K})$
\\
\hline\hline

\rule[0pt]{0pt}{12pt}
\rule[-5pt]{0pt}{0pt}
$\lie{sl}_{2n}$ & $\lie{sp}_n$ & 
\begin{xy}
\ar@{-} (0,0) *++!D{\beta_1} *{\circ}="A"; (10,0) *++!D{\beta_2} 
 *{\circ}="B"
\ar@{-} "B"; (20,0)="C" 
\ar@{.} "C"; (30,0)="D" 
\ar@{-} "D"; (40,0) *++!D{\beta_{n-1}} *{\circ}="E"
\ar@{<=} "E"; (50,0) *++!D{\beta_n} *{\circ}="F"
\end{xy}
\\

\rule[0pt]{0pt}{12pt}
 $n\geq 2$ &   & $\{\beta_1\}$, \\

\rule[0pt]{0pt}{12pt}
  &  & $\{\beta_3\}$ if $n=3$, \\

\rule[0pt]{0pt}{12pt}
\rule[-5pt]{0pt}{0pt}
  &  & any subset of $\Pi_K$ if $n=2$ \\

\hline

\rule[0pt]{0pt}{12pt}
$\lie{sl}_{p+q+2}$ & $\lie{sl}_{p+1} \oplus \lie{sl}_{q+1} \oplus \bbC$ &
\begin{xy}
\ar@{-} (0,0) *++!D{\beta_1} *{\circ}="A"; (10,0) *++!D{\beta_2} 
 *{\circ}="B"
\ar@{-} "B"; (20,0)="C" 
\ar@{.} "C"; (30,0)="D" 
\ar@{-} "D"; (40,0) *++!D{\beta_p} *{\circ}="E"
\end{xy}
\\

\rule[0pt]{0pt}{12pt}
\rule[-5pt]{0pt}{0pt}
$p+q\geq 1$ & $0\leq p\leq q$ &  
\begin{xy}
\ar@{-} (0,10) *++!D{\beta_{p+1}} *{\circ}="A"; (10,10) *++!D{\beta_{p+2}} 
 *{\circ}="B"
\ar@{-} "B"; (20,10)="C" 
\ar@{.} "C"; (30,10)="D" 
\ar@{-} "D"; (40,10) *++!D{\beta_{p+q}} *{\circ}="E"
\end{xy}
\\

\rule[0pt]{0pt}{12pt}
  &  &  
 $\{\beta_1\}$, $\{\beta_p\}$,  $\{\beta_{p+1}\}$, $\{\beta_{p+q}\}$, \\

\rule[0pt]{0pt}{12pt}
  &  &  
 $\{\beta_i\} (\forall i)$ if $p=1$, \\

\rule[0pt]{0pt}{12pt}
\rule[-5pt]{0pt}{0pt}
  &  &  
 any subset of $\Pi_K$ if $p=0$ \\

\hline

\rule[0pt]{0pt}{12pt}
\rule[-5pt]{0pt}{0pt}
$\lie{so}_{2n+2}$ & $\lie{so}_{2n}\oplus\bbC$ &
\begin{xy}
\ar@{-} (0,0) *++!D{\beta_1} *{\circ}="A"; (10,0)="C" 
\ar@{.} "C"; (20,0)="D" 
\ar@{-} "D"; (30,0) *+!DR{\beta_{n-2}} *{\circ}="E"
\ar@{-} "E"; (35,8.6)  *+!L{\beta_{n-1}} *{\circ}
\ar@{-} "E"; (35,-8.6)  *+!L{\beta_n} *{\circ}
\end{xy}
\\

\rule[0pt]{0pt}{12pt}
\rule[-5pt]{0pt}{0pt}
$n\geq 3$  &  & $\{\beta_{n-1}\}$, $\{\beta_n\}$
\\
\hline

\rule[0pt]{0pt}{12pt}
\rule[-5pt]{0pt}{0pt}
$\lie{so}_{2n+1}$ & $\lie{so}_{2n}$ &
\begin{xy}
\ar@{-} (0,0) *++!D{\beta_1} *{\circ}="A"; (10,0)="C" 
\ar@{.} "C"; (20,0)="D" 
\ar@{-} "D"; (30,0) *+!DR{\beta_{n-2}} *{\circ}="E"
\ar@{-} "E"; (35,8.6)  *+!L{\beta_{n-1}} *{\circ}
\ar@{-} "E"; (35,-8.6)  *+!L{\beta_n} *{\circ}
\end{xy}
\\

\rule[0pt]{0pt}{12pt}
\rule[-5pt]{0pt}{0pt}
$n\geq 3$  & & 
 any subset of $\Pi_K$
\\
\hline

\rule[0pt]{0pt}{12pt}
\rule[-5pt]{0pt}{0pt}
$\lie{so}_{2n+2}$ & $\lie{so}_{2n+1}$ &
\begin{xy}
\ar@{-} (0,0) *++!D{\beta_1} *{\circ}="A"; (10,0)="C" 
\ar@{.} "C"; (20,0)="D" 
\ar@{-} "D"; (30,0) *++!D{\beta_{n-1}} *{\circ}="E"
\ar@{=>} "E"; (40,0) *++!D{\beta_n} *{\circ}="F"
\end{xy}
\\

\rule[0pt]{0pt}{12pt}
\rule[-5pt]{0pt}{0pt}
$n\geq 3$ & & 
 any subset of $\Pi_K$
\\
\hline

\rule[0pt]{0pt}{12pt}
\rule[-5pt]{0pt}{0pt}
$\lie{so}_{2n+2}$ & $\lie{sl}_{n+1} \oplus \bbC$ &
\begin{xy}
\ar@{-} (0,0) *++!D{\beta_1} *{\circ}="A"; (10,0) *++!D{\beta_2} 
 *{\circ}="B"
\ar@{-} "B"; (20,0)="C" 
\ar@{.} "C"; (30,0)="D" 
\ar@{-} "D"; (40,0) *++!D{\beta_{n}} *{\circ}="E"
\end{xy}
\\

\rule[0pt]{0pt}{12pt}
\rule[-5pt]{0pt}{0pt}
$n\geq 3$ & & 
 $\{\beta_1\}$, $\{\beta_n\}$
\\
\hline

\rule[0pt]{0pt}{12pt}
$\lie{sp}_{p+q}$ & $\lie{sp}_{p} \oplus \lie{sp}_{q}$ &
\begin{xy}
\ar@{-} (0,0) *++!D{\beta_1} *{\circ}="A"; (10,0)="C" 
\ar@{.} "C"; (20,0)="D" 
\ar@{-} "D"; (30,0) *++!D{\beta_{p-1}} *{\circ}="E"
\ar@{<=} "E"; (40,0) *++!D{\beta_p} *{\circ}="F"
\end{xy}
\\

\rule[0pt]{0pt}{12pt}
\rule[-5pt]{0pt}{0pt}
&$1\leq p\leq q$& 
\begin{xy}
\ar@{-} (0,0) *++!D{\beta_{p+1}} *{\circ}="A"; (10,0)="C" 
\ar@{.} "C"; (20,0)="D" 
\ar@{-} "D"; (30,0) *++!D{\beta_{p+q-1}} *{\circ}="E"
\ar@{<=} "E"; (40,0) *++!D{\beta_{p+q}} *{\circ}="F"
\end{xy}
\\

\rule[0pt]{0pt}{12pt}
  &  &  
 $\{\beta_1\}$, $\{\beta_{p+1}\}$,\\

\rule[0pt]{0pt}{12pt}
  &  &  
 $\{\beta_p\}$ if $p\leq 3$,\ \ $\{\beta_{p+q}\}$ if $p\leq 2$,
 \ \  $\{\beta_{p+q}\}$ if $q\leq 3$, \\

\rule[0pt]{0pt}{12pt}
  &  &  
 $\{\beta_1, \beta_2\}$ if $p=2$,
 \ \ $\{\beta_{p+1}, \beta_{p+2}\}$ if $q=2$, \\

\rule[0pt]{0pt}{12pt}
  &  &  
 $\{\beta_i\} (\forall i)$ if $p=1$, 
 \ \ $\{\beta_i, \beta_j\} (\forall i, j)$ if $p=1$ \\

\hline

\rule[0pt]{0pt}{12pt}
\rule[-5pt]{0pt}{0pt}
$\lie{f}_4$ & $\lie{so}_9$ & 
\begin{xy}
\ar@{-} (0,0) *++!D{\beta_1} *{\circ}="A"; (10,0) *++!D{\beta_2} 
 *{\circ}="B"
\ar@{-} "B"; (20,0)*++!D{\beta_3}  *{\circ}="C"
\ar@{=>} "C"; (30,0) *++!D{\beta_4}  *{\circ}="D"
\end{xy}
\\

\rule[0pt]{0pt}{12pt}
\rule[-5pt]{0pt}{0pt}
  &  &  
 $\{\beta_i\} (\forall i)$,
 $\{\beta_1, \beta_2\}$ \\

\hline

\rule[0pt]{0pt}{12pt}
\rule[-5pt]{0pt}{0pt}
$\lie{e}_6$ & $\lie{so}_{10} \oplus \bbC$ &
\begin{xy}
\ar@{-} (0,0) *++!DR{\beta_1} *{\circ}="A"; (10,0) *+!DR{\beta_2} 
 *{\circ}="B"
\ar@{-} "B"; (20,0) *+!DR{\beta_3} *{\circ}="C" 
\ar@{-} "C"; (25,8.6)  *+!L{\beta_4} *{\circ}
\ar@{-} "C"; (25,-8.6)  *+!L{\beta_5} *{\circ}
\end{xy}
\\

\rule[0pt]{0pt}{12pt}
\rule[-5pt]{0pt}{0pt}
  &  & 
 $\{\beta_1\}$\\

\hline

\rule[0pt]{0pt}{12pt}
\rule[-5pt]{0pt}{0pt}
$\lie{e}_6$ & $\lie{f}_4$ &
\begin{xy}
\ar@{-} (0,0) *++!D{\beta_1} *{\circ}="A"; (10,0) *++!D{\beta_2} 
 *{\circ}="B"
\ar@{=>} "B"; (20,0)*++!D{\beta_3}  *{\circ}="C"
\ar@{-} "C"; (30,0) *++!D{\beta_4}  *{\circ}="D"
\end{xy}
\\

\rule[0pt]{0pt}{12pt}
\rule[-5pt]{0pt}{0pt}
  &  & 
 $\{\beta_1\}$\\

\hline

\caption{$G$-spherical $G/Q$}
\label{table:g.spherical}
\vphantom{$\Biggm|$}
\end{longtable}

\bibliographystyle{amsalpha}

\begin{thebibliography}{KNOT13}

\bibitem[BH00]{Brion.Helminck.2000}
Michel Brion and Aloysius~G. Helminck, \emph{On orbit closures of symmetric
  subgroups in flag varieties}, Canad. J. Math. \textbf{52} (2000), no.~2,
  265--292. 

\bibitem[BR96]{Benson.Ratcliff.1996}
Chal Benson and Gail Ratcliff, \emph{A classification of multiplicity free
  actions}, J. Algebra \textbf{181} (1996), no.~1, 152--186. 

\bibitem[Bri86]{Brion.MM.1986}
Michel Brion, \emph{Quelques propri\'et\'es des espaces homog\`enes
  sph\'eriques}, Manuscripta Math. \textbf{55} (1986), no.~2, 191--198.
 

\bibitem[Car85]{Carter.1985}
Roger~W. Carter, \emph{Finite groups of {L}ie type}, Pure and Applied
  Mathematics (New York), John Wiley \& Sons Inc., New York, 1985, Conjugacy
  classes and complex characters, A Wiley-Interscience Publication. 

\bibitem[FG10]{FG}
Michael Finkelberg and Victor Ginzburg, \emph{On mirabolic {$D$}-modules}, Int.
  Math. Res. Not. IMRN (2010), no.~15, 2947--2986. 

\bibitem[FGT09]{FGT.2009}
Michael Finkelberg, Victor Ginzburg, and Roman Travkin, \emph{Mirabolic affine
  {G}rassmannian and character sheaves}, Selecta Math. (N.S.) \textbf{14}
  (2009), no.~3-4, 607--628. 

\bibitem[HT12]{Henderson.Trapa.2012}
Anthony Henderson and Peter~E. Trapa, \emph{The exotic {R}obinson-{S}chensted
  correspondence}, J. Algebra \textbf{370} (2012), 32--45. 

\bibitem[HW93]{Helminck.Wang.1993}
A.~G. Helminck and S.~P. Wang, \emph{On rationality properties of involutions
  of reductive groups}, Adv. Math. \textbf{99} (1993), no.~1, 26--96.


\bibitem[Kac80]{Kac.1980}
V.~G. Kac, \emph{Some remarks on nilpotent orbits}, J. Algebra \textbf{64}
  (1980), no.~1, 190--213. 

\bibitem[Kat09]{Kato.2009}
Syu Kato, \emph{An exotic {D}eligne-{L}anglands correspondence for symplectic
  groups}, Duke Math. J. \textbf{148} (2009), no.~2, 305--371. 

\bibitem[Kna02]{Knapp.2002}
Anthony~W. Knapp, \emph{Lie groups beyond an introduction}, second ed.,
  Progress in Mathematics, vol.~140, Birkh\"auser Boston Inc., Boston, MA,
  2002. 

\bibitem[KNOT13]{KNOT.2012}
Kensuke {Kondo}, Kyo {Nishiyama}, Hiroyuki {Ochiai}, and Kenji {Taniguchi},
  \emph{{Closed orbits on partial flag varieties and double flag variety of
  finite type}}, Kyushu J. Math. \textbf{67} (2013), no.~2.

\bibitem[Lea98]{Leahy.1998}
Andrew~S. Leahy, \emph{A classification of multiplicity free representations},
  J. Lie Theory \textbf{8} (1998), no.~2, 367--391. 

\bibitem[Lit94]{Littelmann.1994}
Peter Littelmann, \emph{On spherical double cones}, J. Algebra \textbf{166}
  (1994), no.~1, 142--157. 

\bibitem[Lus]{Lu}
George Lusztig, \emph{Character sheaves. {I--V}}, Adv. in Math., \textbf{56}
  (1985), 193-237; \textbf{57} (1985), 226-265; \textbf{57} (1985), 266-315;
  \textbf{59} (1986), 1-63; \textbf{61} (1986), 103-155.

\bibitem[LV83]{Lusztig.Vogan.1983}
George Lusztig and David~A. Vogan, Jr., \emph{Singularities of closures of
  {$K$}-orbits on flag manifolds}, Invent. Math. \textbf{71} (1983), no.~2,
  365--379. 

\bibitem[MWZ99]{MWZ.1999}
Peter Magyar, Jerzy Weyman, and Andrei Zelevinsky, \emph{Multiple flag
  varieties of finite type}, Adv. Math. \textbf{141} (1999), no.~1, 97--118.


\bibitem[MWZ00]{MWZ.2000}
Peter Magyar, Jerzy Weyman, and Andrei Zelevinsky, \emph{Symplectic multiple flag varieties of finite type}, J. Algebra
  \textbf{230} (2000), no.~1, 245--265. 

\bibitem[NO11]{NO.2011}
Kyo Nishiyama and Hiroyuki Ochiai, \emph{Double flag varieties for a symmetric
  pair and finiteness of orbits}, J. Lie Theory \textbf{21} (2011), no.~1,
  79--99. 

\bibitem[Pan93]{Panyushev.1993}
D.~I. Panyushev, \emph{Complexity and rank of double cones and tensor product
  decompositions}, Comment. Math. Helv. \textbf{68} (1993), no.~3, 455--468.
 

\bibitem[Pan99]{Panyushev.1999.MMath}
Dmitri~I. Panyushev, \emph{On the conormal bundle of a {$G$}-stable
  subvariety}, Manuscripta Math. \textbf{99} (1999), no.~2, 185--202.
 

\bibitem[Pet11]{Petukhov.2011}
Alexey Petukhov, \emph{A geometric approach to $ (\mathfrak{g}, \mathfrak{k})
  $-modules of finite type}, arXiv preprint arXiv:1105.5020 (2011).

\bibitem[RS90]{Richardson.Springer.1990}
R.~W. Richardson and T.~A. Springer, \emph{The {B}ruhat order on symmetric
  varieties}, Geom. Dedicata \textbf{35} (1990), no.~1-3, 389--436.


\bibitem[RS93]{Richardson.Springer.1993}
R.~W. Richardson and T.~A. Springer, \emph{Combinatorics and geometry of {$K$}-orbits on the flag
  manifold}, Linear algebraic groups and their representations ({L}os
  {A}ngeles, {CA}, 1992), Contemp. Math., vol. 153, Amer. Math. Soc.,
  Providence, RI, 1993, pp.~109--142. 

\bibitem[RS94]{Richardson.Springer.1994}
R.~W. Richardson and T.~A. Springer, \emph{Complements to: ``{T}he {B}ruhat order on symmetric varieties''
  [{G}eom. {D}edicata {\bf 35} (1990), no. 1-3, 389--436]}, Geom. Dedicata \textbf{49} (1994), no.~2, 231--238.
 

\bibitem[Spr86]{Springer.1986}
T.~A. Springer, \emph{Algebraic groups with involutions}, Proceedings of the
  1984 {V}ancouver conference in algebraic geometry (Providence, RI), CMS Conf.
  Proc., vol.~6, Amer. Math. Soc., 1986, pp.~461--471. 

\bibitem[Spr98]{Springer.1998}
T.~A. Springer, \emph{Linear algebraic groups}, 
Second ed. Progress in Mathematics, vol.~9, Birkh\"auser Boston Inc., Boston, MA, 1998.


\bibitem[Ste68]{Steinberg.1968}
Robert Steinberg, \emph{Endomorphisms of linear algebraic groups}, Memoirs of
  the American Mathematical Society, No. 80, American Mathematical Society,
  Providence, R.I., 1968. 

\bibitem[Ste03]{Stembridge.2003}
John~R. Stembridge, \emph{Multiplicity-free products and restrictions of {W}eyl
  characters}, Represent. Theory \textbf{7} (2003), 404--439 (electronic).

\bibitem[Tan12]{Tanaka.2012}
Yuichiro Tanaka, \emph{Classification of visible actions on flag varieties},
  Proc. Japan Acad. Ser. A Math. Sci. \textbf{88} (2012), no.~6, 91--96.

\bibitem[Tra09]{Travkin.2009}
Roman Travkin, \emph{Mirabolic {R}obinson-{S}chensted-{K}nuth correspondence},
  Selecta Math. (N.S.) \textbf{14} (2009), no.~3-4, 727--758. 

\bibitem[Vin86]{Vinberg.1986}
{\`E}.~B. Vinberg, \emph{Complexity of actions of reductive groups},
  Funktsional. Anal. i Prilozhen. \textbf{20} (1986), no.~1, 1--13, 96.

\bibitem[VK78]{Vinberg.Kimelfeld.1978}
{\`E}.~B. Vinberg and B.~N. Kimel'fel'd, \emph{Homogeneous
  domains on flag manifolds and spherical subsets of semisimple {L}ie groups},
  Funktsional. Anal. i Prilozhen. \textbf{12} (1978), no.~3, 12--19, 96.

\bibitem[Vus74]{Vust.1974}
Thierry Vust, \emph{Op\'eration de groupes r\'eductifs dans un type de c\^ones
  presque homog\`enes}, Bull. Soc. Math. France \textbf{102} (1974), 317--333.

\end{thebibliography}

\end{document}